\documentclass[a4paper,reqno,11pt]{amsart}

\usepackage[foot]{amsaddr}
\usepackage[margin=2.5cm]{geometry}
\usepackage{mathtools}
\usepackage{amsthm}
\usepackage{dsfont}
\usepackage{amssymb}
\usepackage{bm}
\usepackage{enumitem}
\usepackage{tipa}
\usepackage{upgreek}
\usepackage{scalerel}
\usepackage{graphicx}
\usepackage{color}
\usepackage{xcolor}
\usepackage[colorlinks,linkcolor={black},citecolor={blue},urlcolor={black}]{hyperref}
\usepackage{epsfig,color}
\usepackage{tikz}
\usetikzlibrary{decorations.pathreplacing,positioning}
\usepackage{caption}
\usepackage{subcaption}
\usepackage{multirow}
\usepackage[capitalise]{cleveref}
\usepackage{glossaries}
\expandafter\def\csname ver@etex.sty\endcsname{3000/12/31}

\usepackage{autonum}
\usepackage{xpatch} 
\usepackage[square,numbers]{natbib}

\numberwithin{equation}{section}
\newtheorem{theorem}{Theorem}[section]
\newtheorem{lemma}[theorem]{Lemma}
\newtheorem{corollary}[theorem]{Corollary}
\newtheorem{proposition}[theorem]{Proposition}

\xpatchcmd{\proof} 
  {\itshape}
  {\bfseries}
  {}
  {}

\theoremstyle{definition}
\newtheorem{definition}[theorem]{Definition}

\theoremstyle{definition}
\newtheorem{remark}[theorem]{Remark}
\newtheorem{problem}{Problem}
\newtheorem{example}[theorem]{Example}

\let\emptyset\varnothing

\renewcommand{\d}{{\mathrm{\,d}}}
\DeclareMathAlphabet{\mathbbmsl}{U}{bbm}{m}{sl}

\makeatletter
\renewcommand\subsubsection{\@startsection{subsubsection}{3}{\z@}%
                                     {.5\linespacing\@plus.7\linespacing}{-.5em}
                                     {\normalfont\bfseries}}

\title[{Fractional Sobolev paths on Wasserstein spaces and energy minimization}]{
{Fractional Sobolev paths on Wasserstein spaces and their energy-minimizing particle representations}}
\author{Ehsan Abedi$^\dagger$}

\address{$^\dagger$Faculty of Mathematics \\
         Bielefeld University \\
         Postfach 10 01 31 \\
         33501 Bielefeld \\
         Germany.}
\email{ehsan.abedi@math.uni-bielefeld.de}

\thanks{}



\begin{document}

\begin{abstract}
    We study a generalization of the Monge--Kantorovich optimal transport problem. Given a prescribed family of time-dependent probability measures $(\mu_t)$, we aim to find, among all path-continuous stochastic processes whose one-dimensional time marginals coincide with $(\mu_t)$ (if there is any), a process that minimizes a given energy. After discussing a sufficient condition for the energy to ensure the existence of a minimizer, we investigate fractional Sobolev energies. Given a deterministic path $(\mu_t)$ on a $p$-Wasserstein space with fractional Sobolev regularity $W^{\alpha,p}$, where $1/p < \alpha < 1$, we provide conditions under which we prove the existence of a process that minimizes the energy and construct a process that realizes the regularity of $(\mu_t)$. While continuous paths of low regularity on Wasserstein spaces naturally appear in stochastic analysis, they can also arise deterministically as solutions to the continuity equation.
This paper is devoted to the deterministic setting to gain some understanding of the required conditions. The subsequent companion paper (arXiv:2503.10859) focuses on the stochastic setting and applications to SPDEs.
\end{abstract}

\maketitle


\section{Introduction and main results}\label{sec:introduction_d}
\subsection{Problem formulation and main results}\label{subsec:intro_main_det}
Let $(\mathcal{X},d)$ be a complete separable metric space and $\mathcal{B}(\mathcal{X})$ be the $\sigma$-algebra of Borel sets of $\mathcal{X}$. Let $P(\mathcal{X})$ denote the set of Borel probability measures on $\mathcal{X}$ and $P_p(\mathcal{X}) \subset P(\mathcal{X})$ denote the subset of measures with finite $p$-th moment for $p \geq 1$.
Consider a prescribed family of probability measures $(\mu_t) \coloneqq (\mu_t)_{t \in I} \subset P(\mathcal{X})$ indexed by $t$ in a time interval $I \coloneqq [0,T] \subset \mathbb{R}$.
In many applications, it is of interest to have a stochastic process $(X_t)$ on a suitable path space $\Gamma_T \subset \mathcal{X}^{[0,T]} $ whose one-dimensional time marginals coincide with $(\mu_t)$.
We know that this requirement does not uniquely determine the process in general.
In this paper, we are interested in processes whose path laws, denoted by $\pi \in P(\Gamma_T)$ and referred to as lifts,
carry further intrinsic information about $(\mu_t)$, as a result of a minimization problem analogous to the transport problem by Kantorovich in 1942 \cite{Kantorovich2006}, following Monge \cite{Monge1781}.
Typical choices for $\Gamma_T$ are the space of continuous paths and the space of c\`adl\`ag paths.
Throughout this paper, we consider $\Gamma_T \coloneqq C([0,T];\mathcal{X})$ endowed with the $\sigma$-algebra $\mathcal{C}$ generated by the corresponding evaluation maps. We also equip $\Gamma_T$ with the supremum distance and recall that the Borel $\sigma$-algebra generated by the corresponding open sets coincides with $\mathcal{C}$. 
We formulate the problem mentioned above as a variational problem on the space of path measures on $\Gamma_T$:

\vspace{6pt}
    
\begin{minipage}{0.97\textwidth}
    \begin{problem}\label{prob:variational_problem}
        Given $(\mu_t)_{t \in I} \subset P(\mathcal{X})$ with $I \coloneqq [0,T] \subset \mathbb{R}$ and a measurable functional $\Psi: \Gamma_T \to [0,+\infty]$, consider the variational problem
        \begin{equation}\label{eq:variational_problem}
            \inf_{\pi \in \mathrm{Lift}(\mu_t)} \int_{\Gamma_T} \Psi (\gamma) \d \pi (\gamma),
        \end{equation}
        where the infimum is taken over the set of \emph{lifts} of $(\mu_t)$ defined by
        \begin{equation}\label{eq:set_lift_intro}
            \mathrm{Lift}(\mu_t) \coloneqq \Big\{\pi \in P(\Gamma_T) : \quad (e_t)_{\#} \pi  = \mu_t \text{ for all } t \in I \Big\},
        \end{equation}
        where $e_t: \Gamma_T \to \mathcal{X}$ is the evaluation map defined by $e_t(\gamma) \coloneqq \gamma_t$ for $\gamma \in \Gamma_T$. 
    \end{problem}
\end{minipage}

\vspace{6pt}

This extends Monge--Kantorovich's classical problem, with two fixed marginals $\mu,\nu \in P(\mathcal{X})$, to a time-dependent setting with infinitely many fixed marginals $(\mu_t)_{t \in I} \subset P(\mathcal{X})$.
In contrast to the set of couplings $\mathrm{Cpl}(\mu,\nu) \subset P(\mathcal{X}^2)$, which is always non-empty, the set of liftings $\mathrm{Lift}(\mu_t)$ $ \subset P(\Gamma_T) \subset P(\mathcal{X}^I)$ can be empty. In this case, the infimum \eqref{eq:variational_problem} is $+\infty$ by the usual convention.

One can always obtain a process from a lift $\pi$ by looking at its associated canonical process.
From a physical perspective, this provides a particle representation, and the objective function in \cref{prob:variational_problem} can be interpreted as the total cost of the particles' motion. We thus call $\int \Psi \d \pi$ the \emph{energy} of $\pi$ with respect to the \emph{energy functional} $\Psi$. 
The probabilistic formulation of the problem above is given in \cref{rmk:probabilistic_formulation}.

When $\mu,\nu \in P_p (\mathcal{X})$ for some $p \geq 1$ and the cost function in Monge--Kantorovich’s problem is taken as $d(x,y)^p$, we know an optimal coupling, denoted by $\Upsilon \in \mathrm{OptCpl} (\mu,\nu)$, exists and defines the $p$-(Kantorovitch--Rubinstein--)Wasserstein distance: $W^p_p(\mu,\nu) = \int_{\mathcal{X}^2} d(x,y)^p \d \Upsilon (x,y) $.

Now let $(\mu_t)_{t \in I} \subset P_p(\mathcal{X})$ be a path on the Wasserstein space $(P_p(\mathcal{X}),W_p)$ with a certain path regularity $|\mu|< +\infty$. 
Here, $|\cdot| \in [0,+\infty]$ represents a functional that captures a certain regularity of curves in metric spaces (e.g. absolute continuity or H\"{o}lder regularity).
In view of the two-marginal case, one can similarly ask whether there exists a lift $\pi$ that realizes this regularity in the sense that
 \begin{equation}\label{eq:optimality_intro}
    |\mu|^p =  {\scaleobj{0.74}{\int}}_{\hspace{-2pt}\Gamma_T} |\gamma|^p \d \pi (\gamma).
\end{equation} 
Similarly, one expects that any alternative lift produces higher energy.
This is indeed the case for the regularities previously studied and those in this paper.
This observation has motivated introducing \cref{prob:variational_problem}, where the general functional $\Psi$ is chosen here as $|\cdot|^p$.
We call a lift \emph{minimizing} or \emph{optimal} if it achieves the minimum possible energy. 
We further call it \emph{realizing} if it attains equality \eqref{eq:optimality_intro} (we will observe that in some cases, a minimizing lift exists but a realizing one doesn't). Lifts that realize the regularity of Wasserstein curves have been previously constructed under different assumptions on regularity, some of which are highlighted here:
\begin{itemize}[leftmargin=20pt]
    \item[\textbf{-}] {Constant-speed geodesics in $p$-Wasserstein spaces with $p=2$}: Lott--Villani \cite{LottVillani2009,Villani2009} and Sturm \cite{Sturm2006I} constructed a lift $\pi$ on $C(I;\mathcal{X})$ that realizes $W_p^p(\mu_s,\mu_t) = \int_{C} d(\gamma_s,\gamma_t)^p \d \pi $ for all $s,t \in I$. This lift has well-known applications in the geometry of metric measure spaces.
    \item[\textbf{-}] {$p$-absolutely continuous curves in $p$-Wasserstein spaces with $p>1$ on separable Hilbert spaces}: Ambrosio--Gigli--Savar\'e \cite{AGS2008GFs} characterize these curves via the continuity equation and obtained a vector field $v_t$ and a lift $\pi$ on $C(I;\mathcal{X})$ that realizes the metric speed i.e. $|\dot\mu_t|^p = \int_{\mathcal{X}} |v_t|^p \d \mu_t = \int_{C} |\dot{\gamma}_t|^p \d \pi $ for a.e. $t \in I$. 
    \item[\textbf{-}] {$p$-absolutely continuous curves in $p$-Wasserstein spaces with $p>1$ on complete separable metric spaces}:
    Lisini \cite{Lisini2007,Lisini2016} significantly extended the above results and constructed a lift $\pi$ on $ C(I;\mathcal{X})$ that realizes the metric speed i.e. $|\dot{\mu}_t|^p = \int_{C} |\dot{\gamma}_t|^p \d \pi$ for a.e. $t \in I$.
    \item[\textbf{-}] {C\`adl\`ag curves of bounded variation in $p$-Wasserstein spaces with $p=1$ on complete separable metric spaces}: Recently, \cite{AbediLiSchultz2024} constructed a lift $\pi$ on $D(I;\mathcal{X})$ that realizes the total variation measure i.e. $|D\mu|=\int_{D}|D\gamma|\d\pi$ as measures. The results are applied to the current equation. 
\end{itemize}
The reason for adopting a larger path space in the last case, namely the space of c\`adl\`ag curves $D(I;\mathcal{X})$, is that even $1$-absolutely continuous curves in 1-Wasserstein spaces, unlike the case $p>1$, cannot generally be lifted to measures on continuous paths. This can also be noticed through the Kolmogorov–\v Centsov continuity criterion, where $p=1$ is excluded. In this work, we consider $p> 1$ and study continuous Wasserstein curves with lower regularity than above. 

Our study of low-regularity paths on Wasserstein spaces is motivated by measure-valued solutions to conservative stochastic PDEs, such as stochastic Fokker--Planck--Kolmogorov equations. Low-regularity paths can also arise deterministically as solutions to the continuity equation, though less trivially. In this paper, we focus on the deterministic case, and study the random setting in the subsequent companion paper \cite{Abedi2025processes}.
Since, in low regularity, it is less straightforward which norm can be minimized and potentially realized, our first step is to identify a sufficient condition on the energy to guarantee the existence of a minimizer. Below is \cref{prop:existence_of_minimizer}---a simple observation using the direct method in the calculus of variations.

\begin{proposition}[Existence of a minimizer]\label{prop:existence_of_minimizer_intro}
    Let $(\mathcal{X}, d)$ be a complete separable metric space, and $I \coloneqq [0,T] \subset \mathbb{R}$. Let $\Psi: C(I;\mathcal{X}) \to [0,+\infty]$ be {\normalfont \textcircled{\small{1}}} a lower semi-continuous map {\normalfont \textcircled{\small{2}}} whose sublevels are relatively compact in $C(I;\mathcal{X})$. Assume that the infimum  \eqref{eq:variational_problem} is finite. Then there exists a minimizer $\pi \in P(C(I;\mathcal{X})) $ to \cref{prob:variational_problem}. 
\end{proposition}

Now, we consider energy functionals $\Psi: C(I;\mathcal{X}) \to [0,+\infty]$ of the form
\begin{equation}\label{eq:energy_functional_intro}
    \Psi (\gamma) \coloneqq d(\gamma_0,\bar{x}) + | \gamma | ,
\end{equation}
where $\bar{x} \in \mathcal{X}$ is an arbitrary point and  $| \cdot | : C(I;\mathcal{X}) \to [0,+\infty]$ is a (semi-)norm. 
When $(\mathcal{X},d)$ has enough structure so that closed bounded sets are compact, Arzel\`a-Ascoli theorem can verify which norms $|\gamma|$ make the sublevels of $\Psi$ relatively compact. 
This is summarized for some commonly used norms in low-regularity settings in \textbf{\cref{table:norms}} at the end of this section.
We note that the modulus of continuity $w_\delta (\gamma)$, $p$-variation $| \gamma |_{p\textrm{-}\mathrm{var}}$, and its infinitesimal characterization $ | \gamma |_{p\textrm{-}\mathrm{var}\textrm{-}\mathrm{limsup}}$ fail these conditions. 
In particular, quadratic variation $[\gamma]$ in the sense of stochastic analysis also fails; see \cref{rmk:quadratic_variation}.
By contrast, H\"{o}lder regularity $| \gamma |_{\upgamma\textrm{-}\mathrm{H\ddot{o}l}} $, fractional Sobolev regularity $| \gamma |_{W^{\alpha,p}}$, certain Besov regularity $| \gamma |_{b^{\alpha,p}}$, and Sobolev regularity $| \gamma |_{W^{1,p}}$ satisfy the required conditions. Definitions and properties of these norms are given in \cref{subsec:cont_paths,subsec:Holder_paths,subsec:var_paths,subsec:Walphap,subsec:W1p}.
The function space $W^{1,p} (I;\mathcal{X})$ used in this paper coincides with the set of $p$-absolutely continuous curves $AC^p (I;\mathcal{X})$; thus, the last item in the table is already addressed by Lisini's results \cite{Lisini2007}.

Following this observation, we consider paths in the \emph{fractional Sobolev} space $W^{\alpha,p}$ or certain \emph{Besov} space $b^{\alpha,p}$. These spaces also contain H\"{o}lder spaces for specific exponents.
Given a metric space $(\mathcal{X},d)$, an integrability parameter
$1\leq p<\infty$, and a regularity parameter $0<\alpha < 1$, these function spaces are defined as:

\begin{itemize}
    \item $W^{\alpha,p}([0,1];\mathcal{X})$ is the space of measurable functions $X:[0,1] \to \mathcal{X}$ such that
    \begin{equation}\label{eq:Walphap_integral_intro}
                | X |_{W^{\alpha,p}} \coloneqq \left( \iint_{[0,1]^2} \frac{d(X_s,X_t)^p}{|t-s|^{1+\alpha p}} \d s \d t \right)^{1/p} < + \infty.
            \end{equation}
    \item $b^{\alpha,p}([0,1];\mathcal{X})$ is the space of continuous functions $X:[0,1] \to \mathcal{X}$ such that
    \begin{equation}\label{eq:Walphap_sum_intro}
                | X |_{b^{\alpha,p}} \coloneqq \left( \sum_{m=0}^{\infty} 2^{m(\alpha p -1)} \sum_{k=0}^{2^m-1} d \big(X_{t_{k}^{(m)}}, X_{t_{k+1}^{(m)}}\big)^p \right)^{1/p} < + \infty,
            \end{equation}
             where $t_k^{(m)} \coloneqq \frac{k}{2^m}$.
\end{itemize}
These norms have found several applications in the theory of rough paths.
It is proven by Liu--Pr\"{o}mel--Teichmann \cite{LiuPromelTeichmann2020} that under the condition  $1<p<\infty$ and $\frac{1}{p}<\alpha < 1$, the integral
\eqref{eq:Walphap_integral_intro} and the sum \eqref{eq:Walphap_sum_intro} provide us with equivalent norms on the space of continuous paths, and thus we have $W^{\alpha,p}([0,1];\mathcal{X}) = b^{\alpha,p}([0,1];\mathcal{X})$. 

As the next step, we start from a path measure $\pi$ of finite $W^{\alpha,p}$-energy and show that its curve of one-dimensional time marginals $t \mapsto \mu_t \coloneqq (e_t)_{\#} \pi$ inherits the same kind of regularity, and, moreover, its regularity is bounded from above by the energy of $\pi$.
In fact, this transfer of regularity from the path measure to its curve of marginals holds for all lower-semi continuous norms listed in \cref{table:norms}, as shown in \cref{subsec:from_pi_to_mu}, but is not reported here.
Below is \cref{thm:lift_to_mu_Walphap}. 

\begin{theorem}\label{thm:lift_to_mu_Walphap_intro}
    Let $(\mathcal{X},d)$ be a complete separable metric space.
    Let $\pi \in P(C([0,T];\mathcal{X})) $ satisfy
     \begin{equation}\label{eq:lift_to_mu_Walphap_integrability_intro}
        \int_{\Gamma_T} \Big( d(\gamma_0,\bar{x})^p + | \gamma |_{W^{\alpha,p}}^p  \Big) \d \pi(\gamma) < + \infty
    \end{equation}
    for some $1<p<\infty$ and $ \frac{1}{p} < \alpha < 1$ and $\bar{x}\in \mathcal{X}$.
    Then, $t \mapsto \mu_t \coloneqq {(e_t)}_{\#} \pi $ is in $ W^{\alpha, p} ([0,T];P_p(\mathcal{X}))$, and moreover,
    \begin{equation}
        |\mu|_{W^{\alpha,p}}^p \leq \int_{\Gamma_T} |\gamma|_{W^{\alpha,p}}^p \d \pi(\gamma).
    \end{equation}
    The same statement holds for $|\cdot|_{b^{\alpha,p}}$.
\end{theorem}

By combining \cref{prop:existence_of_minimizer_intro}, \cref{thm:lift_to_mu_Walphap_intro}, and the observation on $|\cdot|_{W^{\alpha,p}}$ in \cref{table:norms}, we immediately draw the following conclusion as the first main result: 

\begin{theorem}[Existence of a minimizing lift]\label{thm:existence_of_minimizer_Walphap_intro}
    Let $(\mathcal{X},d)$ be a complete separable metric space in which closed bounded sets are compact, and $I \coloneqq [0,T] \subset \mathbb{R}$.
    Let $(\mu_t)_{t \in I} \subset P(\mathcal{X})$ be such that $\mu_0 \in P_p(\mathcal{X})$ and it has a lift with finite $W^{\alpha,p}$-energy with $1<p<\infty$ and $\frac{1}{p}< \alpha <  1$, i.e., \eqref{eq:lift_to_mu_Walphap_integrability_intro} holds for a lift. Then, there exists a minimizer $\pi \in P(C(I;\mathcal{X}) ) $ to \cref{prob:variational_problem} for the energy $\Psi(\gamma)  = |\gamma|^p_{W^{\alpha,p}}$. In particular, 
    \begin{enumerate}[label=(\roman*), font=\normalfont]
        \item $\pi$ is concentrated on $W^{\alpha,p}(I;\mathcal{X}) \subset C(I;\mathcal{X})$; 
        \item $(e_t)_\#\pi=\mu_t$ for all $t\in I$;
        \item $\pi$ satisfies
        \begin{align}\label{eq:minimizer_pi_Walphap}
        |\mu|_{W^{\alpha,p}}^p \leq \int_{\Gamma_T} |\gamma|_{W^{\alpha,p}}^p \d \pi (\gamma) < + \infty.
        \end{align}
    \end{enumerate}
     The same statement holds for $|\cdot|_{b^{\alpha,p}}$.
\end{theorem}

A natural question arises as to when a minimizer $\pi$ attains equality in \eqref{eq:minimizer_pi_Walphap}.
Although $W^{\alpha,p}$ is a larger function space than $W^{1,p}$, the next proposition shows that if the $W^{\alpha, p}$-regularity of a curve $(\mu_t)$ is realized by a lift on continuous paths $C(I;\mathcal{X})$, it imposes a global and rather restrictive condition on the collection $(\mu_t)_{t \in I} \subset P_p(\mathcal{X})$.
This is due to the non-local structure of the norms $|\cdot|_{W^{\alpha,p}}$ and $|\cdot|_{b^{\alpha,p}}$.
We refer to the imposed condition as \emph{compatibility}, a term adopted from \cite{Boissard2015,PanaretosZemel2020}.
Its definition is given after the statement and is discussed in detail in \cref{subsubsec:compatibility}.
The observation below is \cref{prop:equality_implies_compatibility}. 

\begin{proposition}\label{prop:equality_implies_compatibility_intro}
Let $(\mathcal{X},d)$ be a complete separable metric space, and $I \coloneqq [0,T] \subset \mathbb{R}$. Let $(\mu_t) \in  W^{\alpha,p}(I;P_p(\mathcal{X}))$ with $1<p<\infty$ and $\frac{1}{p}< \alpha <  1$. Assume that $(\mu_t)$ has a lift $\pi \in P(C(I;\mathcal{X}) ) $ whose $W^{\alpha,p}$-energy satisfies the equality
    \begin{align}\label{eq:equality_implies_compatibility_intro}
        |\mu|_{W^{\alpha,p}}^p = \int_{\Gamma_T} |\gamma|_{W^{\alpha,p}}^p \d \pi (\gamma),
    \end{align}
    then $(\mu_t)_{t\in I}$ is compatible in $P_p(\mathcal{X})$.
\end{proposition}

\begin{definition}[Compatibility of measures in $P_p(\mathcal{X})$]\label{def:compatibility_intro}
    We say a collection of measures $\mathcal{M} \subset P_p(\mathcal{X})$ is compatible if, for every finite subcollection of $\mathcal{M}$, there exists a multi-coupling such that all of its two-dimensional marginals are optimal.
\end{definition}
Some examples of compatibility ($\bullet$) and non-compatibility ($\circ$) include the following:
\begin{itemize}
    \item[$\bullet$] Any collection of measures lying on a Wasserstein geodesics is compatible, as discussed in \cref{rmk:geodesics_are_compatible}. A compatible collection, however, need not lie on a Wasserstein geodesic.
    \item[$\bullet$] All probability measures on $\mathbb{R}$ with finite $p$-moment are compatible.
    \item[$\bullet$] Gaussian measures on $\mathbb{R}^{\mathrm{d}}$  can form a compatible collection under suitable conditions on their covariance matrices. See \cite[page 49]{PanaretosZemel2020}. 
    \item[$\bullet$] For more nontrivial compatible examples, see \cite[Proposition 4.1]{Boissard2015} and \cite[Section 2.3.2]{PanaretosZemel2020}.
    \item[$\circ$] The compatibility can easily fail---for instance, under rotation. See \cref{exp:compatibility_vs_noncompatibility}.
\end{itemize} 

\cref{prop:equality_implies_compatibility_intro} tells us that it is impossible to have a lift on continuous paths satisfying the equality \eqref{eq:equality_implies_compatibility_intro} without the compatibility of $(\mu_t)$. Therefore, our next step is to assume the compatibility property and construct a lift that realizes the regularity of $(\mu_t)$.
In this regard, we make two remarks:
First, compatibility in the sense above provides a multi-coupling only for finite sub-collections.
Second, for an infinite compatible collection of measures $(\mu_t)_{t \in I}$, the multi-couplings of finite subcollections arising in the definition of compatibility need not be consistent because optimal couplings may not be unique.
Thus, applying Kolmogorov’s extension theorem to obtain a unique lift on the infinite product space $\mathcal{X}^I$ is not immediate, even in the case of Wasserstein geodesics.
 
To achieve a lift realizing  $W^{\alpha,p}$-regularity of  $(\mu_t)$, we adapt Lisini's construction  \cite{Lisini2007}, which glues consecutive optimal couplings (\textbf{{\normalfont \textcircled{\small{A}}}} below), into our construction (\textbf{{\normalfont \textcircled{\small{B}}}} below), which takes the multi-coupling provided by the compatibility assumption. 
The key difference is that, here, two-dimensional marginals on non-consecutive time points matter. This distinction is also illustrated in \cref{fig:construction}. 
Here, we work in a geodesic space $(\mathcal{X},d)$ and, for simplicity, take $I = [0,1]$. 
In what follows,  $\text{Pr}^{i,j}: \mathcal{X}^N \to {\mathcal{X}^2}$ is the projection map to $(i,j)$-th component for $i,j \in \{1,\cdots, N \in \mathbb{N} \}$. 

\vspace{6pt}

\noindent
\textbf{Construction.}
Let $(\mu_t)_{t\in [0,1]} \subset P_p(\mathcal{X})$ be a collection on a geodesic space $\mathcal{X}$. For each $n \in \mathbb{N}_0$,
    \begin{enumerate}
        \item[1.] Divide the time interval $[0,1]$ into the dyadic dissection $t_i^{(n)} \coloneqq \frac{i}{2^n}, i \in \{0,1,\cdots 2^n\}$. 
        \item[2.]
        \begin{itemize}
            \item[\textbf{{\normalfont \textcircled{\small{A}}}}]\label{itm:ConstructionA}
            When $(\mu_t) \subset P_p(\mathcal{X})$ is an \emph{arbitrary} collection:
            \\
            Let $\Upsilon_{n} \in P(\mathcal{X}^{2^n+1})$ be a multi-coupling such that
            \begin{equation}
                    (\text{Pr}^{i,i+1})_{\#} \Upsilon_{n} \in \text{OptCpl}\big(\mu_{t^{(n)}_{i}},\mu_{t^{(n)}_{i+1}}\big) 
            \end{equation}
            for all $i \in \left\{ 0,1, \cdots, 2^n-1 \right\}$.
            \\
            The existence of such a measure follows from the \emph{gluing lemma} for optimal couplings.
        \item[\textbf{{\normalfont \textcircled{\small{B}}}}]\label{itm:ConstructionB}
        When $(\mu_t)\subset P_p(\mathcal{X})$ is a \emph{compatible} collection:
        \\
        Let $\Upsilon_{n} \in P(\mathcal{X}^{2^n+1})$ be a multi-coupling such that
        \begin{equation}\label{eq:compatibility_dyadic}
                (\text{Pr}^{i,i+\frac{2^n}{2^m}})_{\#} \Upsilon_{n} \in \text{OptCpl}\big(\mu_{t^{(n)}_{i}},\mu_{t^{(n)}_{i+\frac{2^n}{2^m}}}\big) 
        \end{equation}
        for all $i \in \big\{k \frac{2^n}{2^m} \big| k \in \{ 0,1,\cdots, 2^m-1\} \big\} $ and $ m \in \{0,1,\cdots n\}$.
        \\
        The existence of such a measure follows from the very \emph{assumption of compatibility}. 
        \end{itemize}
        
        \item[3.] Construct the path measure $\pi_n \coloneqq (\ell)_{\#} \Upsilon_n \in P(\Gamma_1)$, where $\ell : \mathcal{X}^{2^n+1} \to \Gamma_1$ is a $(\mathcal{B}(\mathcal{X}^{2^n+1})_{\Upsilon_n},\mathcal{C})$-measurable geodesic selection and interpolation map connecting the points with constant-speed geodesics. $\mathcal{B}(\mathcal{X}^{2^n+1})_{\Upsilon_n}$ denotes $\Upsilon_n$-completion of $\mathcal{B}(\mathcal{X}^{2^n+1})$.
        
        \item[4.] Take the limit $n \to \infty$ and verify the narrow convergence of the sequence  $\{\pi_n\}_{n \in \mathbb{N}}$. 
    \end{enumerate}

\begin{figure}
    \centering
    \begin{tikzpicture}[scale=1.25]
    \draw (0,0) -- (4,0);
    \draw (5,0) -- (9,0);
    \filldraw [black] (0,0) circle (1pt);
    \draw (0,-0.1) node[anchor=north] {\scriptsize $\mu_{t_0^{(2)}}$};
    \filldraw [black] (1,0) circle (1pt);
    \draw (1,-0.1) node[anchor=north] {\scriptsize $\mu_{t_1^{(2)}}$};
    \filldraw [black] (2,0) circle (1pt);
    \draw (2,-0.1) node[anchor=north] {\scriptsize $\mu_{t_2^{(2)}}$};
    \filldraw [black] (3,0) circle (1pt);
    \draw (3,-0.1) node[anchor=north] {\scriptsize $\mu_{t_3^{(2)}}$};
    \filldraw [black] (4,0) circle (1pt);
    \draw (4,-0.1) node[anchor=north] {\scriptsize $\mu_{t_4^{(2)}}$};
    \filldraw [black] (5,0) circle (1pt);
    \draw (5,-0.1) node[anchor=north] {\scriptsize $\mu_{t_0^{(2)}}$};
    \filldraw [black] (6,0) circle (1pt);
    \draw (6,-0.1) node[anchor=north] {\scriptsize $\mu_{t_1^{(2)}}$};
    \filldraw [black] (7,0) circle (1pt);
    \draw (7,-0.1) node[anchor=north] {\scriptsize $\mu_{t_2^{(2)}}$};
    \filldraw [black] (8,0) circle (1pt);
    \draw (8,-0.1) node[anchor=north] {\scriptsize $\mu_{t_3^{(2)}}$};
    \filldraw [black] (9,0) circle (1pt);
    \draw (9,-0.1) node[anchor=north] {\scriptsize $\mu_{t_4^{(2)}}$};
    \draw [black][-] (0,0.1) .. controls (0.2,0.5) and (0.8,0.5) .. (1,0.1);
    \draw [black][-] (1,0.1) .. controls (1.2,0.5) and (1.8,0.5) .. (2,0.1);
    \draw [black][-] (2,0.1) .. controls (2.2,0.5) and (2.8,0.5) .. (3,0.1);
    \draw [black][-] (3,0.1) .. controls (3.2,0.5) and (3.8,0.5) .. (4,0.1);
    \draw [black][-] (5,0.1) .. controls (5.2,0.5) and (5.8,0.5) .. (6,0.1);
    \draw [black][-] (6,0.1) .. controls (6.2,0.5) and (6.8,0.5) .. (7,0.1);
    \draw [black][-] (7,0.1) .. controls (7.2,0.5) and (7.8,0.5) .. (8,0.1);
    \draw [black][-] (8,0.1) .. controls (8.2,0.5) and (8.8,0.5) .. (9,0.1);
    \draw [black][-] (5,0.2) .. controls (5.4,1) and (6.6,1) .. (7,0.2);
    \draw [black][-] (7,0.2) .. controls (7.4,1) and (8.6,1) .. (9,0.2);
    \draw [black][-] (5,0.3) .. controls (5.8,1.8) and (8.2,1.8) .. (9,0.3);
    \draw (0.5,0.7) node[anchor=north] {{\color{darkgray}\scriptsize Opt}};
    \draw (1.5,0.7) node[anchor=north] {{\color{darkgray}\scriptsize Opt}};
    \draw (2.5,0.7) node[anchor=north] {{\color{darkgray}\scriptsize Opt}};
    \draw (3.5,0.7) node[anchor=north] {{\color{darkgray}\scriptsize Opt}};
    \draw (5.5,0.7) node[anchor=north] {{\color{darkgray}\scriptsize Opt}};
    \draw (6.5,0.7) node[anchor=north] {{\color{darkgray}\scriptsize Opt}};
    \draw (7.5,0.7) node[anchor=north] {{\color{darkgray}\scriptsize Opt}};
    \draw (8.5,0.7) node[anchor=north] {{\color{darkgray}\scriptsize Opt}};
    \draw (6,1.1) node[anchor=north] {{\color{darkgray}\scriptsize Opt}};
    \draw (8,1.1) node[anchor=north] {{\color{darkgray}\scriptsize Opt}};
    \draw (7,1.72) node[anchor=north] {{\color{darkgray}\scriptsize Opt}};
    \draw (2,-0.5) node[anchor=north] {\scriptsize Construction \hyperref[itm:ConstructionB]{{\normalfont \textcircled{\tiny{A}}}} for arbitrary collections};
    \draw (7,-0.5) node[anchor=north] {\scriptsize Construction \hyperref[itm:ConstructionB]{{\normalfont \textcircled{\tiny{B}}}} for compatible collections};
    \end{tikzpicture}
    \captionsetup{font=scriptsize}
    \caption{An illustration of two different constructions, showing which two-dimensional marginals of $\Upsilon_n$ are optimal. Here, $n=2$ and the time interval $[0,1]$ is divided into $2^2=4$ equal pieces.}
    \label{fig:construction}
\end{figure}
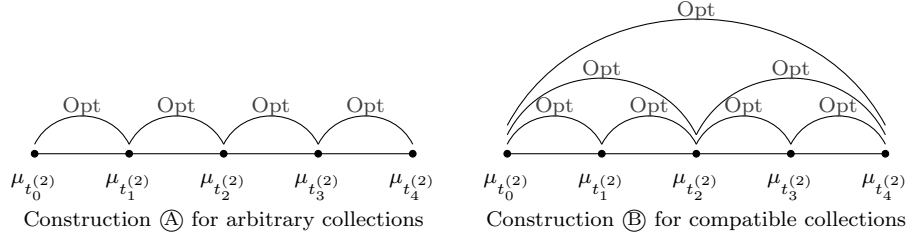

Now, relying only on the properties of the Wasserstein curve (its regularity and compatibility), we construct a lift that realizes its regularity, thereby removing the assumption of the existence of a lift with finite energy in \cref{thm:existence_of_minimizer_Walphap_intro}. 
As the second main result, we present \cref{thm:optimal_lift_mu_Walphap_compatible}:

\begin{theorem}[Construction of a realizing lift]\label{thm:optimal_lift_mu_Walphap_compatible_intro}
    Let $(\mathcal{X}, d)$ be a complete, separable, and locally compact length metric space (e.g. $\mathbb{R}^\mathrm{d}$), and $I \coloneqq [0,T] \subset \mathbb{R}$.
    Let $(\mu_t) \in  W^{\alpha,p}(I;P_p(\mathcal{X}))$ with $1<p<\infty$ and $\frac{1}{p}< \alpha <  1$.
    Assume that $(\mu_t)_{t \in I}$ is compatible in $P_p(\mathcal{X})$.
    Then, construction \hyperref[itm:ConstructionB]{{\normalfont \textcircled{\small{B}}}} converges narrowly (up to a subsequence) to a probability measure $\pi \in P(C(I;\mathcal{X}))$ satisfying
    \begin{enumerate}[label=(\roman*), font=\normalfont]
        \item $\pi$ is concentrated on $W^{\alpha,p}(I;\mathcal{X}) \subset C(I;\mathcal{X})$; 
        \item $(e_t)_\#\pi=\mu_t$ for all $t\in I$;
        \item $(e_s,e_t)_\# \pi \in \mathrm{OptCpl}(\mu_s, \mu_t)$ for all $s,t \in I$; and in particular,
        \begin{align}\label{eq:optimality_pi_1_intro}
        |\mu|_{W^{\alpha,p}}^p = \int_{\Gamma_T} |\gamma|_{W^{\alpha,p}}^p \d \pi (\gamma) .
        \end{align}
    \end{enumerate}
    The same statement holds for $|\cdot|_{b^{\alpha,p}}$.
    \end{theorem}

    \begin{remark}[A weaker compatibility condition]\label{rmk:ompatibility_dyadic}
        What we need for the result above is actually the compatibility of measures only on \emph{dyadic} time points in the way specified in \eqref{eq:compatibility_dyadic}, which is, a priori, weaker than compatibility for \emph{all} finite time points (\cref{def:compatibility_intro}).
    \end{remark}

    An immediate consequence of the theorem above and the following well-known embeddings for $\frac{1}{p}< \alpha < \upgamma \leq  1$, which are based on \emph{Garsia--Rodemich--Rumsey inequality} \cite{Garsia_Rodemich_Rumsey1970,FrizVictoir2010book,FrizVictoir2006}, \vspace{-2pt}
    \begin{equation}
    C^{\upgamma\textrm{-}\mathrm{H\ddot{o}l}}  \overset{\textrm{trivial}}{ \subset} W^{\alpha, p} \,  \, \subset C^{\alpha - \frac{1}{p}\textrm{-}\mathrm{H\ddot{o}l}}, \qquad
        W^{\alpha, p} \,  \, \subset C^{\frac{1}{\alpha}\textrm{-}\mathrm{var}},
    \end{equation}
    is the construction of a lift for $\upgamma$-H\"{o}lder compatible $p$-Wasserstein paths. 
    Below is \cref{crl:optimal_lift_mu_Holder_compatible}. 

    \begin{corollary}\label{crl:optimal_lift_mu_Holder_compatible_intro}
        Let $(\mathcal{X}, d)$ be a complete, separable, and locally compact length metric space, and $I \coloneqq [0,T] \subset \mathbb{R}$.
        Let $(\mu_t) \in C^{\upgamma\textrm{-}\mathrm{H\ddot{o}l}} (I;P_p(\mathcal{X}))$ for some $1<p<\infty$ and $\frac{1}{p}<\upgamma \leq 1$.
        Assume that $(\mu_t)_{t \in I}$ is  compatible in $P_p(\mathcal{X})$.
        Then, construction {\normalfont \textcircled{\small{B}}} converges narrowly (up to a subsequence) to a probability measure $\pi \in P(C(I;\mathcal{X}))$ satisfying
        \begin{enumerate}[label=(\roman*), font=\normalfont]
            \item $\pi$ is concentrated on $W^{\alpha,p}(I;\mathcal{X}) \subset C^{(\alpha - \frac{1}{p})\textrm{-} \mathrm{H\ddot{o}l}}(I;\mathcal{X})$ for any $\alpha \in (\frac{1}{p},\upgamma)$;
            \item $(e_t)_\#\pi=\mu_t$ for all $t\in I$;
            \item $(e_s,e_t)_\# \pi \in \mathrm{OptCpl}(\mu_s, \mu_t)$ for all $s,t \in I$; and for any $\alpha \in (\frac{1}{p},\upgamma)$, we have  \eqref{eq:optimality_pi_1_intro} and 
            \begin{align}\label{eq:optimal_lift_mu_Holder_compatible}
               | \mu |_{\upgamma\textrm{-}\mathrm{H\ddot{o}l}}^p \geq c \int_{\Gamma_T} | \gamma |^p_{\alpha - \frac{1}{p}\textrm{-}\mathrm{H\ddot{o}l}} \d \pi (\gamma) \geq c | \mu |_{\alpha-\frac{1}{p}\textrm{-}\mathrm{H\ddot{o}l}}^p,
            \end{align}
            where $c = c(\upgamma,\alpha,p,T)$ is an explicit positive constant.
        \end{enumerate}
    \end{corollary}

    \begin{remark}
        In addition to the estimate above, we have for any $\alpha \in (\frac{1}{p},\upgamma)$: 
        \begin{equation}
             | \mu |_{\upgamma\textrm{-}\mathrm{H\ddot{o}l}}^p \geq c \, \int_{\Gamma_T} | \gamma |^p_{\frac{1}{\alpha} \textrm{-}\mathrm{var}} \d \pi (\gamma)  \geq c \, | \mu |_{p\textrm{-}\mathrm{var}}^p,
        \end{equation}
        where $c = c(\upgamma,\alpha,p,T)$ is another explicit positive constant.
    \end{remark}

    \cref{exp:p_Holder_p_Wasserstein}  shows that $1/p$-H\"{o}lder paths on $p$-Wasserstein space do not generally have lifts on continuous paths. This shows the sharpness of the assumption $ 1/p < {\upgamma}$ in the result above.
    
    Another natural question is whether it is possible to remove the assumptions of the existence of a lift with finite energy and compatibility from  \cref{thm:existence_of_minimizer_Walphap_intro} and \cref{thm:optimal_lift_mu_Walphap_compatible_intro}, respectively, while still obtaining a minimizing lift.
    We demonstrate that the answer to this question can be negative. In other words, if compatibility is dropped, the infimum \eqref{eq:variational_problem} for the energy $\Psi(\gamma) = |\gamma|^p_{W^{\alpha,p}}$ can be $+\infty$. More specifically, in \cref{exp:non_compatibile_infinite_energy}, we show:

    \begin{proposition}\label{prop:main_counterexample}
        There exists a non-compatible curve $(\mu_t) \in  W^{\alpha,p}(I;P_p(\mathcal{X}))$ for some $1<p<\infty$ and $\frac{1}{p}< \alpha< 1$, whose only existing lift $\pi$ on continuous paths has infinite $|\cdot|^p_{W^{\alpha,p}}$-energy.
    \end{proposition}
    
    This shows that, to remove the aforementioned assumptions, we need alternative conditions, at least to ensure the finiteness of the energy, which we need for proving the convergence of the constructions.
    A case where Construction \hyperref[itm:ConstructionA]{{\normalfont \textcircled{\small{A}}}} is used for Wasserstein paths of low regularity appears in an independent work by Taghvaei–Mehta \cite{Taghvaei2016}, where the authors formally applied it and obtained a lift for solutions $(\mu_t)$ of a stochastic PDE in linear-Gaussian filtering. Since Gaussians form a compatible collection, this was essentially Construction \hyperref[itm:ConstructionB]{{\normalfont \textcircled{\small{B}}}}, making convergence expected and highlighting why generalizing to nonlinear filtering---where compatibility can easily fail---is challenging. We leave further discussion in the stochastic setting to subsequent work.

    As a final result, we give a dynamic formulation of the Wasserstein distance using Besov energy. We first show in \cref{lemma:bound_distance_balphap,lemma:characterization_geodesics_balphap} that on a metric space $(\mathcal{X},d)$ and for any $1< p < \infty$ and $0<\alpha<1$, the following are equivalent:
    \begin{itemize}
        \item[{\normalfont 1.}] $\gamma: [0,1] \to \mathcal{X}$ is a constant-speed geodesic.
        \item[{\normalfont 2.}] $\gamma: [0,1] \to \mathcal{X}$ is continuous and 
        $
        d(\gamma_0,\gamma_1)^p = \left(1-2^{-(p-\alpha p)}\right) |\gamma|^p_{b^{\alpha, p}}. 
        $
    \end{itemize}
    Notice that $\left(1-2^{-(p-\alpha p)}\right) \in (0,1)$ on the parameter range above. This characterization allows us to give the result below, which is a generalization of the \emph{metric Benamou--Brenier formula} \cite{BenamouBrenier2000}, \cite[Eq.\,(8.0.3)]{AGS2008GFs}, \cite[Corollary 4]{Lisini2007}, to the fractional setting. 
    Below is \cref{crl:dynamic_formulation_Wp} and $Geo([0,1];\mathcal{X})$ denotes the set of constant-speed geodesics.

    \begin{corollary}\label{crl:dynamic_formulation_Wp_bap_intro}
    Let $(\mathcal{X}, d)$ be a complete, separable, and geodesic metric space.  Let $1<p<\infty$ and $ 0< \alpha <  1$. Then for every $\mu,\nu \in P_p(\mathcal{X})$, we have 
    \begin{multline}
        W_p^p(\mu,\nu) = \\
        \left(1-2^{-(p-\alpha p)}\right)\min \left\{ \int_{\Gamma_1} |\gamma|_{b^{\alpha,p}}^p \d \pi (\gamma) : \quad \pi \in P(C([0,1];\mathcal{X})), \, (e_0)_\# \pi  = \mu, \, (e_1)_\# \pi  = \nu \right\}. 
    \end{multline}
    In addition, $\pi$ is a minimizer if and only if  $(e_0,e_1)_{\#} \pi \in \mathrm{OptCpl}(\mu,\nu)$ and  $\pi (Geo([0,1];\mathcal{X})) = 1$. 
\end{corollary}

\begin{table}
    \begin{center}
    \captionsetup{font=footnotesize}
    \caption{Under the assumption that $(\mathcal{X},d)$ is a complete separable metric space in which closed bounded sets are compact, the table shows when the energy functional \eqref{eq:energy_functional_intro} satisfies the two properties in \cref{prop:existence_of_minimizer_intro}, which guarantees the existence of a minimizer.
    } 
    \label{table:norms}
    {\footnotesize
    \renewcommand{\arraystretch}{1.5}
    \begin{tabular}{l|c|c|c|c|c}
        \hline
        \multirow{2}{*}{${\lvert \gamma \rvert}$} &
        \multirow{2}{*}{\shortstack{\textbf{parameter} \\ \textbf{range}}} &
        \multirow{2}{*}{\textbf{definition}} &
        \multirow{2}{*}{\shortstack{ \textcircled{\scriptsize{1}} \textbf{lower} \\\textbf{  semi-continuity} }} &
       \multirow{2}{*}{\shortstack{  \textcircled{\scriptsize{2}} \textbf{relatively} \\\textbf{ compact sublevels}}} &
       \multirow{2}{*}{\textbf{OK}}
       \\
       &&&&& 
       \\
        \hline
        $w_\delta (\gamma)$
        &
        $0 < \delta \leq T$
        &
        \cref{def:modulus_of_continuity}
        &
        $\checkmark$
        &
        $\times$
        &
        $\times$
        \\
        \hline
        $| \gamma |_{\upgamma\textrm{-}\mathrm{H\ddot{o}l}} $
        &
        $0<\upgamma \leq 1$
        &
        \cref{def:Holder_regularity}
        &
        $\checkmark$
        &
        $\checkmark$
        &
        $\checkmark$
        \\
        \hline
        $ | \gamma |_{p\textrm{-}\mathrm{var}} $
        &
        \multirow{2}{*}{$1 < p < \infty $}
        &
        \cref{def:p-variation}
        &
        $\checkmark$
        &
        $\times$
        &
        $\times$
        \\
        $ | \gamma |_{p\textrm{-}\mathrm{var}\textrm{-}\mathrm{limsup}}$
        &
        &
        \cref{eq:def_limsup_var}
        &
        $\times$
        &
        $\times$
        &
        $\times$
        \\
    \hline
        $| \gamma |_{W^{\alpha,p}}$
        &
        $1<p<\infty$ 
        &
        \cref{def:Walphap} 
        &
        \multirow{ 2}{*}{$\checkmark$}
        &
        \multirow{ 2}{*}{$\checkmark$}
        &
        \multirow{ 2}{*}{$\checkmark$}
        \\
        $| \gamma |_{b^{\alpha,p}}$
        &
        $\frac{1}{p}<\alpha < 1$
        &
        \cref{eq:Walphap_sum}
        &
        &
        \\
        \hline
        $| \gamma |_{W^{1,p}}$
        &
        $1<p<\infty$
        &
        \cref{def:W1p}
        &
        $\checkmark$
        &
        $\checkmark$
        &
        $\checkmark$
        \\
        \hline
    \end{tabular}
    }
    \end{center}
\end{table}

We conclude this section with a couple of remarks.
\begin{remark}[{Probabilistic formulation of \cref{prob:variational_problem} and measurability concerns}]\label{rmk:probabilistic_formulation}
    The variational problem for the case $\Gamma_T = C(I;\mathcal{X}) $ reads as follows in the language of stochastic analysis:
    
    \vspace{6pt}
    \begin{minipage}{0.97\textwidth}    \begin{problem}\label{prob:variational_problem_probabilistic}
        Let $\Psi: C(I;\mathcal{X}) \to [0,+\infty]$ be a measurable functional. Given $(\mu_t)_{t \in I} \subset P(\mathcal{X})$ with $I \coloneqq [0,T] \subset \mathbb{R}$, consider the variational  problem
        \begin{equation}\label{eq:variational_problem_probability_intro}
                \inf_{\mathrm{Process}(\mu_t)} \mathbb{E} \Big[ \Psi(X) \Big]
        \end{equation}
        where the infimum is taken over the set of all \emph{path-continuous stochastic processes} whose one-dimensional time marginals coincides with $(\mu_t)$:
        \begin{equation}\label{eq:set_process_intro}
            \mathrm{Process}(\mu_t) \coloneqq \Big\{ (\Omega,\mathcal{F},\mathbb{P}) \textrm{ and } X: (\Omega,\mathcal{F},\mathbb{P}) \to (C(I;\mathcal{X}),\mathcal{C}) : \quad (X_t)_{\#} \mathbb{P}  = \mu_t \text{ for all } t \in I \Big\}. 
        \end{equation}
    \end{problem}
    \end{minipage}
    
\vspace{6pt}
\noindent
In other words, the task is to find a probability space $(\Omega,\mathcal{F},\mathbb{P})$ and a stochastic process $(X_t)$ with marginals $(\mu_t)$ such that the expectation \eqref{eq:variational_problem_probability_intro} is minimized.
As before, if $\mathrm{Process}(\mu_t) = \emptyset$, the infimum \eqref{eq:variational_problem_probability_intro} is set $+\infty$.
As in the classical optimal transport, only the distribution of these processes enters the cost function. Thus, we will only focus on the formulation of \cref{prob:variational_problem}. 

\noindent We emphasize a remark on measurability.
The choices for $\Psi$ that we study involve a measurable norm $|\cdot| : C(I;\mathcal{X}) \to [0,+ \infty]$. However, this is not necessarily measurable as a function from $\mathcal{X}^I \to [0,+ \infty]$.
This is the case, for instance, for $|\cdot|_{\upgamma\textrm{-}\mathrm{H\ddot{o}l}}$ and $|\cdot|_{W^{\alpha,p}}$. Hence, it is important to keep in mind that the set $\mathrm{Process}(\mu_t)$ is \emph{not} the set of \emph{all} processes $X: (\Omega, \mathcal{F}, \mathbb{P}) \to (\mathcal{X}^I,\mathcal{B}(\mathcal{X})^I)$ but rather the set of \emph{path-continuous} processes $X: (\Omega,\mathcal{F},\mathbb{P}) \to (C(I;\mathcal{X}),\mathcal{C})$.
Otherwise, the integral \eqref{eq:variational_problem_probability_intro} is not defined.
We recall that $\mathcal{B}(\mathcal{X})^I$ and $\mathcal{C}$ are the corresponding $\sigma$-algebras defined as the smallest $\sigma$-algebras such that all evaluation maps in these path spaces are measurable.
The aforementioned measurability issue arises for the same reason that  $C(I;\mathcal{X}) \notin \mathcal{B}(\mathcal{X})^I$. See e.g. \cite{RogersWilliams2000_V1,StroockVaradhan2006,FrizVictoir2010book} and the discussion in \cref{subsec:path_measures}.
\end{remark}

\begin{remark}[Failure of quadratic variation in the sense of stochastic analysis]\label{rmk:quadratic_variation}
Given $\gamma \in C([0,T];\mathcal{X})$, the stochastic calculus quadratic variation $[\gamma]_T$ over $[0,T]$ is usually defined by taking a limit over shrinking partitions of $[0,T]$.  
By the example in \cref{rmk:non-lsc-inf-variation}, one sees that the map $\gamma \mapsto [\gamma]_T$ from $C([0,T];\mathcal{X})\to [0,+\infty]$ is not lower semi-continuous. Furthermore, in \cref{exp:infinitesimal_variation}, by taking $\upgamma = \tfrac{1}{2}$ and $p>2$, one sees that there exists a determinist curve $(\mu_t) \subset P_p(\mathcal{X})$ with non-zero quadratic variation with respect to $W_p$, whose only path-continuous process $(X_t)$, equivalently its lift $\pi$, has absolutely continuous sample paths. In particular, 
    \begin{equation}
        [\mu]_T^2 > \mathbb{E} \Big[ [X]_T^2 \Big] = \int_{\Gamma_T} [\gamma]_T^2 \d \pi = 0,
    \end{equation}
    which shows that $[\mu]_T$ cannot even be bounded by the quadratic variation energy of its lift.  
\end{remark}

\subsection{Motivating example}\label{subsec:motivating_examples}
Here, we would like to draw attention to the fact that low-regularity paths can also arise deterministically as solutions to the continuity equation, with an explicit example given in \cref{exp:infinitesimal_variation}. 
    
\begin{example}[\textbf{Continuity equation with a weaker integrability}]
    In the first case of this example, we revisit the well-known relationship between the continuity equation and absolutely continuous Wasserstein curves under the $p$-integrability condition, \cite[Chapter 8]{AGS2008GFs} and \cite[Section 1]{Lisini2007}.
    In the second case, we assume a weaker integrability condition and highlight the estimates applied in the fractional setting.

     \vspace{2pt}
    
    \noindent
    \textbf{Setting.} Let $\mathrm{d} \in \mathbb{N}$ and  $\mathcal{X} = \mathbb{R}^\mathrm{d}$ be equipped with the Euclidean distance. As before, $I \coloneqq [0,T] \subset \mathbb{R}$ is a time interval. Suppose that $(\mu_t)_{t \in [0,T]}$ is a narrowly continuous Borel probability measure-valued solution to the continuity equation (CE)
    \begin{align}\label{eq:CE}
        \partial_t \mu_t  & = - \nabla \cdot (\mu_t v_t) \quad \qquad \text{in } (0,T) \times \mathbb{R}^{\mathrm{d}}  
    \end{align}
        with initial condition $\mu_0 \in P(\mathbb{R}^{\mathrm{d}})$ and for some Borel velocity vector field $v : [0,T] \times \mathbb{R}^{\mathrm{d}} \to \mathbb{R}^{\mathrm{d}}$, where $v_t(x) \coloneqq v(t,x)$.
        Throughout this example, we assume that 
        the couple $(\mu_t,v_t)$ satisfies the (global) 1-integrability condition:        \begin{equation}\label{eq:integrability_CE_1}
            \int_{0}^{T} \int_{\mathbb{R}^{\mathrm{d}}} |v_t | \d \mu_t \d t < + \infty,
        \end{equation}
        which guarantees the conservation of mass.
        We interpret \eqref{eq:CE} in the distributional sense (i.e., in duality with smooth functions with compact support). 
        Under the assumption \eqref{eq:integrability_CE_1}, the superposition principle of \cite{Ambrosio2008} ensures the existence of a path measure $\pi \in P (C(I;\mathbb{R}^\mathrm{d}))$ such that
        \begin{enumerate}[label=(\roman*), font=\normalfont]
            \item $\pi$ is concentrated on the set of curves $\gamma \in AC(I;\mathbb{R}^\mathrm{d})$ with $\dot{\gamma}_t = v_t (\gamma_t)$ for a.e. $t \in I$;
            \item $(e_t)_\#\pi=\mu_t$ for all $t\in I$.
        \end{enumerate}

        \vspace{2pt}
    
        \noindent
        \textbf{Case 1 ($W^{1,p}$-regularity).}
        Assume $\mu_0  \in P_p(\mathbb{R}^{\mathrm{d}})$ and the $p$-integrability condition,
        \begin{equation}\label{eq:integrability_CE_p}
            \int_{0}^{T} \int_{\mathbb{R}^{\mathrm{d}}} |v_t |^p \d \mu_t \d t < + \infty,
        \end{equation}
        for some $1<p<\infty$. The first implication of these conditions is that they ensure $\mu_t \in P_p(\mathbb{R}^\mathrm{d})$ for all $t \in (0,T]$.
        We can rewrite \eqref{eq:integrability_CE_p} as the $W^{1,p}$-energy of the lift $\pi$:
        \begin{equation}\label{eq:integrability_CE_W1p}
            \int_{\Gamma_T} |\gamma|^p_{W^{1,p};[0,T]} \d \pi (\gamma) < + \infty,
        \end{equation}
        where $W^{1,p}$-regularity of an $p$-absolutely continuous curve $\gamma : [0,T] \to \mathbb{R}^{\mathrm{d}}$ over an interval $[s,t] \subset [0,T]$ is given by (see \cref{subsec:W1p} for the definition of $W^{1,p}$-regularity) $$|\gamma|_{W^{1,p};[s,t]} = \left( \int_{[s,t]} |\dot{\gamma}_r|^p \d r \right)^{1/p},$$
        and we omit the time interval when $[s,t] = [0,T]$. An application of \emph{H\"{o}lder's inequality} yields
        \begin{equation}\label{eq:Holder_ineq_intro}
            |\gamma_t - \gamma_s| \leq  |t-s|^{1 - \frac{1}{p}} | \gamma |_{W^{1,p};[s,t]}.
        \end{equation}
    Raising it to power $p$, integrating with respect to the measure $\pi$, and noting that $(e_{s},e_{t})_{\#} \pi \in \text{Cpl} (\mu_{s}, \mu_{t})$ is not necessarily optimal, we obtain an upper bound for the Wasserstein distance 
    \begin{equation}
        W_p^p(\mu_s,\mu_t) \leq |t-s|^{p - 1} \int_{\Gamma_T} | \gamma |^p_{W^{1,p};[s,t]} \d \pi (\gamma).
    \end{equation}
    This can be used to infer that $(\mu_t)$ is an absolutely continuous curve. Moreover, as shown in \cite[Theorem 4]{Lisini2007} by the Lebesgue differentiation theorem, we obtain an upper bound for the metric speed $|\dot{\mu}_t| \leq \Vert \dot{\gamma}_t \Vert_{L^p (\pi)}$ for a.e. $t \in I$. Raising it to power $p$ and integrating in time yields
    $$
    |\mu|_{W^{1,p}}^p \leq \int_{\Gamma_T} |\gamma|_{W^{1,p}}^p \d \pi (\gamma),
    $$
    where again $|\mu|_{W^{1,p}}^p = \int_{[0,T]} |\dot{\mu}_t|^p \d t $. 
    This norm will be replaced by $|\cdot|_{W^{\alpha,p}}$ in the second case. 
    
    \vspace{2pt}
    
    \noindent
    \textbf{Case 2 ($W^{\alpha,p}$-regularity).}
    We still assume that $(\mu_t,v_t)$ satisfy the 1-integrability condition \eqref{eq:integrability_CE_1}, but not the $p$-integrability or its equivalent formulation \eqref{eq:integrability_CE_W1p} for $p>1$ (we also assume that there is no other velocity vector field satisfying the $p$-integrability and the continuity equation). 
    Instead, let us assume a \emph{weaker} integrability condition directly on $\pi$:
    \begin{equation}\label{eq:integrability_CE_Walphap}
            \int_{\Gamma_T} |\gamma|^p_{W^{\alpha,p};[0,T]} \d \pi (\gamma) < + \infty,
    \end{equation}
    for some $1<p<\infty$ and $1/p<\alpha < 1$, with the same initial condition $\mu_0  \in P_p(\mathbb{R}^{\mathrm{d}})$. These again ensure $(\mu_t) \subset P_p(\mathbb{R}^{\mathrm{d}})$ by \cref{thm:lift_to_mu_Walphap_intro}. 
    Such a situation as above is present in \cref{exp:infinitesimal_variation}, where the superposition of absolutely continuous curves results in, instead of an absolutely continuous Wasserstein curve, a H\"{o}lder curve of infinite length.
    To proceed, let us denote the $W^{\alpha,p}$-regularity of a curve $\gamma : [0,T] \to \mathbb{R}^{\mathrm{d}}$ over an interval $[s,t] \subset [0,T]$ by
    \begin{equation}
        | \gamma |_{W^{\alpha,p};[s,t]} \coloneqq \left( \iint_{[s,t]^2} \frac{|\gamma_u - \gamma_v|^p}{|u-v|^{1+\alpha p}} \d u \d v \right)^{1/p}.
    \end{equation}
    In the fractional setting, the H\"{o}lder's inequality \eqref{eq:Holder_ineq_intro} is replaced by the \emph{Garsia--Rodemich--Rumsey inequality} \cite{Garsia_Rodemich_Rumsey1970,FrizVictoir2010book}
    \begin{equation}
        |\gamma_t - \gamma_s| \leq \bar{c} |t-s|^{\alpha - \frac{1}{p}} | \gamma |_{W^{\alpha,p};[s,t]},
    \end{equation}
    where $ \bar{c} = \bar{c} (\alpha,p)$ is a constant. 
    Again, raising it to power $p$ and integrating it with respect to $\pi$, we get
    \begin{equation}
        W_p^p(\mu_s,\mu_t) \leq \bar{c}^p |t-s|^{\alpha p - 1} \int_{\Gamma_T} | \gamma |^p_{W^{\alpha,p};[s,t]} \d \pi (\gamma),
    \end{equation}
    which immediately implies that $(\mu_t)$ is a continuous curve. 
    Similar to the previous case, one can easily show, as stated in \cref{thm:lift_to_mu_Walphap_intro}, that $W^{\alpha,p}$-regularity of $(\mu_t)$ is finite: 
    $$
    |\mu|_{W^{\alpha,p}}^p \leq \int_{\Gamma_T} |\gamma|_{W^{\alpha,p}}^p \d \pi (\gamma).
    $$
\end{example}

\subsection{Organization of the paper}
The rest of the paper is structured as follows:

\begin{itemize}
    \item \textbf{\cref{sec:preliminaries_d} (Preliminaries).}
     We collect the required notions and results. In addition, the following two elementary results may be of independent interest: 
    \begin{itemize}[label=$\circ$,leftmargin=*]
        \item \cref{lemma:balphap_Xn}. Computation of $b^{\alpha,p}$-regularity for piece-wise geodesic curves. 
        \item \cref{crl:tightness_Holder_discrete}. A discrete relaxation of the well-known Kolmogorov--Lamperti tightness condition for deterministic path measures.
    \end{itemize} 
    \item \textbf{\cref{sec:results_d} (Main results).} We prove the main theorems. 
    \item \textbf{\cref{sec:counterexamples_d} (Counterexamples).}  We give examples to understand non-compatibility. 
\end{itemize}

\subsection{Acknowledgements}
The author would like to thank his supervisor, Matthias Erbar, for invaluable guidance and numerous discussions. He also thanks Timo Schultz for insightful conversations, which significantly improved the results and eventually led to  \cref{exp:infinitesimal_variation,exp:compatibility_vs_noncompatibility,exp:non_compatibile_infinite_energy}. He also thanks Zhenhao Li for helpful comments on some technical details. Additionally, he extends his thanks to Kohei Suzuki and Vitalii Konarovskyi for insightful discussions. This work is supported by the Deutsche Forschungsgemeinschaft (DFG, German Research Foundation) -- Project-ID 317210226 -- SFB 1283. 

\section{Preliminaries}\label{sec:preliminaries_d}
\subsection{Continuous paths on metric spaces}\label{subsec:cont_paths}
      Throughout the paper, $(\mathcal{X},d)$ is a metric space and $[0,T] \subset \mathbb{R}$ is a time interval. Any additional structure on $(\mathcal{X},d)$ will be explicitly stated.
      Let $\Gamma_T \coloneqq C([0, T ]; \mathcal{X})$ denote the space of all continuous paths $X: [0,T] \to \mathcal{X}$.
      When $T=1$, we write $\Gamma \coloneqq \Gamma_1$.       
      For any two paths $X,Y \in \Gamma_T$, the supremum distance is defined by 
    \begin{equation}\label{eq:sup_distance}
            d_{\infty}(X,Y) \coloneqq \sup_{t \in [0,T]} d(X_t,Y_t),
    \end{equation}
    and the topology induced by which is called supremum or uniform topology. If $(\mathcal{X}, d)$ is complete and separable, so is  $(\Gamma_T, d_\infty)$. For compactness in $\Gamma_T$, we need the modulus of continuity:
    \begin{definition}[Modulus of continuity]\label{def:modulus_of_continuity}
        Given $0 < \delta \leq T$, the modulus of continuity of a path $X \in C([0,T];\mathcal{X})$ is defined by 
        \begin{equation}\label{eq:modulus_of_continuity}
        w_\delta (X) \coloneqq \sup_{|t-s|\leq \delta } d(X_s,X_t).
    \end{equation}
    When $\delta = T$, we write $|X|_0 \coloneqq w_T(X)$.
    \end{definition}
    
    The Arzel\`a-Ascoli theorem provides a characterization of relative compactness in $\Gamma_T$ (see, e.g., \cite[Theorem 7.2]{Billingsley1999} or \cite[Theorem 1.4]{FrizVictoir2010book}).
    \begin{theorem}[Arzel\`a-Ascoli]\label{thm:Arzela-Ascoli}
        Let $(\mathcal{X},d)$ be a complete metric space in which closed bounded sets are compact. 
        A set $\mathcal{A} \subset C([0,T];\mathcal{X})$ is relatively compact if and only if the following holds
         \begin{enumerate}[label=(\roman*), font=\normalfont]
            \item $\mathcal{A}$ is bounded at $t=0$, i.e., given an arbitrary fixed point $\bar{x} \in \mathcal{X},
            $\begin{equation}
                \sup_{\gamma \in \mathcal{A}}  d (\gamma_0, \bar{x}) < \infty,
            \end{equation} 
            \item $\mathcal{A}$ is equicontinuous, i.e., 
            \begin{equation}
                \lim_{\delta \to 0} \sup_{\gamma \in \mathcal{A} } w_{\delta} ( \gamma) = 0. 
            \end{equation} 
        \end{enumerate}
    \end{theorem} 
    \begin{remark}
        Arzel\`a-Ascoli theorem is sometimes formulated with the following condition replaced with (i)
        \begin{enumerate}
            \item[(i$'$)] $\mathcal{A}$ is bounded, i.e., 
            \begin{equation}
                \sup_{\gamma \in \mathcal{A}} \sup_{t \in [0,T]}d (\gamma_t, \bar{x}) < \infty.
            \end{equation}
        \end{enumerate}
        Clearly, (i$'$) $\Rightarrow$ (i). 
        Conversely, (i) and equicontinuity imply (i$'$).
        Hence, one can use either of these conditions.
        We will use the version (i) as stated in the theorem. 
    \end{remark}
    \begin{lemma}[{Lower semi-continuity of $\gamma \mapsto w_\delta(\gamma)$}]
            Let $0<\delta \leq T$.
            The map $\gamma \mapsto w_\delta(\gamma) $ from $ \Gamma_T \to [0,\infty)$ is lower semi-continuous with respect to pointwise convergence, and in particular, with respect to uniform convergence. 
        \end{lemma}
        \begin{proof}
            Let $(\gamma^n) \subset \Gamma_T$, $n \in \mathbb{N}$, be a sequence of paths such that $\gamma^n \to \gamma$ pointwise on $[0,T]$.
            Take arbitrary time points $s,t \in [0,T]$ such that $|t-s|\leq  \delta$. We have 
            \begin{align}
                d(\gamma_s,\gamma_t)
                 = \liminf_{n\to \infty } d(\gamma_s^n,\gamma_t^n)  \leq \liminf_{n\to \infty } \sup_{|v-u|\leq \delta}  d(\gamma_u^n,\gamma_v^n) = \liminf_{n\to \infty } w_\delta (\gamma^n)
            \end{align}
            Taking the supremum over all $|t-s| \leq  \delta$ yields the result. 
        \end{proof}

        \begin{remark}[Non-compactness of sublevels of $\gamma \mapsto d (\gamma_0, \bar{x} ) +  w_\delta (\gamma)$]\label{lemma:non_compact_sublevel_modulus}
             Let $\mathcal{X} = \mathbb{R}^{\mathrm{d}}$, $\bar{x} = 0$, and $\delta = T$. Consider continuous curves in the sublevels of $\gamma \mapsto |\gamma_0| +  w_T (\gamma)$. These are continuous functions bounded by a constant, which obviously do not need to be equicontinuous.
        \end{remark}

\subsection{H\"{o}lder paths on metric spaces}\label{subsec:Holder_paths}

    \begin{definition}[{$\upgamma$-H\"{o}lder} continuity]\label{def:Holder_regularity}
        Given $\upgamma \in [0,1]$, the $\upgamma$-H\"{o}lder continuity of a  path $X \in C ([0,T] ; \mathcal{X})$ over $[s,t] \subset [0,T]$ is defined by 
        \begin{equation}
            | X |_{\upgamma\textrm{-}\mathrm{H\ddot{o}l};[s,t]} \, \coloneqq \sup_{s \leq u < v \leq t} \frac{d(X_u,X_v)}{|v-u|^\upgamma}.
        \end{equation}
        $C^{\upgamma\textrm{-} \mathrm{H\ddot{o}l}} ([0,T];\mathcal{X})$ denotes the set of all paths
        $X \in C ([0,T] ; \mathcal{X})$ such that 
        \begin{equation}
            | X |_{\upgamma\textrm{-}\mathrm{H\ddot{o}l}}  \coloneqq | X |_{\upgamma\textrm{-}\mathrm{H\ddot{o}l};[0,T]} < \infty.
        \end{equation}
    \end{definition}
     From the above and \cref{def:modulus_of_continuity}, we have that $C^{0\textrm{-} \mathrm{H\ddot{o}l}} ([0,T];\mathcal{X}) = C ([0,T];\mathcal{X})$ and 
    \begin{equation}\label{eq:def:0_Holder_cont}
         | X |_{0\textrm{-}\mathrm{H\ddot{o}l}} \coloneqq \sup_{0\leq u < v \leq T} d(X_u,X_v) \eqqcolon | X |_{0}.
    \end{equation}
    Below, a simple yet useful characterization of H\"{o}lder continuity is provided. It states that for a continuous curve $X:[0,1] \to \mathcal{X}$ to be H\"{o}lder, it is enough to verify the H\"{o}lder condition only at the dyadic time points of $[0,1]$.
    The result is taken from \cite[Lemma 2]{LyonsVictoir2007}, whose proof is inspired by the classical proof of Kolmogorov--\v Centsov continuity theorem (see e.g. \cite[Theorem 2.8]{KaratzasShreve2012}).
     
    \begin{theorem}[{A discrete characterization of $C^{\upgamma\textrm{-} \mathrm{H\ddot{o}l}}$ \cite[Lemma 2]{LyonsVictoir2007}}]\label{thm:disc_char_Holder}
            Given $0 < \upgamma \leq 1$, the following are equivalent:
            \begin{itemize}
                \item $X:[0,1] \to \mathcal{X}$ satisfies
                \begin{equation}\label{eq:Holder_cont_cond}
                    d\big(X_s, X_t\big) \leq c |t-s|^\upgamma \qquad \forall t,s \in [0,1],
                \end{equation}
                for some constant $c$.
                \item $X:[0,1] \to \mathcal{X}$ is continuous and satisfies
                \begin{equation}\label{eq:Holder_disc_cond}
                     d \big(X_{t_{k}^{(m)}},X_{t_{k+1}^{(m)}}\big) \leq \tilde{c} |\Delta t _m |^\upgamma \qquad \forall m \in \mathbb{N}_0, \, k \in \{0,1,\cdots, 2^m-1\},
                \end{equation}
                  for some  constant $\tilde{c}$, where $t_k^{(m)} \coloneqq \frac{k}{2^m} $ and $\Delta t_m \coloneqq \frac{1}{2^m}$.  
            \end{itemize}
        \end{theorem}    
        While one direction is trivial, we stated the theorem in this way for its resemblance to the forthcoming \cref{thm:Walphap_balphap}. It is important to note that in the second condition, we evaluate $X$ only at a countable dense subset of the interval $[0,1]$. Thus, continuity needs to be added as an extra assumption.
        The result above differs slightly from the original statement in \cite{LyonsVictoir2007}, where a (unique) continuous modification is constructed (for which one needs the space to be complete). Here due to the continuity assumption, the continuous modification coincides with the original curve at all times (and no additional assumption on the metric space is needed).
        
        We recall some properties of the H\"{o}lder regularity. For below, see e.g. \cite[Lemma 5.12]{FrizVictoir2010book}. 
        \begin{lemma}[{Lower semi-continuity of $\gamma \mapsto | \gamma |_{\upgamma\textrm{-}\mathrm{H\ddot{o}l}}$}]
            Let $0<\upgamma\leq 1$.
            The map $\gamma \mapsto | \gamma |_{\upgamma\textrm{-}\mathrm{H\ddot{o}l}} $ from $\Gamma_T \to [0,\infty]$ is lower semi-continuous with respect to pointwise convergence, and in particular, with respect to uniform convergence. 
        \end{lemma}

        As a straightforward consequence of Arzel\`a-Ascoli, we have:
        \begin{lemma}[Compact sublevels of $\gamma \mapsto d (\gamma_0, \bar{x} ) +   | \gamma |_{\upgamma\textrm{-}\mathrm{H\ddot{o}l}}$]\label{lemma:compact_sublevel_Holder}
            Let $(\mathcal{X},d)$ be a complete metric space in which closed bounded sets are compact.
            Given $0<\upgamma\leq 1$ and an arbitrary point $\bar{x}\in \mathcal{X}$, the map $\gamma \mapsto \Psi (\gamma) \coloneqq d (\gamma_0, \bar{x} ) +  | \gamma |_{\upgamma\textrm{-}\mathrm{H\ddot{o}l}}$ from $\Gamma_T \to [0,+\infty]$ has compact sublevels in $\Gamma_T$. 
        \end{lemma}

\subsection{\texorpdfstring{$p$}{p}-variation paths on metric spaces}\label{subsec:var_paths}
    Let $$D \coloneqq \big\{ s = t_0 < t_1 < \cdots < t_N = t \big\}, \quad N \in \mathbb{N}$$
    be a dissection of the time interval $[s,t] \subset \mathbb{R}$.
    The mesh of $D$ is defined as
    $$
    | D | \coloneqq \max_{ i \in \{1,\cdots, N \} }  |t_{i} - t_{i-1}|. 
    $$
    The $p$-variation ($p \geq 1 $) of a path $X:[0,T]\to \mathcal{X}$ over a fixed dissection $D$ is defined as
    \begin{equation}
        \sum_{t_i \in D} d(X_{t_{i} }, X_{t_{i+1}} )^p,
    \end{equation}
    with the convention $t_{N+1} = t_{N}$.
    We let $\mathcal{D}([s,t])$ denote the set of all partitions of $[s,t]$ and $\mathcal{D}_{\delta}([s,t])$ denote the set of all partitions of $[s,t]$ whose mesh size is less than or equal to $\delta>0$.

    \begin{definition}[{$p$-variation}]\label{def:p-variation}
         Given  $p\geq 1$, the $p$-variation of a path $X\in C([0,T]; \mathcal{X})$ over $[s,t] \subset [0,T]$ is defined by 
         \begin{equation}
             | X |_{p\textrm{-}\mathrm{var};[s,t]} \coloneqq \left( \sup_{t_i \in \mathcal{D}([s,t])} \sum_{i} d(X_{t_i}, X_{t_{i+1}})^p\right)^{1/p}. 
         \end{equation}
         $C^{p\textrm{-} \mathrm{var}} ([0,T];\mathcal{X})$ denotes the set of all continuous paths $X$ such that
        \begin{equation}
            | X |_{p\textrm{-}\mathrm{var}}  \coloneqq | X |_{p\textrm{-}\mathrm{var};[0,T]} < \infty.
        \end{equation}
    \end{definition}
    
    We collect some useful results on $p$-variation. First, we recall that if $1 \leq q \leq p < \infty $, then
    \begin{equation}\label{eq:var_p_q_embedding}
        C^{q\textrm{-} \mathrm{var}}([0,T];\mathcal{X}) \subset C^{p\textrm{-} \mathrm{var}}([0,T];\mathcal{X}).
    \end{equation}
    Two subsequent lemmas are taken from \cite[Proposition 5.6]{FrizVictoir2010book}.

    \begin{lemma}[An infinitesimal characterization of $C^{p\textrm{-} \mathrm{var}}$]\label{lemma:infinitesimal_char_var}
        Let $p\geq 1$ and $X \in C([0,T];\mathcal{X})$. Then the following are equivalent:
        \begin{itemize}
            \item $X$ is of finite $p$-variation, i.e., 
            \begin{equation}\label{eq:def_sup_var}
                 | X |_{p\textrm{-}\mathrm{var}} \coloneqq \left( \sup_{t_i \in \mathcal{D}([0,T])} \sum_{i} d(X_{t_i}, X_{t_{i+1}})^p\right)^{1/p} < \infty . 
             \end{equation}                \item $X$ satisfies
            \begin{equation}\label{eq:def_limsup_var}
                | X |_{p\textrm{-}\mathrm{var}\textrm{-}\mathrm{limsup}} \coloneqq \left( \limsup_{\delta \to 0 } \sup_{t_i \in \mathcal{D}_{\delta}([0,T])} \sum_{i} d(X_{t_i}, X_{t_{i+1}})^p \right)^{1/p} < \infty. 
            \end{equation}
        \end{itemize}
    \end{lemma}
    To distinguish these norms, we refer to the first one as \emph{variation}, while the second one as \emph{infinitesimal variation}, emphasizing that they are not necessarily identical.
    This is only the case for $p=1$ (see e.g. \cite[Proposition 1.14]{FrizVictoir2010book}):
    \begin{equation}\label{eq:var_limsupvar_1}
        | X |_{1\textrm{-}\mathrm{var}}  = | X |_{1\textrm{-}\mathrm{var}\textrm{-}\mathrm{limsup}}.
    \end{equation}
   Regarding this distinction, the following lemma is useful:
    \begin{lemma}\label{lemma:limsupvar0}
        Let $1 \leq q < p < \infty$ and $X \in C^{q\textrm{-} \mathrm{var}} ([0,T];\mathcal{X})$. Then $| X |_{p\textrm{-}\mathrm{var}\textrm{-}\mathrm{limsup}} = 0$. 
    \end{lemma}

    As a next result, we recall the lower semi-continuity of $p$-variation, as stated in \cite[Lemma 5.12]{FrizVictoir2010book}.

    \begin{lemma}[{Lower semi-continuity of $\gamma \mapsto | \gamma |_{p\textrm{-}\mathrm{var}}$}]
        Let $1 \leq p < \infty$.
        The map $\gamma \mapsto |\gamma |_{p\textrm{-}\mathrm{var}}$ from $ \Gamma_T \to [0,\infty]$ is lower semi-continuous with respect to pointwise convergence, and in particular, with respect to uniform convergence. 
    \end{lemma}
    
    Although $| X |_{p\textrm{-}\mathrm{var}\textrm{-}\mathrm{limsup}}$ characterizes curves of finite $p$-variation in an infinitesimal way, it does not enjoy many of the nice properties that $| X |_{p\textrm{-}\mathrm{var}}$ does.
    For instance, in the case $p > 1 $, unlike $X \mapsto | X |_{p\textrm{-}\mathrm{var}}$, the mapping $X \mapsto  | X |_{p\textrm{-}\mathrm{var}\textrm{-}\mathrm{limsup}}$ is not necessarily lower semi-continuous, as shown in the example below. 
    
    \begin{remark}[{Non-lower semi-continuity of $ \gamma \mapsto | \gamma |_{p\textrm{-}\mathrm{var}\textrm{-}\mathrm{limsup}}$} when $1<p$]\label{rmk:non-lsc-inf-variation}
         Let $(\mathcal{X}, d)$ be a geodesic space. Take $p > 1$ and let $\gamma$ be a continuous path such that $ 0 < | \gamma |_{p\textrm{-}\mathrm{var}\textrm{-}\mathrm{limsup}}< \infty$. Let $(D_n)_{n \in \mathbb{N}}$ be a sequence of dissections of $[0,T]$ with shrinking mesh size $|D_n| \to 0$. Define $\gamma^{D_n}$ to be the piecewise geodesic approximation of $\gamma$ (i.e. the curve that coincides with $\gamma$ at points of $D_n$ and connects in between with geodesics). We know that $\gamma^{D_n} \to \gamma$ uniformly (see e.g. \cite[Lemma 5.19]{FrizVictoir2010book}). For each $n$, the path $\gamma^{D_n} $ is of bounded $1$-variation, which implies, by \cref{lemma:limsupvar0}, that $ | \gamma^{D_n} |_{p\textrm{-}\mathrm{var}\textrm{-}\mathrm{limsup}}=0$ for all $p>1$. Consequently, 
         $$
         \liminf_{n \to \infty} |\gamma^{D_n} |_{p\textrm{-}\mathrm{var}\textrm{-}\mathrm{limsup}}=0 < | \gamma |_{p\textrm{-}\mathrm{var}\textrm{-}\mathrm{limsup}}, 
         $$
         which shows non-lower semi-continuity of the map $ \gamma \mapsto | \gamma |_{p\textrm{-}\mathrm{var}\textrm{-}\mathrm{limsup}}$ at this point. 
    \end{remark}

    \begin{remark}[Non-compactness of sublevels of $\gamma \mapsto d (\gamma_0, \bar{x} ) +  | \gamma |_{p\textrm{-}\mathrm{var}}$ and $  | \gamma |_{p\textrm{-}\mathrm{var}\textrm{-}\mathrm{limsup}}$]
         As in \cref{lemma:non_compact_sublevel_modulus}, even on $\mathcal{X} = \mathbb{R}^{\mathrm{d}}$, paths whose $p$-variations are bounded by a constant are not necessarily equicontinuous.
    \end{remark}
         
\subsection{Fractional Sobolev paths on metric spaces}\label{subsec:Walphap}
    Fractional Sobolev regularity is particularly useful for our work as it captures the regularity of the paths in both small and large time intervals. This is also reflected in the fact that the elements in fractional Sobolev spaces are of finite variation and have certain H\"{o}lder regularity (see the embedding \eqref{eq:FS_H_v_embeddings} below).
    Let us first recall the classical definition:
    \begin{definition}[{Fractional  Sobolev space $W^{\alpha, p}$}]\label{def:Walphap}
        Given $1\leq p<\infty$ and $0<\alpha < 1$, the fractional Sobolev regularity of a measurable function $X: [0,T] \to \mathcal{X}$ over $[s,t] \subset [0,T]$ is defined by  
        \begin{equation}
            | X |_{W^{\alpha,p};[s,t]} \coloneqq \left( \iint_{[s,t]^2} \frac{d(X_u,X_v)^p}{|v-u|^{1+\alpha p}} \d u \d v \right)^{1/p}.
        \end{equation}
        The fractional Sobolev space $W^{\alpha, p}([0,T]; \mathcal{X})$ is the space of measurable functions $X$ such that
        \begin{equation}
            | X |_{W^{\alpha,p}} \coloneqq | X |_{W^{\alpha,p};[0,T]} < \infty.
        \end{equation}
    \end{definition}
    The \emph{fractional Sobolev} spaces $W^{\alpha, p}$ (also known as Sobolev--Slobodeckij spaces) can be viewed as a specific instance of the broader class of \textit{Besov} spaces $B^{\alpha,p,q}$ (see e.g. \cite[Equation (2.1)]{LiuPromelTeichmann2020}). In particular, when $q=p$, we have $B^{\alpha,p,p} = W^{\alpha, p}$. 
    While the classical definition of $B^{\alpha,p,q}$ typically involves double integrals similar to the one mentioned above, there are also other function spaces whose definition involves double sums. Recently, \cite{LiuPromelTeichmann2020} has proved that these definitions are in fact equivalent. 
    For the purpose of our work, we only give their result for the case $q=p$: 
    
    \begin{theorem}[{A discrete characterization of $W^{\alpha,p}$ \cite[Theorem 2.2]{LiuPromelTeichmann2020}}]\label{thm:Walphap_balphap}
        Given $1<p<\infty$ and $\frac{1}{p}<\alpha < 1$, the following are equivalent:
        \begin{itemize}
            \item $X:[0,1] \to \mathcal{X}$ is measurable and  
            \begin{equation}\label{eq:Walphap_integral}
                | X |_{W^{\alpha,p}} \coloneqq \left( \iint_{[0,1]^2} \frac{d(X_s,X_t)^p}{|t-s|^{1+\alpha p}} \d s \d t \right)^{1/p} < \infty.
            \end{equation}
            \item $X:[0,1] \to \mathcal{X}$ is continuous and 
            \begin{equation}\label{eq:Walphap_sum}
                | X |_{b^{\alpha,p}} \coloneqq  \left( \sum_{m=0}^{\infty} \sum_{k=0}^{2^m-1} \frac{d \big(X_{t_{k}^{(m)}}, X_{t_{k+1}^{(m)}}\big)^p }{|\Delta t_m|^{1+\alpha p }}  |\Delta t_m|^2 \right)^{1/p} < \infty,
            \end{equation}
             where $t_k^{(m)} \coloneqq \frac{k}{2^m} $ and $\Delta t_m \coloneqq \frac{1}{2^m}$.
        \end{itemize}
        Furthermore, $| \cdot |_{W^{\alpha,p}}$ and $| \cdot |_{b^{\alpha,p}} $  are equivalent on the set of continuous paths, i.e., there exist positive constants  $c_1,c_2$ depending only on $(\alpha, p)$ such that
        \begin{equation}\label{eq:equiv_norm_Walphap}
             c_1 | X |_{W^{\alpha,p}} \leq | X |_{b^{\alpha,p}}   \leq c_2  | X |_{W^{\alpha,p}}
        \end{equation}
        for all $X \in C([0,1];\mathcal{X})$. 
    \end{theorem}
    This result is reminiscent of the discrete characterization of H\"{o}lder curves provided by \cite[Lemma 2]{LyonsVictoir2007}, as stated in \cref{thm:disc_char_Holder}.
    In the second definition, we are once again evaluating $X$ only at dyadic points (i.e. a countable subset of $[0,1]$) but this time, in the form of a double sum.
    Thus, continuity must be an additional assumption.
    This is, however, not needed for the first definition. In fact, under the conditions $1<p<\infty$ and $\frac{1}{p}<\alpha < 1$, the finiteness of  $| X |_{W^{\alpha,p}}$  automatically implies the continuity of $X$ (see \cref{thm:FractionalSobolev-Holder}).
    We repeatedly rely on the equivalence result above. We proceed with the following simple observation:
    
         \begin{remark}[A (trivial) H\"{o}lder-Fractional Sobolev embedding]\label{rmk:Holder-FractionalSobolev}
             It is easy to check that under $0 < \alpha < \upgamma \leq 1$, the $W^{\alpha,p}$-regularity of a curve $X \in C^{\upgamma\textrm{-} \mathrm{H\ddot{o}l}}([0,1];\mathcal{X} )$ is finite. 
             The first norm \eqref{eq:Walphap_integral} can be estimated by
            \begin{align}
                | X |_{W^{\alpha,p}}^p
                & \leq | X |_{\upgamma\textrm{-}\mathrm{H\ddot{o}l}}^p \iint_{[0,1]^2} \frac{1}{|t-s|^{1+\alpha p - \upgamma p}} \d s \d t, \\
                & = | X |_{\upgamma\textrm{-}\mathrm{H\ddot{o}l}}^p \frac{2}{(\upgamma p - \alpha p)(\upgamma p - \alpha p+1)} < + \infty, 
            \end{align}
            where the double integral is finite only when $\alpha  < \upgamma$. Likewise, the second norm \eqref{eq:Walphap_sum} can be estimated by 
            \begin{align}
                | X |^p_{b^{\alpha,p}} 
                & \leq | X |^p_{\upgamma\textrm{-}\mathrm{H\ddot{o}l}} \sum_{m=0}^{\infty} 2^{m  (\alpha p - \upgamma p )} \\
                & =  | X |^p_{\upgamma\textrm{-}\mathrm{H\ddot{o}l}} \frac{1}{1-2^{-(\upgamma p - \alpha p)}} < + \infty,
                \end{align}
            where the geometric series converges only when $\alpha  < \upgamma$.
           In particular, we have the continuous embedding
            \begin{equation}
                C^{\upgamma\textrm{-}\mathrm{H\ddot{o}l}} \subset W^{\alpha, p} \quad \textrm{if} \quad  \alpha  < \upgamma.
            \end{equation}
        \end{remark}
        
        The estimate \eqref{eq:distance_estimate_GRR} below can be derived by applying the \emph{Garsia--Rodemich--Rumsey inequality}, which was originally introduced in \cite{Garsia_Rodemich_Rumsey1970} for the case $\mathcal{X} = \mathbb{R}$.
        A proof for $\mathbb{R}^\mathrm{d}$ can be found in \cite[Theorem 2.1.3]{StroockVaradhan2006} and for the general metric setting in \cite[Theorem A.1]{FrizVictoir2010book}). The results below are well-known \cite[Corollary A.2-3]{FrizVictoir2010book}. 
        
        \begin{theorem}[Fractional Sobolev-H\"{o}lder and -variation embeddings]\label{thm:FractionalSobolev-Holder}  
            Given $1<p<\infty$ and $ \frac{1}{p} < \alpha < 1$, let $X \in W^{\alpha, p} ([0,T];\mathcal{X})$.
            Then there exists a constant $\bar{c}$ depending only on $(\alpha, p)$ such that for all $0 \leq s < t \leq T$,
            \begin{equation}\label{eq:distance_estimate_GRR}
                d (X_s,X_t) \leq \bar{c} |t-s|^{\alpha - \frac{1}{p}} | X |_{W^{\alpha,p};[s,t]},
            \end{equation}
             and in particular,
            \begin{align}
                | X |_{\alpha - \frac{1}{p}\textrm{-}\mathrm{H\ddot{o}l};[s,t]}  & \leq \bar{c} | X |_{W^{\alpha,p};[s,t]}, \\
                | X |_{\frac{1}{\alpha}\textrm{-}\mathrm{var};[s,t]}
                & \leq \bar{c} |t-s|^{\alpha-\frac{1}{p}} | X |_{W^{\alpha,p};[s,t]}, 
            \end{align}
            where a possible choice of the constant is $\bar{c} = \big( 32 \frac{\alpha p +1}{\alpha p -1} \big)^{1/p}$. 
        \end{theorem}
        As a result of the theorem above, we have the following continuous embeddings:
        \begin{equation}\label{eq:FS_H_v_embeddings}
            W^{\alpha, p} \subset C^{\alpha - \frac{1}{p}\textrm{-}\mathrm{H\ddot{o}l}} \quad \textrm{and} \quad W^{\alpha, p} \subset C^{\frac{1}{\alpha}\textrm{-}\mathrm{var}}. 
        \end{equation}
        Note that from \eqref{eq:distance_estimate_GRR}, we can immediately conclude that  $X$ has finite $({\alpha - 1/p})^{-1}$-variation.
        But, as shown in \cite[Theorem 2]{FrizVictoir2006}, by applying H\"{o}lder's inequality with a clever choice of exponents, an even stronger statement can be made, namely, $X$ has finite $({\alpha })^{-1}$-variation.
        A consequence of this embedding is that under the conditions $1<p<\infty$ and $ \frac{1}{p} < \alpha < 1$, the elements in $W^{\alpha,p}$ have zero infinitesimal $p$-variation by \cref{lemma:limsupvar0}:
        \begin{equation}
            X \in W^{\alpha,p} \quad \Rightarrow \quad | X |_{p\textrm{-}\mathrm{var}\textrm{-}\mathrm{limsup}} = 0, 
        \end{equation}
        which is not unexpected once the discrete definition \eqref{eq:Walphap_sum} of this space is considered. 
        We now discuss some properties of the $W^{\alpha,p}$-semi-norm.  
        \begin{lemma}[Lower semi-continuity of $\gamma \mapsto|\gamma |_{W^{\alpha,p}}$ and $\gamma \mapsto|\gamma |_{b^{\alpha,p}}$]\label{prop:l.s.c.functional}
            The maps $\gamma \mapsto | \gamma |_{W^{\alpha,p}}$ and $\gamma \mapsto|\gamma |_{b^{\alpha,p}}$ both from $ \Gamma_T \to [0,\infty]$ are lower semi-continuous with respect to pointwise convergence, and in particular, with respect to uniform convergence. 
        \end{lemma}
        \begin{proof}
            Since the function $|\cdot |_{b^{\alpha,p}}$ is the limit of a monotone non-decreasing sequence of continuous functions, it is indeed lower semi-continuous. To confirm the other statement, let $(\gamma^n) \subset W^{\alpha,p} ([0,T];\mathcal{X})$, $n \in \mathbb{N}$, be a sequence of  continuous paths such that $\gamma^n \to \gamma$ pointwise on $[0,T]$.
            It is enough to show the lower semi-continuity of the $p$-th power of the function:
            \begin{equation}
                | \gamma |^p_{W^{\alpha,p}} \leq \liminf_{n \to \infty } | \gamma^n |^p_{W^{\alpha,p}}.
            \end{equation}
            First, triangle inequality yields
            \begin{equation}
                |d(\gamma^n_s,\gamma^n_t) - d(\gamma_s,\gamma_t)| \leq d(\gamma_t^n,\gamma_t) + d(\gamma_s^n,\gamma_s).
            \end{equation}
            By taking the limit as $n \to \infty$, two terms on the right-hand side vanish and we obtain
            \begin{equation}\label{eq:pointwise_conv_dst}
                \lim_{n \to \infty} d(\gamma^n_s,\gamma^n_t) = d(\gamma_s,\gamma_t).
            \end{equation}
            Second, observe that
            \begin{align}
                \liminf_{n \to \infty } | \gamma^n |_{W^{\alpha,p}}^p & =  \liminf_{n \to \infty } \iint_{[0,T]^2} \frac{d(\gamma^n_s,\gamma^n_t)^p}{|t-s|^{1+\alpha p}} \d s \d t \\
                & \geq  \iint_{[0,T]^2} \left( \liminf_{n \to \infty }  \frac{d(\gamma^n_s,\gamma^n_t)^p}{|t-s|^{1+\alpha p}} \right) \d s \d t \\
                & = \iint_{[0,T]^2} \frac{d(\gamma_s,\gamma_t)^p}{|t-s|^{1+\alpha p}} \d s \d t = | \gamma |_{W^{\alpha,p}}^p
            \end{align}
            where we used Fatou's lemma and \eqref{eq:pointwise_conv_dst}.
            The final claim of the proposition follows from the fact that uniform convergence implies pointwise convergence. 
        \end{proof}
        \begin{lemma}[Compact sublevels of $\gamma \mapsto d (\gamma_0, \bar{x} ) +  | \gamma |_{W^{\alpha,p}}$]\label{lemma:compact_sublevel}
            Let $(\mathcal{X},d)$ be a complete metric space in which closed bounded sets are compact.
            Given $1<p<\infty$ and $\frac{1}{p}<\alpha < 1$ and an arbitrary point $\bar{x}\in \mathcal{X}$, the map $\gamma \mapsto \Psi (\gamma) \coloneqq d (\gamma_0, \bar{x} ) +  | \gamma |_{W^{\alpha,p}}$ from $\Gamma_T \to [0,+\infty]$ has compact sublevels in $\Gamma_T$. 
        \end{lemma}
        \begin{proof}
        Let $c \in \mathbb{R}$ be a positive constant. 
        Consider all curves in $\mathcal{A} \coloneqq \left\{ \gamma \in \Gamma_T \, : \,    \Psi (\gamma) \leq c \right\}$. The bound $\Psi (\gamma) \leq c$ firstly implies 
            \begin{equation}\label{eq:Psi_c_bound0}
                \sup_{\gamma \in \mathcal{A}}  d (\gamma_0, \bar{x} )  < + \infty, 
            \end{equation}
        and secondly, we have $| \gamma |_{W^{\alpha,p}} < c$.
        From the latter and under the conditions $1<p<\infty$ and $\frac{1}{p}<\alpha < 1$, we can derive a precise estimate of the H\"{o}lder regularity of the curve using an application of the Garsia--Rodemich--Rumsey inequality, \cref{thm:FractionalSobolev-Holder}, which implies the existence of a constant $\bar{c}(\alpha,p)$ such that for all $0 \leq s < t \leq T$,
        \begin{align}
            d(\gamma_t,\gamma_s)
            & \leq \bar{c}(\alpha,p) |t-s|^{\alpha-\frac1p} \left( \iint_{[s,t]^2} \frac{d(\gamma_u, \gamma_v)^p}{|u-v|^{1+\alpha p}} \d u \d v \right)^{1/p} \\
            & \leq \bar{c}(\alpha,p) |t-s|^{\alpha-\frac1p} \left( \iint_{[0,T]^2} \frac{d(\gamma_u , \gamma_v)^p}{|u-v|^{1+\alpha p}} \d u \d v \right)^{1/p} \\
            & \leq \bar{c}(\alpha,p)  c |t-s|^{\alpha-\frac1p}. 
        \end{align}
        As a result, the modulus of continuity \eqref{eq:modulus_of_continuity} of all curves $\gamma \in \mathcal{A}$ can be estimated by $w_{\delta} (\gamma  ) \leq \bar{c}(\alpha,p) c \delta^{\alpha-\frac1p}$, which tends to zero as $\delta \to 0$ since $ \alpha - \frac1p >0$. In particular, we have
        \begin{equation} \label{eq:Psi_c_equicontinuity}
            \lim_{\delta \to 0} \sup_{\gamma \in \mathcal{A} } w_{\delta} (\gamma) = 0. 
        \end{equation}
        By Arzel\`a-Ascoli \cref{thm:Arzela-Ascoli}, the conditions \eqref{eq:Psi_c_bound0} and \eqref{eq:Psi_c_equicontinuity} ensure that $\mathcal{A}$ is relatively compact in $\Gamma_T$.
        To show that $\mathcal{A}$ is in fact compact, take a sequence $(\gamma^n) \subset \mathcal{A}$.
        By relative compactness of $\mathcal{A}$, we know that there is a convergent subsequence $\gamma^{n_k} \to \gamma$ in $C$.
        By continuity of $\gamma \mapsto d (\gamma_0, \bar{x} )$ and lower semi-continuity of $\gamma \mapsto |\gamma|_{W^{\alpha,p}}$ (\cref{prop:l.s.c.functional}), we have 
        \begin{equation}
            \Psi (\gamma) \leq \liminf_{k \to \infty} \Psi (\gamma_{n_k}) \leq c,
        \end{equation}
        which means that the limit point $\gamma$ also lies in $\mathcal{A}$ and hence the set $\mathcal{A}$ is compact in $\Gamma_T$. 
    \end{proof}

\subsection{Sobolev paths on metric spaces}\label{subsec:W1p}
    The Sobolev space $W^{1,p}$ we use in this paper coincides with the set of $p$-absolutely continuous curves when $1< p<\infty$, as discussed briefly below. 

    \begin{definition}[{$p$-absolute continuity}] 
        Given $1\leq p<\infty$, a function $X:[0,T] \to \mathcal{X}$ is called $p$-absolutely continuous if there exists $m\in L^p([0,T])$ such that
        $$
        d(X_s,X_t) \leq \int_s^t m(r) \d r, \qquad \forall t,s \in [0,T]. 
        $$
        $AC^p([0,T];\mathcal{X})$ denotes the set of all $p$-absolutely continuous curves. 
    \end{definition}

    We know that the metric derivative $|\dot{X}_t|$ of any $X \in AC^1([0,T];\mathcal{X})$ exists a.e. $t\in [0,T]$ \cite{AGS2008GFs}. For any $X\in C([0,T];\mathcal{X})$ and any $[s,t] \subset [0,T]$, we set:
    \begin{equation}
        | X |_{AC^{p};[s,t]} \coloneqq
        \begin{cases}
            \left(\int_{s}^{t} |\dot{X}_r|^p \d r \right)^{1/p}  \quad & \text{ if } X \in AC^1([0,T];\mathcal{X}),\\
            +\infty  & \text{ if } X \in C([0,T];\mathcal{X}) \setminus AC^1([0,T];\mathcal{X}),
        \end{cases}
    \end{equation}
    and we write $| X |_{AC^{p}} \coloneqq | X |_{AC^{p};[0,T]}$ as usual.

    \begin{definition}[{Sobolev space $W^{1, p}$}]\label{def:W1p}
        Given $1<p<\infty$, the Sobolev regularity of a function $X: [0,T] \to \mathcal{X}$ over $[t,s] \subset [0,T]$ is defined by
        \begin{equation}\label{eq:def_W1p}
            | X |_{W^{1,p};[s,t]} \coloneqq \left( \sup_{t_i \in \mathcal{D}([s,t])} \sum_{i} \frac{d(X_{t_i}, X_{t_{i+1}})^p}{|t_{i+1} - t_{i}|^{p-1}} \right)^{1/p}.
        \end{equation}
        The Sobolev space $W^{1, p}([0,T]; \mathcal{X})$ is the space of functions $X$ such that
        \begin{equation}
            | X |_{W^{1,p}} \coloneqq | X |_{W^{1,p};[0,T]} < \infty.
        \end{equation}
    \end{definition}

    On $\mathbb{R}^\d$, it is well-known that the function spaces mentioned above coincide when $p>1$ (see e.g. \cite[Proposition 1.45]{FrizVictoir2010book}). This result also holds in the general metric setting. For a proof in the case $p=2$, see \cite[Theorem 10.2]{ABS2021} or \cite[Proposition 2.3]{BoubelJuillet2025}, which can be generalized to any $p>1$ and is stated below:

    \begin{theorem}
        Given $1<p<\infty$, we have
        \begin{equation}
            AC^p([0,T];\mathcal{X}) = W^{1,p}([0,T];\mathcal{X}).
        \end{equation}
        Furthermore, $| \cdot |_{W^{1,p}}$ and $| \cdot |_{AC^p} $  are equal on the set of continuous paths, i.e.,
        \begin{equation}
            | X |_{AC^{p};[s,t]} = | X |_{W^{1,p};[s,t]},
        \end{equation}
        for all $X\in C([0,T];\mathcal{X}) $ and $0\leq s<t\leq T$.   
    \end{theorem}

    Accordingly, we mainly focus on the definition $|\cdot|_{W^{1,p}}$ as its properties are easier to observe. First, since it is defined as the supremum of a family of continuous functions, it is a lower semi-continuous map, as mentioned in  \cite[Theorem 10.2]{ABS2021}:

    \begin{lemma}[Lower semi-continuity of $\gamma \mapsto|\gamma |_{W^{1,p}}$]
             Let $1 < p < \infty$. The map $\gamma \mapsto | \gamma |_{W^{1,p}}$ from $ \Gamma_T \to [0,\infty]$ is lower semi-continuous with respect to pointwise convergence, and in particular, with respect to uniform convergence. 
    \end{lemma}

    Next, by the very definition of $|X|_{W^{1,p}}$, we have for all $0\leq s<t\leq T$,
    \begin{equation}\label{eq:distance_estimate_W1p}
    d(X_s,X_t) \leq |t-s|^{1-\frac{1}{p}}|X|_{W^{1,p};[s,t]},
    \end{equation}
    which implies the embeddings \cite[Theorem 1.47]{FrizVictoir2010book}
    \begin{equation}
            W^{1, p} \subset C^{1 - \frac{1}{p}\textrm{-}\mathrm{H\ddot{o}l}} \quad \textrm{and} \quad W^{1, p} \subset C^{1\textrm{-}\mathrm{var}}.
        \end{equation}
    Combining \eqref{eq:distance_estimate_W1p} with the Arzelà-Ascoli \cref{thm:Arzela-Ascoli}, one simply concludes:

    \begin{lemma}[Compact sublevels of $\gamma \mapsto d (\gamma_0, \bar{x} ) +  | \gamma |_{W^{1,p}}$]
            Let $(\mathcal{X},d)$ be a complete metric space in which closed bounded sets are compact.
            Given $1<p<\infty$ and an arbitrary point $\bar{x}\in \mathcal{X}$, the map $\gamma \mapsto \Psi (\gamma) \coloneqq d (\gamma_0, \bar{x} ) +  | \gamma |_{W^{1,p}}$ from $\Gamma_T \to [0,+\infty]$ has compact sublevels in $\Gamma_T$. 
    \end{lemma}

\subsection{Geodesics on metric spaces} 

    \begin{definition}[{Space $Geo([0,1];\mathcal{X})$}]
        A curve $X : [0,1] \to \mathcal{X}$ is called a constant-speed geodesic joining two points $x,y \in \mathcal{X}$ if $X_0=x, X_1=y,$ and 
        $$
        d(X_s,X_t) = |t-s| d(X_0,X_1), \quad \forall t,s\in [0,1].
        $$
        $Geo([0,1];\mathcal{X})$ denotes the set of all constant-speed geodesics.
    \end{definition}
    We recall some notions in metric geometry that we will need.
    \begin{definition}[Length-metric space]
        A metric space $(\mathcal{X},d)$ is called a length space if the distance between two points is equal to the infimum of the length of rectifiable curves joining them.
    \end{definition}

    \begin{definition}[Geodesic-metric space]
        A metric space $(\mathcal{X},d)$ is called a geodesic space if every two points in $\mathcal{X}$ can be joined by a (not necessarily unique) constant-speed geodesic.    
    \end{definition}

    Obviously, every geodesic space is a length space. However, a length space need not be a geodesic space in general unless we have additional structures on the metric space, for example:
    \begin{theorem}[Hopf--Rinow]\label{Hopf-Rinow}
            Let $(\mathcal{X},d)$ be a length metric space. If in addition, it is complete and locally compact, then
            \begin{enumerate}[label=(\roman*), font=\normalfont]
                \item every closed bounded subset of $\mathcal{X}$ is compact;
                \item $\mathcal{X}$ is a geodesic space.
            \end{enumerate}
    \end{theorem}

\subsection{Piecewise geodesic approximation on geodesic spaces}
    Here we give a practical lemma, in which we compute $b^{\alpha,p}$-norm of a piecewise geodesic curve in a geodesic space. 
    This will greatly help us in later computations.
    Let $(\mathcal{X},d)$ be a geodesic space. Fix $n \in \mathbb{N}_0$ and let $D_n$ be the dyadic dissection of $[0,1]$,
    \begin{gather}
             D_n \coloneqq \{ 0=t_0^{(n)} < t_1^{(n)} < \cdots < t_{N_n}^{(n)} = 1 \}, \\
             t_i^{(n)} \coloneqq \frac{i}{2^n} \quad i \in \left\{ 0,1, \cdots, 2^n \eqqcolon N_n \right\}. 
        \end{gather}
    Given $N_n+1$ points in $\mathcal{X}$ denoted by
    \begin{equation}\label{eq:points_piecewise_geodesic}
        \{x_0,x_1,\cdots,x_{N_n}\},
    \end{equation}
    let $X^n : [0,1] \to \mathcal{X}$ be a continuous curve that passes the points, i.e., 
    \begin{equation}\label{eq:def_piecewise_geodesic}
    X^n \Big( t_i^{(n)} \Big) = x_i \quad \textrm{for all} \quad i \in \left\{ 0,1, \cdots, N_n \right\}
    \end{equation}
    and connects in between by a constant-speed geodesic, i.e., the speed over each time segment is equal to $2^n d(x_i,x_{i+1})$.
    In other words, $X^n$ is a piecewise geodesic connecting the points \eqref{eq:points_piecewise_geodesic}. 
    In the next lemma, we compute $| X^n |_{b^{\alpha,p}}$. This quantity turns out to be independent of the choice of geodesics when geodesics are not unique.

    \begin{lemma}\label{lemma:balphap_Xn}
        Let $(\mathcal{X},d)$ be a geodesic space.
        Fix $n \in \mathbb{N}_0$ and let $D_n$ be the dyadic dissection of $[0,1]$. Given a set of points $\{x_0,x_1,\cdots,x_{2^n}\}$ in $\mathcal{X}$, let $X^n$ be a piece-wise geodesic curve on $D_n$ connecting them, as defined above in \eqref{eq:def_piecewise_geodesic}. Then we have 
        \begin{equation}\label{eq:balphap_Xn}
             | X^n |_{b^{\alpha,p}}^p =  \sum_{m=0}^{n} 2^{m(\alpha p -1)} \hspace{-20pt} \sum_{\substack{i = k \frac{2^n}{2^m} \\ k \in \{0,1,\cdots, 2^m-1\}}} \hspace{-20pt} d (x_{i},x_{i+\frac{2^n}{2^m}})^p + \frac{2^{n (\alpha p - 1 )}}{2^{(p - \alpha p)}-1}  \sum_{i=0}^{2^n-1} d(x_i,x_{i+1})^p. 
        \end{equation}
    \end{lemma}
    
    \begin{figure}
        \centering
        \begin{tikzpicture}[scale=1.5]
        \draw [->] (0+0.5,0) -- (8-0.5,0);
        \draw (8-0.5,0) node[anchor=north] {\scriptsize time};
        \draw (2,-0.04) -- (2,0.04);
        \draw (2,-0.05) node[anchor=north] {\scriptsize $t_i^{(n)}$};
        \draw (2,-0.4) node[anchor=north] {\color{red}\scriptsize $t_k^{(m)}$};
        \draw (3,-0.04) -- (3,0.04);
        \draw (3,-0.05) node[anchor=north] {\scriptsize $t_{i+1}^{(n)}$};
        \draw (4,-0.04) -- (4,0.04);
        \draw (5,-0.04) -- (5,0.04);
        \draw (6,-0.04) -- (6,0.04);
        \draw (6,-0.05) node[anchor=north] {\scriptsize $t_{i'}^{(n)}$};
        \draw (6,-0.4) node[anchor=north] {\color{red}\scriptsize $t_{k+1}^{(m)}$};
        \draw [<->](4,-0.2) -- (5,-0.2);
        \draw (4.5,-0.2) node[anchor=north] {\scriptsize $\Delta t_n \coloneqq \frac{1}{2^n}$};
        \draw [red][<->](2,-0.9) -- (6,-0.9);
        \draw (4,-0.9) node[anchor=north] {\color{red} \scriptsize $\Delta t_m \coloneqq \frac{1}{2^m}$};
        \draw [black][-]
        (1,0.9)--
        (2,0.6)--
        (3,1.5)--
        (4,0.4)--
        (5,0.8)--
        (6,2.2)--
        (7,2.0)
        ;
        \filldraw [red] (2,0.6) circle (0.8pt);
        \filldraw [black] (3,1.5) circle (0.8pt);
        \filldraw [black] (4,0.4) circle (0.8pt);
        \filldraw [black] (5,0.8) circle (0.8pt);
        \filldraw [red] (6,2.2) circle (0.8pt);
        \draw [red][dashed]
        (2,0.6)--
        (6,2.2);
        \draw (2,0.65) node[anchor=south] {\scriptsize $x_i$};
        \draw (3,1.55) node[anchor=south] {\scriptsize $x_{i+1}$};
        \draw [red][dotted] (2,0)--(2,0.6);
        \draw [red][dotted] (6,0)--(6,2.2);
        \draw (1,2) node[anchor=south] {\scriptsize Case ${\color{red} m} \leq n$.};
        \end{tikzpicture}
        
        \vspace{20pt}
        
        \begin{tikzpicture}[scale=1.6]
        \draw [->] (0+0.5,0) -- (8-0.5,0);
        \draw (8-0.5,0) node[anchor=north] {\scriptsize time};
        \draw (2,-0.04) -- (2,0.04);
        \draw (2,-0.05) node[anchor=north] {\scriptsize $t_i^{(n)}$};
        \draw [red](2.25,-0.04) -- (2.25,0.04);
        \draw [red] (2.50,-0.04) -- (2.50,0.04);
        \draw (2.25-0.06,-0.35) node[anchor=north] {\color{red}\scriptsize $t_k^{(m)}$};
        \draw (2.50+0.06,-0.35) node[anchor=north] {\color{red}\scriptsize $t_{k+1}^{(m)}$};
        \draw [red][dotted][-] (2.25,0.6+0.225) -- (2.25,-0.4);
        \draw [red][dotted][-] (2.50,0.6+0.225+0.225) -- (2.50,-0.4);
        \draw [red][dotted][-] (2.25,-0.75) -- (2.25,-0.9);
        \draw [red][dotted][-] (2.50,-0.75) -- (2.50,-0.9);
        \draw (3,-0.04) -- (3,0.04);
        \draw (3,-0.05) node[anchor=north] {\scriptsize $t_{i+1}^{(n)}$};
        \draw (4,-0.04) -- (4,0.04);
        \draw (5,-0.04) -- (5,0.04);
        \draw (6,-0.04) -- (6,0.04);
        \draw [<->](4,-0.2) -- (5,-0.2);
        \draw (4.5,-0.2) node[anchor=north] {\scriptsize $\Delta t_n \coloneqq \frac{1}{2^n}$};
        \draw [red][<->](2.25,-0.9) -- (2.50,-0.9);
        \draw (2.375,-0.9) node[anchor=north] {\color{red} \scriptsize $\Delta t_m \coloneqq \frac{1}{2^m}$};
        \draw [black][-]
        (1,0.9)--
        (2,0.6)--
        (3,1.5)--
        (4,0.4)--
        (5,0.8)--
        (6,2.2)--
        (7,2.0)
        ;
        \filldraw [black] (2,0.6) circle (0.8pt);
        \filldraw [black] (3,1.5) circle (0.8pt);
        \filldraw [black] (4,0.4) circle (0.8pt);
        \filldraw [black] (5,0.8) circle (0.8pt);
        \filldraw [black] (6,2.2) circle (0.8pt);
        \filldraw [red] (2.25,0.6+0.225) circle (0.8pt);
        \filldraw [red] (2.50,0.6+0.225+0.225) circle (0.8pt);
        \draw (2,0.65) node[anchor=south] {\scriptsize $x_i$};
        \draw (3,1.55) node[anchor=south] {\scriptsize $x_{i+1}$};
        \draw (1,2) node[anchor=south] {\scriptsize Case ${\color{red} m} > n$.};
        \end{tikzpicture}
        \captionsetup{font=footnotesize}
        \caption{Two cases in the computation of $b^{\alpha,p}$-regularity of a piecewise geodesic curve $X^n$ in the proof of \cref{lemma:balphap_Xn}. Note that $n$ is fixed. \textbf{Top:} $D_m$ is a coarser partition than $D_n$. \textbf{Bottom:} $D_m$ is a finer partition than $D_n$.}
        \label{fig:balphap_Xn}
    \end{figure}
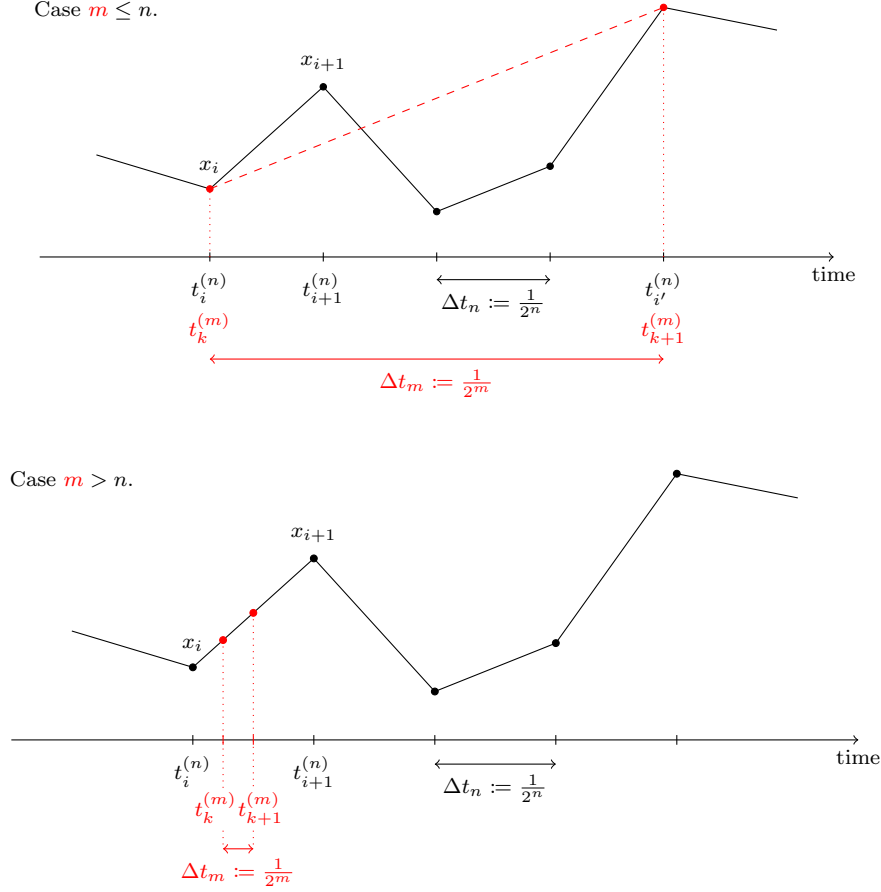

    \begin{proof}
        According to definition \eqref{eq:Walphap_sum}, 
        $$
             | X^n |_{b^{\alpha,p}}^p \coloneqq \sum_{m=0}^{\infty} 2^{m(\alpha p -1)} \sum_{k=0}^{2^m-1} d \big( X^n({t_{k}^{(m)}}),X^n({t_{k+1}^{(m)}}) \big)^p ,
        $$
        where $t_k^{(m)} \coloneqq \frac{k}{2^m} $ and for brevity, we shall denote by $\Delta t_n \coloneqq  \frac{1}{2^n}$ and $\Delta t_m \coloneqq   \frac{1}{2^m}$ for all $m \in \mathbb{N}_0$.
        We split the outer sum into two parts and compute them separately.
        \\
        \textbf{Case $m \leq  n$} (i.e. $D_m$ is a coarser partition than $D_n$). We show that this will sum up to the first term in \eqref{eq:balphap_Xn}. 
        This can be easily done through a re-indexing. Since $D_m \subseteq D_n$, each point $t_k^{(m)}$ of $D_m$ coincides with a point $t_i^{(n)}$ of $D_n$ (as shown in \cref{fig:balphap_Xn} (top)), where the corresponding index $i$ is determined by equating 
        \begin{equation}
        t_k^{(m)} = \frac{k}{2^m}  = \frac{i}{2^n} = t_i^{(n)} \quad \Rightarrow \quad i = k \frac{2^n}{2^m} .
        \end{equation}
        Therefore, we can write
            \begin{align}
                \sum_{m=0}^{n} 2^{m(\alpha p -1)} \sum_{k=0}^{2^m-1} d \big( X^n({t_{k}^{(m)}}), X^n({t_{k+1}^{(m)}}) \big)^p & = \sum_{m=0}^{n} 2^{m(\alpha p -1)} \sum_{k=0}^{2^m-1} d \big( x_{(k)\frac{2^n}{2^m}}, x_{(k+1)\frac{2^n}{2^m}}  \big)^p \\
                & =  \sum_{m=0}^{n} 2^{m(\alpha p -1)} \hspace{-20pt} \sum_{\substack{i = k \frac{2^n}{2^m} \\ k \in \{0,1,\cdots, 2^m-1\}}} \hspace{-20pt} d (x_{i}, x_{i+\frac{2^n}{2^m}})^p. \label{eq:balphap_Xn_m=<n}
            \end{align}
        \textbf{Case $m > n$} (i.e. $D_m$ is a finer partition than $D_n$). We show that this will result in the second term in \eqref{eq:balphap_Xn}. Let $t_i^{(n)} \leq t_{k}^{(m)} \leq t_{k+1}^{(m)} \leq t_{i+1}^{(n)}$ (as in \cref{fig:balphap_Xn} (bottom)). 
        Since the points are connected by geodesics, we know that
        $$
        d \big( X^n({t_{k}^{(m)}}), X^n ({t_{k+1}^{(m)}})  \big) = \frac{\Delta t_m}{\Delta t_n} \,  d (x_{i}, x_{i+1}).
        $$
        Therefore, the remaining part of the sum can be written as 
        \begin{align}
            \sum_{m>n} 2^{m(\alpha p -1)} \sum_{k=0}^{2^m-1} d \big(X^n({t_{k}^{(m)}}),X^n({t_{k+1}^{(m)}})\big)^p
            & = \sum_{m>n} 2^{m(\alpha p -1)}  \sum_{i=0}^{2^n-1} \frac{\Delta t_n}{\Delta t_m} \left( \frac{\Delta t_m}{\Delta t_n} d (x_{i},x_{i+1}) \right)^p \\
            & = 2^{n(p-1)} \sum_{m>n} 2^{- m (p - \alpha p)} \sum_{i=0}^{2^n-1} d (x_{i},x_{i+1})^p \\
            & = \frac{2^{n (\alpha p - 1 )}}{2^{(p-\alpha p) } -1 }  \sum_{i=0}^{2^n-1} d(x_i,x_{i+1})^p \label{eq:balphap_Xn_m>n}
        \end{align}
        Adding \eqref{eq:balphap_Xn_m=<n} and \eqref{eq:balphap_Xn_m>n}, we obtain the result. 
    \end{proof}

\subsection{Path measures and sample path regularity}\label{subsec:path_measures}
    We still consider the metric space $(\mathcal{X},d)$ and, in addition, we let $\mathcal{B}(\mathcal{X})$ denote the $\sigma$-algebra of Borel sets of $\mathcal{X}$ (generated by $d$-open balls). Let $I\coloneqq [0,T]\subset \mathbb{R}$ be a time interval. 
    Here, we first recall two fundamental path spaces, following \cite{RogersWilliams2000_V1,StroockVaradhan2006,FrizVictoir2010book}, and then discuss the measurability of the norms. Recall that $\mathcal{X}^I$ is the space of all functions from $I$ to $\mathcal{X}$. For $t \in I$, define
    \begin{align}
        e_t: \mathcal{X}^I \to \mathcal{X}, \qquad e_t (\gamma) \coloneqq \gamma_t
    \end{align}
    to be the \emph{evaluation map}. Similarly, for $J \subset I$, define
    \begin{align}
        e_J: \mathcal{X}^I \to \mathcal{X}^J, \qquad
        e_J (\gamma) \coloneqq \gamma\big|_{J}
    \end{align}
    to be the \emph{restriction map}. Two relevant path spaces for our work are as follows:
    
    \begin{itemize}
        \item $\big(\mathcal{X}^I, \mathcal{B} (\mathcal{X})^I\big)$. This is the space of all functions from $I$ to $\mathcal{X}$ endowed with the product $\sigma$-algebra $\mathcal{B}(\mathcal{X})^I$, which is defined as the smallest $\sigma$-algebra such that all evaluation maps $e_t:\mathcal{X}^I\to \mathcal{X}$ with $t \in I$ are measurable. Equivalently, one can show that $\mathcal{B}(\mathcal{X})^I$ is generated by the collections of cylindrical sets. In short,
        \begin{align}
            \mathcal{B}(\mathcal{X})^I
             \coloneqq & \sigma \big(e_t : \mathcal{X}^I\to \mathcal{X} \, \big| \ t \in I \big)  \\
             = &  \sigma \big( e_J : \mathcal{X}^I\to \mathcal{X}^J \big|  J \subset I, J \text{ finite} \big).
        \end{align}
    \item $\big(C(I;\mathcal{X}),\mathcal{C} \big)$. This is the space of continuous functions from $I$ to $\mathcal{X}$ endowed with the $\sigma$-algebra $\mathcal{C}$, which is again defined as the smallest $\sigma$-algebra such that all evaluation maps $e_t: C(I;\mathcal{X}) \to \mathcal{X}$ with $t \in I$ are measurable, and equivalently, it is also generated by the collections of cylindrical sets. On the other hand, here we have the supremum distance $d_\infty$, as defined in \eqref{eq:sup_distance}, on the space $C(I;\mathcal{X})$ and there is a natural $\sigma$-algebra of Borel sets (generated by $d_\infty$-open balls). It turns out that this one also coincides with the two just mentioned. In short,
    \begin{align}
        \mathcal{C}
         \coloneqq & \sigma \big(e_t : C(I;\mathcal{X})\to \mathcal{X} \, \big| \, t \in I \big)  \\
         = &  \sigma \big( e_J : C(I;\mathcal{X})\to \mathcal{X}^J \big|  J \subset I, J \text{ finite}\big) \\
        = & \sigma \big( d_\infty\text{-topology} \big).
    \end{align}
    \end{itemize}
    Once the measure structure is fixed, we can discuss probability measures $\pi$ on $(\mathcal{X}^I, \mathcal{B} (\mathcal{X})^I )$ or $(C(I;\mathcal{X}),\mathcal{C} )$, which are referred to as \emph{path measures}. 
    By path measures, we only mean probability measures. 
    Recall that given a path measure $\pi$ on $(\mathcal{X}^I,\mathcal{B}(\mathcal{X})^I)$, it is not possible to answer the following natural question
    $$
    \pi \Big( \big\{ \gamma \in \mathcal{X}^I : \gamma \in C(I;\mathcal{X}) \big\} \Big) \overset{?}{=} 1.
    $$
    as $C(I;\mathcal{X})$ is not a measurable set in $\mathcal{B}(\mathcal{X})^I$, because it is not of the form of $\sigma$-cylinder sets.
    For the same reason, we cannot ask whether a path measure $\pi$ on $(\mathcal{X}^I, \mathcal{B} (\mathcal{X})^I )$ has finite H\"{o}lder energy
    $$
    \int_{\mathcal{X}^I} | \gamma |_{\upgamma\textrm{-}\mathrm{H\ddot{o}l}} \d \pi (\gamma) \overset{?}{<} + \infty
    $$
    because $\gamma \mapsto | \gamma |_{\upgamma\textrm{-}\mathrm{H\ddot{o}l}}$ is not measurable as a function from $(\mathcal{X}^I,\mathcal{B}(\mathcal{X})^I) \to (\bar{\mathbb{R}}_+,\mathcal{B}(\bar{\mathbb{R}}_+))$.  Here $\bar{\mathbb{R}}_+ \coloneqq [0,+\infty]$ and the corresponding $\sigma$-algebra $\mathcal{B}(\bar{\mathbb{R}}_+)$ also contains $\{+\infty\}$. Let us summarize this observation:

    \begin{remark}\label{rmk:measurability_Walphap}
        The functions
        \begin{align}
            \gamma  \mapsto w_{\delta} (\gamma), \qquad 
            \gamma  \mapsto | \gamma |_{\upgamma\textrm{-}\mathrm{H\ddot{o}l}}, \qquad 
            \gamma  \mapsto |\gamma|_{W^{\alpha,p}},
            \qquad 
            \gamma  \mapsto |\gamma|_{W^{1,p}}
        \end{align}
        from $(\mathcal{X}^I,\mathcal{B}(\mathcal{X})^I) \to (\bar{\mathbb{R}}_+,\mathcal{B}(\bar{\mathbb{R}}_+))$ are \emph{not} measurable. However, as functions
        from $(C(I;\mathcal{X}),\mathcal{C}) \to (\bar{\mathbb{R}}_+,\mathcal{B}(\bar{\mathbb{R}}_+))$, they are lower semi-continuous, in particular, measurable.
    \end{remark}
    In contrast to the functions above, the function $\gamma \mapsto | \gamma |_{b^{\alpha,p}}$ has the advantage of measurability. Recall that its definition only relies on a countable sum of distances. As a consequence of Beppo Levi's lemma, we can affirm that $\gamma \mapsto | \gamma |^p_{b^{\alpha}}$ is measurable because it is a pointwise limit of a non-decreasing sequence of measurable functions:
    \begin{remark}\label{rmk:measurability_balphap}
        The function $$\gamma \mapsto |\gamma|_{b^{\alpha,p}}$$ from $(\mathcal{X}^I,\mathcal{B}(\mathcal{X})^I) \to (\bar{\mathbb{R}}_+,\mathcal{B}(\bar{\mathbb{R}}_+))$ is measurable. It is also measurable as a function from $(C(I;\mathcal{X}),\mathcal{C}) \to (\bar{\mathbb{R}}_+,\mathcal{B}(\bar{\mathbb{R}}_+))$.
    \end{remark}

\subsection{Probability measures, narrow convergence, and tightness}\label{subsec:probability_measures}
    In this section, we consider $(\mathcal{X},d)$ to be a complete separable metric space. 
    We let $\mathcal{B}(\mathcal{X})$ be the $\sigma$-algebra of Borel sets of $\mathcal{X}$ (generated by open balls in $\mathcal{X}$) and we let $P(\mathcal{X})$ be the set of all Borel probability measures on $(\mathcal{X}, \mathcal{B}(\mathcal{X}))$.
    The space of continuous and bounded functions $\varphi: \mathcal{X} \to \mathbb{R}$ is denoted by $ C_b(\mathcal{X}) \coloneqq C_b(\mathcal{X};\mathbb{R})$ and it is equipped with the supremum-norm
    $$
    \lVert \varphi \rVert_{\infty } \coloneqq \sup_{x \in \mathcal{X}} |\varphi(x)|.  
    $$
    The narrow topology on $P(\mathcal{X})$ is generated using $C_b(\mathcal{X})$.

    \begin{definition}[Narrow convergence]\label{def:narrow_convergence}
        A sequence $(\mu_n) \subset P(\mathcal{X})$ narrowly converges to $\mu \in P(\mathcal{X})$ if
        \begin{equation}
            \lim_{n \to \infty }\int_{\mathcal{X}} \varphi \d \mu_n = \int_{\mathcal{X}} \varphi \d \mu
        \end{equation}
        for every $\varphi \in C_b(\mathcal{X})$. 
    \end{definition}
    Narrow convergence can be characterized by a subset of $C_b(\mathcal{X})$, namely, the space of Lipschitz bounded functions $\mathrm{Lip}_b(\mathcal{X})$ (see e.g. \cite[Lemma 8.12]{ABS2021}). We also recall that if $(\mu_n) \subset P(\mathcal{X})$ narrowly converges to $\mu \in P(\mathcal{X})$ and $\varphi: \mathcal{X} \to [0,+\infty]$ is a lower semi-continuous function, then
    \begin{equation}\label{eq:narrow_conv_lsc_func}
        \liminf_{n \to \infty } \int_{\mathcal{X}} \varphi \d \mu_n \geq \int_{\mathcal{X}} \varphi \d \mu. 
    \end{equation}

    \begin{definition}[Tightness]\label{def:tightness}
        A family of measures $ \mathcal{K} \subset P(\mathcal{X})$ is said to be tight if for any $ \varepsilon>0$, there exists a compact set $K_\varepsilon \subset \mathcal{X}$ such that for all $\mu \in \mathcal{K}$, we have $\mu (K_\varepsilon^{\complement}) \leq \varepsilon$, or in other words, 
        \begin{equation}\label{eq:tightness}
            \sup_{\mu \in \mathcal{K}} \mu (K_\varepsilon^{\complement}) \leq \varepsilon.
        \end{equation}
    \end{definition}

    We recall the following well-known integral criterion for tightness:

    \begin{lemma}[An integral tightness criterion]\label{lemma:tightness_criterion}
         A family of measures $ \mathcal{K} \subset P(\mathcal{X})$ is tight if and only if there exists a function 
         $\Psi: \mathcal{X} \to [0, + \infty]$ such that
        \begin{enumerate}[font=\normalfont]
            \item its sublevels $\lambda_c (\Psi) \coloneqq \{ |\Psi| \leq c \} \subset \mathcal{X}$ are compact for any $c \geq 0$;
            \item it satisfies the bound
            \begin{equation}\label{eq:tightness_criterion_bound}
                \sup_{\mu \in \mathcal{K}} \int_{\mathcal{X}}  \Psi(x) \d \mu (x) < + \infty.
            \end{equation}
        \end{enumerate}
    \end{lemma}
    
    \begin{theorem}[Prokhorov]\label{thm:Prokhorov}
        A family of measures $\mathcal{K} \subset P (\mathcal{X})$ is tight if and only if it is relatively compact with respect to the narrow topology of $P(\mathcal{X})$. 
    \end{theorem}

\subsection{Tightness conditions for path measures}\label{subsec:tightness_conditions}
    Given the results in the previous sections, we now list some tightness conditions for path measures on $\big(C([0,T];\mathcal{X}),\mathcal{C} \big)$. In this paper, only 
    \cref{prop:tightness_balphap} is used. 
    
    \begin{corollary}[A tightness condition via $W^{\alpha,p}$]\label{prop:tightness_Walphap}
         Let $(\mathcal{X},d)$ be a complete metric space in which closed bounded sets are compact.
         Let the family of measures $ \mathcal{K} \subset P(C([0,T];\mathcal{X}))$ satisfy
         \begin{equation}\label{eq:tightness_Walphap}
             \sup_{\pi \in \mathcal{K}} \int_{\Gamma_T} \Big(  d(\gamma_0, \bar{x}) +  | \gamma |_{W^{\alpha,p}} \Big) \d \pi (\gamma) < +\infty
         \end{equation}
         for some $\bar{x} \in \mathcal{X}$ and $1<p<\infty$ and $\frac{1}{p}<\alpha < 1$. Then $\mathcal{K}$ is tight in $P(C([0,T];\mathcal{X}))$.  
    \end{corollary}
    \begin{proof}
        We first emphasize that the functional $\Psi: \Gamma_T \to [0,+\infty]$ defined as 
        $ \gamma \mapsto \Psi (\gamma) \coloneqq d (\gamma_0, \bar{x} ) +  | \gamma |_{W^{\alpha,p}}.
        $
        is lower-semi continuous (by \cref{prop:l.s.c.functional}), in particular, it is measurable and thus the integral makes sense.  
        By \cref{lemma:compact_sublevel}, we know that the functional $\Psi$ has compact sublevels in $\Gamma_T$ under the condition $1<p<\infty$ and $\frac{1}{p}<\alpha < 1$.
        This, together with the bound above, means that $\Psi$ satisfies both requirements of the integral criterion for tightness  (\cref{lemma:tightness_criterion}). Consequently, the family $\mathcal{K}$ is tight in $P(C([0,T];\mathcal{X}))$.
    \end{proof}

    The following result immediately follows from the previous one and the equivalence of $| \cdot|_{W^{\alpha,p}}$ and $| \cdot |_{b^{\alpha,p}} $ (\cref{thm:Walphap_balphap}) on the set of continuous paths, over which we take the integrals here. 

    \begin{corollary}[A tightness condition via $b^{\alpha,p}$]\label{prop:tightness_balphap}
         Let $(\mathcal{X},d)$ be a complete metric space in which closed bounded sets are compact.
         Let the family of measures $ \mathcal{K} \subset P(C([0,1];\mathcal{X}))$ satisfy
         \begin{equation}\label{eq:tightness_balphap}
             \sup_{\pi \in \mathcal{K}} \int_{\Gamma_1} \Big(  d(\gamma_0, \bar{x}) +  | \gamma |_{b^{\alpha,p}} \Big) \d \pi (\gamma) < +\infty
         \end{equation}
         for some $\bar{x} \in \mathcal{X}$ and $1<p<\infty$ and $\frac{1}{p}<\alpha < 1$. Then $\mathcal{K}$ is tight in $P(C([0,1];\mathcal{X}))$.  
    \end{corollary}

    Let's observe that a suitable H\"{o}lder regularity can guarantee the finiteness of the second part of \eqref{eq:tightness_Walphap}. We thus arrive at the well-known \emph{Kolmogorov--Lamperti} tightness result (cf.\,also\,\cite[Corollary A.11]{FrizVictoir2010book}): 
    \begin{corollary}[A tightness condition via H\"{o}lder estimates] \label{crl:tightness_Holder}
         Let $(\mathcal{X},d)$ be a complete metric space in which closed bounded sets are compact. Let the family of measures $ \mathcal{K} \subset P(C([0,T];\mathcal{X}))$ satisfy
         \begin{equation}
        \sup_{\pi \in \mathcal{K}} \int_{\Gamma_T}  d(\gamma_0, \bar{x}) \d \pi (\gamma) < +\infty
        \end{equation}
        for some $\bar{x} \in \mathcal{X}$ and 
        \begin{equation}\label{eq:tightness_Holder}
            \sup_{\pi \in \mathcal{K}} \int_{\Gamma_T} d(\gamma_t, \gamma_s )^p \d \pi (\gamma) \leq c |t-s|^{p \upgamma} \quad\qquad \forall t,s \in [0,T],
        \end{equation}
        for some constant $c$ and $1<p<\infty$ and $\frac{1}{p}<\upgamma \leq 1$. Then $\mathcal{K}$ is tight in $P(C([0,T];\mathcal{X}))$.
    \end{corollary}
    \begin{proof}
        Choose $\alpha$ such that $\frac{1}{p} < \alpha < \upgamma$.
        For any $\pi \in \mathcal{K}$,  one can easily confirm that $\int_{\Gamma_T}  | \gamma |^p_{W^{\alpha,p}}  \d \pi$ is bounded by a constant independent of $\pi$ (by applying Tonelli’s theorem, using the condition \eqref{eq:tightness_Holder}, and finally recalling \cref{rmk:Holder-FractionalSobolev}). The claim then follows from \cref{prop:tightness_Walphap}.
    \end{proof}
    
    Our next observation is that if \eqref{eq:tightness_balphap} is to be used, it is enough to have the H\"{o}lder regularity \eqref{eq:tightness_Holder} only on the dyadic time points. This leads to a \emph{relaxed} version of the Kolmogorov--Lamperti tightness condition above and is again reminiscent of \cref{thm:disc_char_Holder}.
    
    \begin{corollary}[A tightness condition via H\"{o}lder estimates on a countable set]\label{crl:tightness_Holder_discrete}
         Let $(\mathcal{X},d)$ be a complete metric space in which closed bounded sets are compact. Let the family of measures $ \mathcal{K} \subset P(C([0,1];\mathcal{X}))$ satisfy
         \begin{equation}
        \sup_{\pi \in \mathcal{K}} \int_{\Gamma_1}  d(\gamma_0, \bar{x}) \d \pi (\gamma) < +\infty
        \end{equation}
        for some $\bar{x} \in \mathcal{X}$ and 
        \begin{equation}\label{eq:tightness_Holder_discrete}
         \sup_{\pi \in \mathcal{K}} \int_{\Gamma_1}  d (\gamma_{t_{k}^{(m)}},\gamma_{t_{k+1}^{(m)}})^p \d \pi (\gamma) \leq \tilde{c} |\Delta t_m |^{p\upgamma} \quad\qquad \forall m \in \mathbb{N}_0, \, k \in \{0,1,\cdots, 2^m-1\},
    \end{equation}
    for some constant $\tilde{c}$ and $1<p<\infty$ and $\frac{1}{p}<\upgamma \leq 1$, where $t_k^{(m)} \coloneqq \frac{k}{2^m} $ and $\Delta t_m \coloneqq \frac{1}{2^m}$. Then $\mathcal{K}$ is tight in $P(C([0,1];\mathcal{X}))$.
    \end{corollary}

    \begin{proof}
        Again choose $\alpha$ such that $\frac{1}{p} < \alpha < \upgamma$.
        For any $\pi \in \mathcal{K}$,  one can similarly confirm that $\int_{\Gamma_T}  | \gamma |^p_{b^{\alpha,p}}  \d \pi$ is bounded by a constant independent of $\pi$ (here first by applying  Beppo Levi’s lemma, second using condition \eqref{eq:tightness_Holder_discrete}, and finally recalling \cref{rmk:Holder-FractionalSobolev}). The claim then follows from \cref{prop:tightness_balphap}.
    \end{proof}

\subsection{Set of lifts on continuous paths}\label{subsec:set_of_lifts}
    In this section, let $(\mathcal{X}, d)$ be a complete separable metric space. 
    Given a family of Borel probability measures $(\mu_t)_{t \in I} \subset P(\mathcal{X})$ indexed by $t \in I \coloneqq [0,T] \subset \mathbb{R}$, we define the following (possibly empty) set 
    \begin{equation}
        \mathrm{Lift}(\mu_t) \coloneqq \Big\{\pi \in P(C(I;\mathcal{X})) : \quad (e_t)_{\#} \pi  = \mu_t \text{ for all } t \in I \Big\}.
    \end{equation}
    We study some properties of this set.

    \begin{lemma}\label{lemma:lift_property_convex}
        The set $\mathrm{Lift}(\mu_t)$, provided it is non-empty, is convex. 
    \end{lemma}

    \begin{proof}
        Let $\pi_0,\pi_1 \in \mathrm{Lift}(\mu_t)$ and take $\beta \in [0,1]$. It is clear that $\pi_{\beta} \coloneqq (1-\beta) \pi_0 + \beta \pi_1$ is a path measure. For any $t \in I$ and $\phi \in C_b(\mathcal{X})$, it satisfies 
        \begin{equation}
            \int_{\Gamma_T} \phi (\gamma_t) \d \pi_\beta (\gamma)
            = (1-\beta) \int_{\Gamma_T} \phi (\gamma_t) \d \pi_0 (\gamma) + \beta \int_{\Gamma_T} \phi (\gamma_t) \d \pi_1 (\gamma) = \int_{\mathcal{X}} \phi(x) \d \mu_t(x).
        \end{equation}
    \end{proof}

    \begin{lemma}\label{lemma:lift_property_closed}
        The set $\mathrm{Lift}(\mu_t)$, provided it is non-empty, is closed under narrow convergence.
    \end{lemma}

    \begin{proof}
        Take a sequence $(\pi_n)_{n \in \mathbb{N}} \subset \mathrm{Lift}(\mu_t)$ with $\pi_n \to \pi $ narrowly on $P(\Gamma_T)$. We need to show $\pi \in \mathrm{Lift}(\mu_t)$. Let $\phi \in C_b(\mathcal{X})$ and note that $\phi \circ e_t \in C_b(\Gamma_T)$. Thus, for any $t \in I$, we have
        \begin{equation}
            \int_{\Gamma_T} \phi (\gamma_t) \d \pi (\gamma)
            = \int_{\Gamma_T} \phi \circ e_t (\gamma ) \d \pi (\gamma)
            = \lim_{n \to \infty} \int_{\Gamma_T} \phi \circ e_t (\gamma ) \d \pi_n (\gamma)
            = \int_{\mathcal{X}} \phi(x) \d \mu_t (x).
        \end{equation}
    \end{proof}
    
\subsection{Compatible measures in \texorpdfstring{$P_p(\mathcal{X})$}{Pp(X)}}\label{subsubsec:compatibility}
    We begin with a simple observation. Let $\nu,\mu_0,\mu_1 \in P_2 (\mathbb{R}^{\mathrm{d}})$ be absolutely continuous measures.
    Denote by $T_{\nu}^{\mu_0}$ and $T_{\nu}^{\mu_1}$ the unique optimal maps from $\nu$ to $\mu_0$ and $\mu_1$, respectively. 
    We know that $(\mathrm{id} , T_{\nu}^{\mu_0} , T_{\nu}^{\mu_1})_{\#} \nu$ provides a multi-coupling whose two-dimensional marginal between $\mu_0$ and $\mu_1$ is not necessarily optimal. More precisely,  
    \begin{equation}\label{eq:compatibility_ineq}
        W_2(\mu_0, \mu_1) \leq  \sqrt{ \int_\mathcal{X} \big| T_{\nu}^{\mu_1} (y) - T_{\nu}^{\mu_0} (y) \big|^2 \d \nu(y) }  \eqqcolon W_{2,\nu}(\mu_0, \mu_1),
    \end{equation}
    which can be interpreted as the fact that $(P_2(\mathbb{R}^{\mathrm{d}}),W_2)$ has nonnegative sectional curvature. 
    This motivates introducing two notations.
    First, whenever the inequality in \eqref{eq:compatibility_ineq} is equality, the collection $\{\mu_0,\mu_1,\nu\}$ is referred to as \emph{compatible}.
    Second, it turns out that the quantity on the right-hand side of \eqref{eq:compatibility_ineq} can also be regarded as a \emph{distance} between $\mu_0$ and $\mu_1$, which is induced by $\nu$ and is usually denoted by $ W_{2,\nu}$. 

    The compatibility property was implicitly defined in \cite[Section 4]{Boissard2015} as an admissible property for a collection of transport maps and was further explored in \cite{PanaretosZemel2020}. 
    What we outlined above can be generalized to more than 3 measures and to general spaces, for which we give the following definition, mentioned in \cite[Remark 2.3.2]{PanaretosZemel2020}. Now, let $(\mathcal{X},d)$ be a complete separable metric space and $p \in [1,\infty)$.

    \begin{definition}[Compatibility of measures in $P_p(\mathcal{X})$]\label{def:compatibility}
         We say a collection of measures $\mathcal{M} \subset P_p(\mathcal{X})$ is compatible (in $p$-Wasserstein space), if, for every finite subcollection of $\mathcal{M}$, there exists a multi-coupling such that all of its two-dimensional marginals are optimal.
    \end{definition} 

    A prominent and well-known example of compatible measures is Wasserstein geodesics, as highlighted in the remark below.
    A compatible collection, however, need not lie on a Wasserstein geodesic. Take for instance an arbitrary measure and translate it along a non-geodesic curve. 
    Another example is any collection of measures in $P_p(\mathbb{R})$ is compatible.
    As shown in \cite[page 49]{PanaretosZemel2020}, any collection of Gaussian measures on $\mathbb{R}^{\mathrm{d}}$ that includes the standard Gaussian distribution and such that the covariance matrices of the measures are simultaneously diagonalizable (i.e. one can diagonalize all with one orthogonal matrix) is compatible.
    See more nontrivial examples presented in \cite[Proposition 4.1]{Boissard2015} and explained in \cite[Section 2.3.2]{PanaretosZemel2020}. 

    Compatibility can easily fail under rotation, as we study in \cref{exp:compatibility_vs_noncompatibility}. In that example, we also stress that if one adopts \cref{def:compatibility} as the definition of compatibility, it is not enough to check the optimality condition of the multi-couplings for subcollections of only three measures as in \eqref{eq:compatibility_ineq} (even if the measures, say on $\mathbb{R}^{\mathrm{d}}$, are absolutely continuous with respect to the Lebesgue measure).
    
    \begin{remark}[Wasserstein geodesics are compatible]\label{rmk:geodesics_are_compatible} 
        Let $(\mathcal{X}, d)$ be a complete separable geodesic space and let $\mu:[0,1] \to P_p(\mathcal{X})$ be a constant-speed geodesic for some $p \geq 1$.
        Here we confirm the elementary and well-known fact that the collection $(\mu_t)_{t \in [0,1]}$, parameterized by $t$, is a compatible collection of measures in $p$-Wasserstein space.
        Let $N \in \mathbb{N}$ and take $0 \leq t_1 \leq \cdots \leq t_N \leq 1$. 
        For the finite collection $\{ \mu_{t_i}$ :  $i \in \left\{1, \cdots, N \right\} \}$, there always exists a measure $\Upsilon_N \in P (\mathcal{X}^N)$ such that
        \begin{equation}
             (\mathrm{Pr}^{i,i+1})_{\#} \Upsilon_{N}  \in \mathrm{OptCpl}(\mu_{t_i}, \mu_{t_{i+1}}) , \qquad  \forall i \in \left\{1, \cdots, N -1 \right\}. \label{eq:Geodesics_are_compatible_constr}
        \end{equation}
        We now verify the optimality of all two-dimensional marginals, i.e., 
        \begin{equation}\label{eq:Geodesics_are_compatible_claim}
            (\mathrm{Pr}^{i,j})_{\#} \Upsilon_{N}  \in \mathrm{OptCpl}(\mu_{t_i}, \mu_{t_{j}}) , \qquad  \forall i,j \in \left\{1, \cdots, N  \right\}.
        \end{equation}
        If $|j-i| = 1$ or 0, nothing is needed to prove. Let's consider $j-i  > 1$. By triangle inequality and Minkowski's inequality for $L^p$ functions, we obtain
        \begin{align}
            W_p(\mu_{t_i},\mu_{t_j})
            & \leq \left\lvert \int d(x_i,x_j)^p \d \Upsilon_N \right\rvert^{1/p}\\
            & \leq \left\lvert \int \Big( d(x_i,x_{i+1}) + \cdots + d(x_{j-1},x_{j}) \Big)^p \d \Upsilon_N \right\rvert^{1/p}\\
             & \leq \left\lvert \int d(x_i,x_{i+1})^p \d \Upsilon_N \right\rvert^{1/p} +  \cdots  + \left\lvert \int d(x_{j-1},x_{j+1})^p \d \Upsilon_N \right\rvert^{1/p} \\
             & = W_p(\mu_{t_i},\mu_{t_{i+1}}) + \cdots + W_p(\mu_{t_{j-1}},\mu_{t_{j}})  = |t_j - t_i| W_p(\mu_{0},\mu_{1}).
        \end{align}
        But we know that $W_p(\mu_{t_i},\mu_{t_j}) = |t_j - t_i| W_p(\mu_{0},\mu_{1})$.
        Therefore, all inequalities are actually equalities. In particular, the first one, which proves the claim \eqref{eq:Geodesics_are_compatible_claim}. 
        We observe that the conclusion in \eqref{eq:Geodesics_are_compatible_claim} holds true regardless of how two-dimensional marginals are glued.
    \end{remark}

     \begin{remark}[Compatibility vs. Consistency]
     When a compatible collection of measures is not finite---such as $(\mu_t)_{t \in I}$ that we study---the multi-couplings of finite subcollections arising in the definition of compatibility need not be consistent. This is because optimal plans may not be unique.
    Starting from subcollections of small cardinality to larger ones, it is generally unclear how to produce a consistent family of multi-couplings, even in the case of Wasserstein geodesics. 
    \end{remark}

\subsection{Further remarks}
In this paper, we frequently require a map that connects two points in a geodesic space via a geodesic.
As in some geodesic spaces, there may be multiple geodesics connecting two points, we need a suitable selection map. The following remark is in this regard.    

\begin{remark}[Measurable selection of geodesics]\label{rmk:measurable_selection_geodesics}
         Let $(\mathcal{X}, d)$ be a complete, separable, and geodesic metric space.
         Let $\Upsilon \in P(\mathcal{X} \times \mathcal{X})$. We frequently need to construct a path measure 
         \begin{equation}
             \pi \coloneqq (\ell)_{\#} \Upsilon \, \in P\big( (C([0,1];\mathcal{X}),\mathcal{C}) \big) 
        \end{equation}
        with a measurable geodesic selection map $\ell : \mathcal{X}\times\mathcal{X} \to Geo([0,1];\mathcal{X})$.
        We briefly emphasize that such a map exists and that the above is well-defined (using a similar argument as in the proof of \cite[Proposition 1]{Lisini2007}). Define the multi-valued map $L: \mathcal{X}\times \mathcal{X} \to 2^{Geo ([0,1];\mathcal{X})}$ as follows:
        \begin{equation}
            L(x,y) \coloneqq \{ \gamma \in Geo ([0,1];\mathcal{X}) : \quad \gamma_0= x, \gamma_1 = y \}.
        \end{equation}
        This set is non-empty as $(\mathcal{X},d)$ is a geodesic space. The graph of $L$, denoted by 
        \begin{equation}
            \mathcal{G}(L) \coloneqq \Big\{(x,y,\gamma) \in \mathcal{X}\times\mathcal{X}\times C([0,1];\mathcal{X}): \quad  \gamma \in L(x,y) \Big\},
        \end{equation}
    is a closed set. Indeed, we take a sequence $(x^n,y^n,\gamma^n)_{n \in \mathbb{N}} \subset \mathcal{G}(L)$ convergent to $(x,y,\gamma)$ and show that $(x,y,\gamma) \in \mathcal{G}(L)$. First, we write
    \begin{align}
        d(x,\gamma_0)
        & \leq d(x,x^n) + d(x^n, \gamma_0) \\
        & = d(x,x^n) + d(\gamma_0^n, \gamma_0).
    \end{align}
    Taking the limit $n\to \infty$, the second term on the right-hand side goes to zero as $d_\infty (\gamma^n, \gamma) \to 0$. Thus, we have $\gamma_0  =x$, and similarly $\gamma_1 = y$. To show that $\gamma$ is a constant-speed geodesic, take any $t,s \in [0,1]$ and observe that
    \begin{align}
        d(\gamma_s,\gamma_t)
        & = \lim_{n \to \infty} d(\gamma^n_s,\gamma^n_t) \\
        & = \lim_{n \to \infty} |t-s| d(x^n,y^n) = |t-s|  d(x,y).
    \end{align}
    As $\mathcal{G}(L)$ is closed, it is Borel-measurable, i.e., it belongs to $\mathcal{B}(\mathcal{X}^2\times C)$, which is equal to $\mathcal{B}(\mathcal{X}^2)\otimes\mathcal{C}$, since all spaces are separable here.
    Obviously, $\mathcal{G}(L)$ is also in $\mathcal{B}(\mathcal{X}^2)_\Upsilon \otimes\mathcal{C}$, where $\mathcal{B}(\mathcal{X}^2)_\Upsilon $ denotes the $\Upsilon$-completion of $\mathcal{B}(\mathcal{X}^2)$. It follows from a measurable selection theorem by Aumann (see, e.g., \cite[Theorem III.22]{CastaingValadier1977} and \cite[page 455]{AmbrosioFigalli2009}) that there exists a measurable map $\ell: \big(\mathcal{X}^2,\mathcal{B}(\mathcal{X}^2)_\Upsilon\big) \to \big(C([0,1];\mathcal{X}),\mathcal{C}\big)$ such that $\ell(x,y) \in Geo([0,1];\mathcal{X})$ for $\Upsilon$-a.e. $(x,y)\in \mathcal{X}^2$.
    The above argument for selecting geodesics between two points can easily be generalized to multiple points.
\end{remark}

\section{Main results}\label{sec:results_d}
\subsection{Existence of a minimizer for the variational problem}
    We consider \cref{prob:variational_problem}, as formulated in the introduction.

    \begin{proposition}[Existence of a minimizer]\label{prop:existence_of_minimizer}
        Let $(\mathcal{X}, d)$ be a complete separable metric space, and $I \coloneqq [0,T] \subset \mathbb{R}$. Let $\Psi: C(I;\mathcal{X}) \to [0,+\infty]$ be {\normalfont \textcircled{\small{1}}} a lower semi-continuous map {\normalfont \textcircled{\small{2}}} whose sublevels are relatively compact in $C(I;\mathcal{X})$. Assume that the infimum  \eqref{eq:variational_problem} is finite. Then there exists a minimizer $\pi \in P(C(I;\mathcal{X})) $ to \cref{prob:variational_problem}. 
    \end{proposition}

    \begin{proof}
        Let $(\pi_n)_{n \in \mathbb{N}} \subset \mathrm{Lift}(\mu_t)$ be a minimizing sequence, i.e.,
        \begin{equation}
            \lim_{n \to \infty} \int_{\Gamma_T} \Psi \d \pi_n = \inf_{\pi \in \mathrm{Lift}(\mu_t)} \int_{\Gamma_T} \Psi \d \pi \eqqcolon M < + \infty.
        \end{equation}
        In particular, we have 
        \begin{equation}\label{eq:proof_existence_minimizer}
            \sup_{n} \int_{\Gamma_T} \Psi \d \pi_n  < + \infty.
        \end{equation}
        Since the sublevels of $\Psi$ are compact (which follows from lower semi-continuity and relative compactness), the condition \eqref{eq:proof_existence_minimizer} implies that the family of measures $\{\pi_n \}_{n} \subset P(\Gamma_T)$ is tight by \cref{lemma:tightness_criterion}. By Prokhorov's theorem, there exists a subsequence $\{\pi_{n_k} \}_{k \in \mathbb{N}}$ such that $\pi_{n_k} \to \pi$ narrowly on $P(\Gamma_T)$ as $k \to \infty$. By \cref{lemma:lift_property_closed}, we have $ \pi \in \mathrm{Lift}(\mu_t)$ and thus $
        \int_{\Gamma_T} \Psi \d \pi \geq M$. On the other hand, since $\Psi$ is lower semi-continuous, we obtain the reverse inequality  
        \begin{equation}
            \int_{\Gamma_T} \Psi \d \pi \leq \liminf_{k \to \infty } \int_{\Gamma_T} \Psi \d \pi_{n_k} = M. 
        \end{equation}
        This means that $\int_{\Gamma_T} \Psi \d \pi = M$ and thus $\pi$ is a minimizer.
    \end{proof}
    
\subsection{From path measures to Wasserstein curves}\label{subsec:from_pi_to_mu}
    In this section, we consider the following question: \textit{Given a path measure $\pi $ whose energy with respect to a functional $\Psi$ is finite, what can be deduced about the regularity of the curve of one-dimensional time marginals $t \mapsto \mu_t \coloneqq (e_t)_{\#} \pi$?} The candidates of $\Psi$ which we investigate here are related to the functionals
    \begin{align}
        \gamma  \mapsto w_{\delta} (\gamma), \qquad 
        \gamma  \mapsto | \gamma |_{\theta\textrm{-}\mathrm{H\ddot{o}l}}, \qquad 
        \gamma  \mapsto | \gamma |_{q\textrm{-}\mathrm{var}}, \qquad 
        \gamma  \mapsto |\gamma|_{W^{\alpha,p}},
    \end{align}
    These functionals, as studied in the preliminary  \cref{sec:preliminaries_d}, are lower semi-continuous maps from $\Gamma_T \coloneqq C([0,T];\mathcal{X}) \to [0,+\infty]$.
    In particular, they are measurable and their integrals with respect to $\pi \in P(\Gamma_T)$ make sense.
    In all four cases, using a similar strategy, we prove that $(\mu_t)$ inherits the same kind of regularity, with its regularity bounded from above by the energy of $\pi$.
    In contrast, the same conclusion cannot be drawn for the infinitesimal variation
    $$
    \gamma \mapsto |\gamma|_{q\textrm{-}\mathrm{var}\textrm{-}\mathrm{limsup}}
    $$
    when $q>1$.
    In \cref{exp:infinitesimal_variation}, we provide a counterexample $(\mu_t)$ whose only existing lift fails to provide an upper bound for the infinitesimal variation of $(\mu_t)$.
    As noted in \cref{rmk:non-lsc-inf-variation}, the map $\gamma \mapsto |\gamma|_{q\textrm{-}\mathrm{var}\textrm{-}\mathrm{limsup}}$ is not even lower semi-continuous when $q>1$. 
    
    \begin{theorem}\label{thm:lift_to_mu_C}
        Let $(\mathcal{X},d)$ be a complete separable metric space.
        Let $\pi \in P(C([0,T];\mathcal{X})) $ satisfy
        \begin{equation}\label{eq:lift_to_mu_C_integrability}
            \int_{\Gamma_T} \Big( d(\gamma_0,\bar{x})^p + w_{\delta} (\gamma)^p \Big) \d \pi (\gamma) < + \infty
        \end{equation}
        for some $p \in [1,\infty)$ and $0<\delta\leq T$ and $\bar{x}\in \mathcal{X}$.
        Then, $t \mapsto \mu_t \coloneqq {(e_t)}_{\#} \pi$ is in $ C ([0,T];P_p(\mathcal{X}))$, and moreover,
        \begin{equation}
             w_{\delta} (\mu)^p \leq  \int_{\Gamma_T} w_{\delta} (\gamma)^p \d \pi (\gamma). 
        \end{equation}
    \end{theorem}
    
    \begin{proof}
        The finiteness of the first term in \eqref{eq:lift_to_mu_C_integrability} is nothing but $\mu_0 \coloneqq (e_0)_{\#} \pi \in P_p(\mathcal{X})$. We first prove that $\mu_t \coloneqq (e_t)_{\#} \pi$ actually lies in $P_p(\mathcal{X})$ at all times. We have 
        \begin{align}
            \int_\mathcal{X} d(\bar{x},x)^p \d \mu_t(x) &= \int_{\Gamma_T} d(\bar{x},\gamma_t)^p \d\pi(\gamma)\\
            & \leq 2^{p-1} \int_{\Gamma_T} \Big(  d(\bar{x},\gamma_0)^p + d(\gamma_0,\gamma_t)^p \Big) \d\pi(\gamma). \label{eq_pmoment_lift_to_mu_C} 
        \end{align}
        The first summand on the right-hand side is finite by assumption.
        As for the second summand, suppose first $t\leq \delta$, then
        \begin{equation}
            \int d(\gamma_0,\gamma_t)^p \d\pi \leq \int \left( \sup_{|t-s|\leq \delta }  d(\gamma_s,\gamma_t) \right)^p \d\pi = \int w_{\delta} (\gamma)^p \d \pi < \infty.
        \end{equation}
        Now suppose $\delta < t \leq T$. Take largest $N \in \mathbb{N}$ such that $N \delta \leq t$.
        By  triangular inequality, we can estimate
        \begin{align}
            \int d(\gamma_0,\gamma_t)^p \d\pi
            & \leq \int \Big( d(\gamma_0,\gamma_\delta) + d(\gamma_\delta,\gamma_{2\delta}) + \cdots + d(\gamma_{N\delta},\gamma_t) \Big)^p \d\pi \\
            & \leq (N+1)^p \int w_{\delta} (\gamma)^p \d \pi \\
            & \leq \big(\frac{T}{\delta} + 1 \big)^p \int w_{\delta} (\gamma)^p \d \pi < \infty. 
        \end{align}
        Accordingly, \eqref{eq_pmoment_lift_to_mu_C} is bounded and thus $\mu_t \in P_p(\mathcal{X})$ for all $t \in [0,T]$. 
        \\
        Next, we show that $t \mapsto \mu_t$ is a continuous curve in $p$-Wasserstein space. Fix an arbitrary $t \in [0,T]$ and take a sequence $t_n \to t$ as $n \to \infty$. This sequence will eventually fall in $\delta$-neighbourhood of $t$. Thus, by Lebesgue's dominated convergence theorem on the measure space $(\Gamma_T, \mathcal{B}(\Gamma_T), \pi)$ and the integrability assumption \eqref{eq:lift_to_mu_C_integrability}, we obtain 
        \begin{align}\label{eq:lift_to_mu_C_continuity}
            \lim_{n \to \infty} W_p^p (\mu_t, \mu_{t_n}) & \leq \lim_{n \to \infty} \int_{\Gamma_T} d(\gamma_t,\gamma_{t_n})^p \d \pi (\gamma) = \int_{\Gamma_T} \left( \lim_{n \to \infty} d(\gamma_t,\gamma_{t_n})^p \right) \d \pi (\gamma) = 0.
        \end{align}
        To prove the final claim, consider $s,t \in [0,T]$ such that $|t-s|\leq \delta$. We have
        \begin{align}
            W_p^p(\mu_s,\mu_t) & \leq \int d(\gamma_s,\gamma_t)^p \d\pi(\gamma)\\
            & \leq \int \left( \sup_{|v-u| \leq \delta} d(\gamma_u,\gamma_v) \right)^p \d\pi(\gamma)  =  \int w_{\delta} (\gamma)^p \d \pi (\gamma).
        \end{align}
        Taking supremum over all $|t-s|\leq \delta$ yields the result.
    \end{proof}
    
    \begin{theorem}\label{thm:lift_to_mu_Hol}
        Let $(\mathcal{X},d)$ be a complete separable metric space.
        Let $\pi \in P(C([0,T];\mathcal{X})) $ be concentrated on $C^{\theta\textrm{-} \mathrm{H\ddot{o}l}} ([0,T];\mathcal{X})$ and satisfy 
         \begin{equation}
            \int_{\Gamma_T} \Big( d(\gamma_0,\bar{x})^p + | \gamma |_{\theta\textrm{-}\mathrm{H\ddot{o}l};[0,T]}^p \Big) \d \pi(\gamma) < + \infty
        \end{equation}
        for some $p \in [1,\infty)$ and $\theta\in(0,1]$ and $\bar{x}\in \mathcal{X}$.
        Then, $t \mapsto \mu_t \coloneqq {(e_t)}_{\#} \pi $ is in $ C^{\theta\textrm{-} \mathrm{H\ddot{o}l}} ([0,T];P_p(\mathcal{X}))$, and moreover,
        \begin{equation}
            |\mu|_{\theta\textrm{-}\mathrm{H\ddot{o}l};[s,t]}^p \leq \int_{\Gamma_T} |\gamma|_{\theta\textrm{-}\mathrm{H\ddot{o}l};[s,t]}^p \d \pi(\gamma)
        \end{equation}
        for any $0\leq s<t\leq T$. 
    \end{theorem}
    
    \begin{proof}
        The argument is similar to the proof of \cref{thm:lift_to_mu_C}. Let us prove it independently.  
        By assumption, $\mu_0$ has finite $p$-moment. We prove that $\mu_t \coloneqq (e_t)_{\#} \pi$ has also finite $p$-moment for all $t \in (0,T] $ and thus stays in $p$-Wasserstein space. Given the fixed point $\bar{x}\in \mathcal{X}$, we have
        \begin{align}
            \int_\mathcal{X} d(\bar{x},x)^p \d \mu_t(x) &= \int_{\Gamma_T} d(\bar{x},\gamma_t)^p \d\pi(\gamma)\\
            & \leq 2^{p-1} \int_{\Gamma_T} \Big(  d(\bar{x},\gamma_0)^p + d(\gamma_0,\gamma_t)^p \Big) \d\pi(\gamma)\\
            &\leq 2^{p-1} \left( W_p^p (\delta_{\bar{x}},\mu_0) + |T|^{p\theta} \int_{\Gamma_T} |\gamma|_{\theta\textrm{-}\mathrm{H\ddot{o}l};[0,T] }^p \d\pi(\gamma) \right) <\infty,
        \end{align}
        where we used the triangular inequality, the elementary bound $|a+b|^p \leq 2^{p-1} (|a|^p + |b|^p) $, and the fact that $\pi$ is concentrated on $\theta$-H\"{o}lder curves. 
        We conclude that $\mu_t\in P_p(\mathcal{X})$ for all $t\in[0,T]$.
        \\
        Now, fix $0 \leq s <t \leq T$ and let us look at the Wasserstein curve at all times $s \leq u<v \leq t$. Since $(e_{u},e_{v})_{\#} \pi \in \text{Cpl} (\mu_{u}, \mu_{v})$ is not necessarily optimal, we can estimate
        \begin{align}
            W_p^p(\mu_u,\mu_v) & \leq \int_{\Gamma_T} d(\gamma_u,\gamma_v)^p \d\pi(\gamma)\\
            & \leq |v-u|^{p \theta} \int_{\Gamma_T} |\gamma|_{\theta\textrm{-}\mathrm{H\ddot{o}l};[s,t] }^p \d\pi(\gamma),
        \end{align}
        which, in particular, implies the continuity of $t \mapsto \mu_t$ and furthermore, 
        \begin{equation}
            |\mu|_{\theta\textrm{-}\mathrm{H\ddot{o}l};[s,t]} \coloneqq \sup_{s\leq u< v \leq t} \frac{W_p(\mu_u, \mu_v)}{|v-u|^\theta} \leq \left( \int_{\Gamma_T} |\gamma|_{\theta\textrm{-}\mathrm{H\ddot{o}l};[s,t] }^p \d\pi(\gamma) \right)^{1/p}.
        \end{equation}
        Raising this expression to the power $p$, we obtain the result. 
    \end{proof}
    
    \begin{theorem}\label{thm:lift_to_mu_var}
        Let $(\mathcal{X},d)$ be a complete separable metric space.
        Let $\pi \in P(C([0,T];\mathcal{X})) $ be concentrated on $C^{q\textrm{-} \mathrm{var}} ([0,T];\mathcal{X})$ and satisfy
         \begin{equation}\label{eq:lift_to_mu_var_integrability}
            \int_{\Gamma_T} \Big( d(\gamma_0,\bar{x})^p + | \gamma |_{q\textrm{-}\mathrm{var};[0,T]}^p \Big) \d \pi(\gamma)  < + \infty
        \end{equation}
        for some $1 \leq q \leq p < \infty$ and  $\bar{x}\in \mathcal{X}$.
        Then, $t \mapsto \mu_t \coloneqq {(e_t)}_{\#} \pi $ is in $ C^{p\textrm{-} \mathrm{var}} ([0,T];P_p(\mathcal{X}))$, and moreover,
        \begin{equation}\label{eq:lift_to_mu_var}
            |\mu|_{p\textrm{-}\mathrm{var};[s,t]}^p \leq \int_{\Gamma_T} |\gamma|_{q\textrm{-}\mathrm{var};[s,t]}^p \d \pi(\gamma)
        \end{equation}
        for any $0\leq s<t\leq T$. 
    \end{theorem}
    
    \begin{proof}
        Let us first observe that   
        \begin{equation}
             d(\gamma_s,\gamma_t)^p \leq  \left( \sup_{s,t \in [0,T]} d(\gamma_s,\gamma_t)^q \right) ^{p/q} \leq | \gamma |_{q\textrm{-}\mathrm{var};[0,T]}^p.
        \end{equation}
        Continuing form \eqref{eq_pmoment_lift_to_mu_C}, we can therefore estimate
         \begin{align}
            \int_\mathcal{X} d(\bar{x},x)^p \d \mu_t(x) \leq 2^{p-1} \int_{\Gamma_T} \Big(  d(\bar{x},\gamma_0)^p + | \gamma |_{q\textrm{-}\mathrm{var};[0,T]}^p \Big) \d\pi(\gamma),
        \end{align}
        which yields that $\mu_t\in P_p(\mathcal{X})$ for all time.
        Furthermore, $t \mapsto \mu_t$ is continuous in $p$-Wasserstein space, as \eqref{eq:lift_to_mu_C_continuity} also holds here. 
        \\
        To prove that $(\mu_t)$ has finite $p$-variation, take $0 \leq s <t \leq T$ and let $D=(t_i)$ be an arbitrary partition of the interval $[s,t]$. As before, $(e_{t_{i}},e_{t_{i+1}})_{\#} \pi \in \text{Cpl} (\mu _{t_{i}}, \mu_{t_{i+1}})$ provides us with an upper estimate for the distance, 
        \begin{align}
            \sum_{i} W_p^p (\mu_{t_{i}}, \mu _{t_{i+1}}) & \leq  \int_{\Gamma_T} \sum_{i} d(\gamma_{t_{i}}, \gamma _{t_{i+1}})^p \d\pi(\gamma) \\
            & \leq \int_{\Gamma_T}  |\gamma|_{p\textrm{-}\mathrm{var};[s,t]}^p \d\pi(\gamma) \\
            & \leq \int_{\Gamma_T}  |\gamma|_{q\textrm{-}\mathrm{var};[s,t]}^p \d\pi(\gamma),
        \end{align}
        where we used the fact that $|\gamma|_{p\textrm{-}\mathrm{var}}\leq |\gamma|_{q\textrm{-}\mathrm{var}}$ on the space of continuous paths if $1 \leq q \leq p$.  
        Now the right-hand side no longer depends on the choice of partition. Taking the supremum over all partitions $D \in \mathcal{D}([s,t])$ yields the result. 
    \end{proof}
    
    \begin{theorem}\label{thm:lift_to_mu_Walphap}
        Let $(\mathcal{X},d)$ be a complete separable metric space.
        Let $\pi \in P(C([0,T];\mathcal{X})) $ be concentrated on $W^{\alpha, p} ([0,T];\mathcal{X})$ and satisfy
        \begin{equation}\label{eq:lift_to_mu_Walphap_assumptiom}
            \int_{\Gamma_T} \Big( d(\gamma_0,\bar{x})^p + | \gamma |_{W^{\alpha,p};[0,T]}^p  \Big) \d \pi(\gamma) < + \infty
        \end{equation}
        for some $1<p<\infty$ and $ \frac{1}{p} < \alpha < 1$ and $\bar{x}\in \mathcal{X}$.
        Then, $t \mapsto \mu_t \coloneqq {(e_t)}_{\#} \pi $ is in $ W^{\alpha, p} ([0,T];P_p(\mathcal{X}))$, and moreover,
        \begin{equation}\label{eq:lift_to_mu_Walphap}
            |\mu|_{W^{\alpha,p};[s,t]}^p \leq \int_{\Gamma_T} |\gamma|_{W^{\alpha,p};[s,t]}^p \d \pi(\gamma)
        \end{equation}
        for any $0\leq s<t\leq T$. The same statement holds for $|\cdot|_{b^{\alpha,p}}$.
    \end{theorem}

    \begin{proof}
        By assumption, $\pi$-a.e. $\gamma$ has finite $W^{\alpha,p}$-norm.
        This allows us to estimate the modulus of continuity of  $\gamma$ by applying Garsia--Rodemich--Rumsey inequality, as described in \eqref{eq:distance_estimate_GRR}, and  obtain
        $$ d (\gamma_s,\gamma_t) \leq \bar{c} |t-s|^{\alpha - \frac{1}{p}} | \gamma |_{W^{\alpha,p};[s,t]} \quad \mathrm{for }\,\, \pi \textrm{-}\mathrm{a.e.} \, \gamma,$$
        where the constant $\bar{c}$ depends only on $(\alpha,p)$.
        \\
        Now let us show that $\mu_t\in P_p(\mathcal{X})$ for all $t\in(0,T]$.  Using the triangular inequality and the estimate above, we have
        \begin{align}
            \int d(\bar{x},x)^p \d \mu_t(x) &= \int d(\bar{x},\gamma_t)^p \d\pi(\gamma)\\
            & \leq 2^{p-1} \int \Big(  d(\bar{x},\gamma_0)^p + d(\gamma_0,\gamma_t)^p \Big) \d\pi(\gamma)\\
            & \leq 2^{p-1} \left( W_p^p (\delta_{\bar{x}},\mu_0) + \bar{c}^p |t|^{\alpha p - 1 }\int |\gamma|_{W^{\alpha,p};[0,t] }^p \d\pi(\gamma) \right) \\
            & \leq 2^{p-1} \left( W_p^p (\delta_{\bar{x}},\mu_0) + \bar{c}^p |T|^{\alpha p - 1 }\int |\gamma|_{W^{\alpha,p};[0,T] }^p \d\pi(\gamma) \right) < \infty 
            \label{eq:lift_to_mu_Walphap_proof_estimate_1}
        \end{align}
        which shows $\mu_t \in P_p(\mathcal{X})$.
        Next, to prove the second claim, observe that
         \begin{align}
             |\mu|_{W^{\alpha,p};[s,t]}^p
             & = \iint_{[s,t]^2} \frac{W_p^p (\mu_u,\mu_v)}{|v-u|^{1+\alpha p}} \d u \d v \\
             & \leq \iint_{[s,t]^2} \frac{\int d(\gamma_u,\gamma_v)^p \d \pi (\gamma)}{|v-u|^{1+\alpha p}} \d u \d v \\
             & = \int \left( \iint_{[s,t]^2} \frac{ d(\gamma_u,\gamma_v)^p}{|v-u|^{1+\alpha p}} \d u \d v \right) \d \pi (\gamma) \\
             & = \int |\gamma|_{W^{\alpha,p};[s,t]}^p \d \pi(\gamma), \label{eq:lift_to_mu_Walphap_proof_estimate_2}
         \end{align}
         where we used Tonelli's theorem to interchange integrals. Similarly, we have (when $T=1$) that
         \begin{align}
             |\mu|_{b^{\alpha,p}}^p
             & \coloneqq \sum_{m=0}^{\infty} 2^{m(\alpha p -1)} \sum_{k=0}^{2^m-1} W_p^p (\mu_{t_{k}^{(m)}}, \mu_{t_{k+1}^{(m)}}) \\
             & \leq  \lim_{M \to \infty} \int \sum_{m=0}^{M} 2^{m(\alpha p -1)} \sum_{k=0}^{2^m-1} d (\gamma_{t_{k}^{(m)}}, \gamma_{t_{k+1}^{(m)}})^p \d \pi (\gamma) \\
             & = \int \sum_{m=0}^{\infty} 2^{m(\alpha p -1)} \sum_{k=0}^{2^m-1} d (\gamma_{t_{k}^{(m)}}, \gamma_{t_{k+1}^{(m)}})^p \d \pi (\gamma) \\
             & = \int |\gamma|_{b^{\alpha,p}}^p \d \pi (\gamma)
         \end{align}
         where we used Beppo Levi's lemma to interchange limit and integral.  
    \end{proof}

   As a final result in this section, we show that if a path measure $\pi$ and its time marginals $(\mu_t)$ satisfy the equality in the previous theorem, then $(\mu_t)$ is compatible. 
    \begin{proposition}\label{prop:equality_implies_compatibility}
    Let $(\mathcal{X},d)$ be a complete separable metric space.
    Let $\pi \in P(C([0,T];\mathcal{X})) $ be concentrated on $W^{\alpha, p} ([0,T];\mathcal{X})$ and satisfy \eqref{eq:lift_to_mu_Walphap_assumptiom} for some $1<p<\infty$ and $ \frac{1}{p} < \alpha < 1$ and $\bar{x}\in \mathcal{X}$.
    Assume that $\pi$ and $t \mapsto \mu_t \coloneqq {(e_t)}_{\#} \pi $ satisfy the equality
        \begin{align}
            |\mu|_{W^{\alpha,p};[0,T]}^p = \int_{\Gamma_T} |\gamma|_{W^{\alpha,p};[0,T]}^p \d \pi (\gamma),
        \end{align}
        then $(\mu_t)_{t\in [0,T]}$ is compatible in $P_p(\mathcal{X})$.
    \end{proposition}

    \begin{proof}
        By \cref{thm:lift_to_mu_Walphap}, $(\mu_t)$ is in $P_p(\mathcal{X})$. To show that $(\mu_t)$ is  compatible in $P_p(\mathcal{X})$, we write 
         \begin{align}
             0 = \int_{\Gamma_T} |\gamma|_{W^{\alpha,p}}^p \d \pi(\gamma) - |\mu|_{W^{\alpha,p}}^p = \iint_{[0,T]^2} \underbrace{ \left(  \frac{\int_{\Gamma_T} d(\gamma_s,\gamma_t)^p \d \pi (\gamma)}{|t-s|^{1+\alpha p}} - \frac{W_p^p (\mu_s,\mu_t)}{|t-s|^{1+\alpha p}}  \right)}_{\eqqcolon f(s,t)}  \d s \d t .
         \end{align}
         Since $\pi$ is a lift of $(\mu_t)$, the function $f$ is non-negative. Thus, the equality above implies that $f=0$ almost everywhere. 
         It follows that $f=0$ everywhere because $f$ is continuous. Indeed, the application
         $$
         (s,t) \mapsto W_p^p (\mu_s,\mu_t)
         $$
         is continuous because $(\mu_t)$ is continuous. Similarly, the application
         $$
         (s,t) \mapsto \int_{\Gamma_T} d(\gamma_s,\gamma_t)^p \d \pi (\gamma)   
         $$
         is also continuous. Take a sequence $(s_m,t_m) \to (s,t)$ as $m \to \infty$ and note that
         \begin{equation}
             \lim_{m \to \infty } \int_{\Gamma_T} d(\gamma_{s_m},\gamma_{t_m})^p \d \pi (\gamma) = \int_{\Gamma_T} \left(  \lim_{m \to \infty } d(\gamma_{s_m},\gamma_{t_m})^p \right) \d \pi (\gamma) = \int_{\Gamma_T} d(\gamma_{s},\gamma_{t})^p \d \pi (\gamma)
         \end{equation}
         by Lebesgue's dominated convergence theorem, where the dominated function
         $$
             d (\gamma_s,\gamma_t)^p \leq \bar{c}^p T^{\alpha p  - 1} | \gamma |^p_{W^{\alpha,p};[0,T]},
         $$
         coming from Garsia--Rodemich--Rumsey inequality \eqref{eq:distance_estimate_GRR}, is integrable by assumption. 
         \\
         Finally, $f=0$ means that $(e_s,e_t)_\# \pi \in \mathrm{OptCpl}(\mu_s, \mu_t)$ for  all $(s,t) \in [0,T]^2$. This immeditely imply that $(\mu_t)$ is compatible, as for every finite collection $\{ \mu_{t_i}$: $t_i \in [0,T]$, $i \in \left\{1, \cdots, N \right\}\}$, we have a multi-coupling given by $(e_{t_1},e_{t_2}, \cdots, e_{t_N})_\# \pi$ whose two-dimensional marginals are all optimal.
    \end{proof}

    \subsection{From Wasserstein curves to path measures: a superposition principle}\label{subsec:from_mu_to_pi}
    In this section, we go in the reverse direction as in the previous section: \textit{We start with a curve $t \mapsto \mu_t $ in a Wasserstein space with fractional Sobolev regularity and compatibility and then construct a lift $\pi$ that realizes its regularity.}  
    To this end, we require the underlying space $\mathcal{X}$ to have additional structure. We follow Construction \hyperref[itm:ConstructionB]{{\normalfont \textcircled{\small{B}}}} as explained in the introduction.

    \begin{theorem}\label{thm:optimal_lift_mu_Walphap_compatible}
        Let $(\mathcal{X}, d)$ be a complete, separable, and locally compact length metric space, and $I \coloneqq [0,T] \subset \mathbb{R}$.
        Let $(\mu_t) \in  W^{\alpha,p}(I;P_p(\mathcal{X}))$ with $1<p<\infty$ and $\frac{1}{p}< \alpha <  1$.
        Assume that $(\mu_t)_{t \in I}$ is compatible in $P_p(\mathcal{X})$.
        Then, construction \hyperref[itm:ConstructionB]{{\normalfont \textcircled{\small{B}}}} converges narrowly (up to a subsequence) to a probability measure $\pi \in P(C(I;\mathcal{X}))$ satisfying
        \begin{enumerate}[label=(\roman*), font=\normalfont]
            \item $\pi$ is concentrated on $W^{\alpha,p}(I;\mathcal{X})$;
            \item $(e_t)_\#\pi=\mu_t$ for all $t\in I$;
            \item $(e_s,e_t)_\# \pi \in \mathrm{OptCpl}(\mu_s, \mu_t)$ for all $s,t \in I$; and in particular,
            \begin{align}\label{eq:optimality_pi_1}
            |\mu|_{W^{\alpha,p}}^p = \int_{\Gamma_T} |\gamma|_{W^{\alpha,p}}^p \d \pi (\gamma) .
            \end{align}
        \end{enumerate}
        The same statement holds for $|\cdot|_{b^{\alpha,p}}$.
    \end{theorem}

    \begin{proof}
        Let us prove the result for the case $T=1$, i.e., $I = [0,1]$, which is clearly not restrictive. We write
        $\Gamma \coloneqq \Gamma_1 \coloneqq C([0,1];\mathcal{X})$. 
        \\
        \textbf{Step 0 (Construction of $\{\pi_n\}_{n \in \mathbb{N}} \subset P(\Gamma)$).}
        We carry out construction \hyperref[itm:ConstructionB]{{\normalfont \textcircled{\small{B}}}}. 
        We take the dyadic sequence of partitions $(D_n)_{n \in \mathbb{N}}$ of $I$, i.e., for any integer $n \in \mathbb{N}$, we divide the interval $I$ into $N_n \coloneqq 2^n$ equal pieces,
        \begin{align}
            & D_n \coloneqq \{ 0=t_0^{(n)} < t_1^{(n)} < \cdots < t_{N_n}^{(n)} = 1 \}, \label{eq:const_notation_Dn}\\
            & t_i^{(n)} \coloneqq \frac{i}{2^n} \quad i \in \left\{ 0,1, \cdots, 2^n = N_n \right\}.  \label{eq:const_notation_tin}
        \end{align}
        Let $\bm{\mathcal{X}}_n$ denote the product space
        \begin{equation}\label{eq:const_notation_Xn}
            \bm{\mathcal{X}}_n \coloneqq \mathcal{X}_0 \times \mathcal{X}_1 \times \cdots \times \mathcal{X}_{N_n}
        \end{equation}
        with $\{ \mathcal{X}_i \}$ representing copies of $\mathcal{X}$. 
        The compatibility assumption, in the sense of \cref{def:compatibility}, ensures the existence of a measure on the product space,
        \begin{equation}
            \Upsilon_{n} \in P(\bm{\mathcal{X}}_n),
        \end{equation}
        such that it is a multi-coupling, i.e.,
        \begin{equation}\label{eq:1D_marginal_Upsilon}
            (\text{Pr}^i)_{\#} \Upsilon_{n} = \mu_{t_i^{(n)}}, \quad \forall i \in \left\{ 0,1, \cdots, N_n \right\},
        \end{equation}
        and, moreover, 
        \begin{gather}\label{eq:2D_marginal_Upsilon}
            (\text{Pr}^{i,i+\frac{2^n}{2^m}})_{\#} \Upsilon_{n} \in \text{OptCpl}\big(\mu_{t^{(n)}_{i}},\mu_{t^{(n)}_{i+\frac{2^n}{2^m}}}\big) , 
        \end{gather}
        for all $i \in \big\{k \frac{2^n}{2^m} \big| k \in \{ 0,1,\cdots, 2^m-1\} \big\} $ and $ m \in \{0,1,\cdots n\}$.    
        The maps $\text{Pr}^i$, $\text{Pr}^{i,j}$ are projections from $\bm{\mathcal{X}}_n$ to the $(i)$-th, $(i,j)$-th component, respectively.
        (See \cref{fig:construction} for an illustration of the condition above and recall \cref{rmk:ompatibility_dyadic}.)
    
        Next, since $(\mathcal{X},d)$ is assumed to be a complete and locally compact length metric space, the Hopf--Rinow \cref{Hopf-Rinow} ensures the existence of at least one geodesic between any two points.
        We consider a geodesic selection and interpolation map connecting the points with constant-speed geodesics: 
        \begin{equation}\label{eq:def_geod_interp_0}
            \ell : \, \bm{x}=(x_0,\cdots ,x_{N_n}) \, \in \, \bm{\mathcal{X}}_n \mapsto \ell_{\bm{x}} \in C(I;\mathcal{X})
        \end{equation}
        by defining 
        \begin{align}\label{eq:def_geod_interp}
            \ell_{\bm{x}}(t) & \coloneqq x_i, \quad t = t^{(n)}_i, \quad  i \in \{ 0,1, \cdots, N_n \};
        \end{align}
        and connecting by a constant speed geodesics in between (i.e. the speed over each time segment is equal to $2^n d(x_i,x_{i+1})$).
        In other words, $\ell_{\bm{x}}$ is a piecewise geodesic connecting the points.
        As discussed in \cref{rmk:measurable_selection_geodesics} for the two-point case, one can always find a $(\mathcal{B}(\mathcal{X}^{2^n+1})_{\Upsilon_n},\mathcal{C})$-measurable geodesic selection map $\ell_{\bm{x}}$, despite the possible non-uniqueness of geodesics. Here, $\mathcal{B}(\mathcal{X}^{2^n+1})_{\Upsilon_n}$ denotes the $\Upsilon_n$-completion of $\mathcal{B}(\mathcal{X}^{2^n+1})$.            
        Therefore, using this measurable map, we can build a sequence of path measures:
        \begin{equation}\label{eq:pi_n_Upsilon}
             \pi_n \coloneqq (\ell)_{\#} \Upsilon_n \in P(C(I;\mathcal{X})), \quad \forall n \in \mathbb{N}. 
        \end{equation}
        \\
        \textbf{Step 1 (Tightness of $\{\pi_n\}_{n \in \mathbb{N}} \subset P(\Gamma)$).} 
        We show that the family measures $\{\pi_n\}$ is tight in $P(C([0,1];\mathcal{X}))$.
        To this end, we use the tightness condition developed for path measures in \cref{prop:tightness_balphap} using the $b^{\alpha,p}$-norm. Note that since $(\mathcal{X},d)$ is assumed to be a complete and locally compact length metric space, the Hopf--Rinow theorem \ref{Hopf-Rinow} ensures that closed bounded sets in $\mathcal{X}$ are compact. Hence, we can use \cref{prop:tightness_balphap}.
        Our goal is to show 
        \begin{equation}\label{eq:tightness_computation_claim}
             \sup_{n \in \mathbb{N}} \int_{\Gamma} \Big(  d(\gamma_0, \bar{x})^p +  | \gamma |^p_{b^{\alpha,p}} \Big) \d \pi_n (\gamma) < +\infty. 
        \end{equation}
        For the first term, we have
        \begin{align}\label{eq:tightness_computation_step_00}
            \int_{\Gamma} d(\gamma_0, \bar{x})^p \d \pi_n (\gamma)
            & = \int_{\bm{\mathcal{X}}_n} d(x_0, \bar{x})^p \d \Upsilon_n (\bm{x}) 
            = \int_{\mathcal{X}} d (x,\bar{x})^p \d \mu_0 (x) < + \infty, 
        \end{align}
        where \eqref{eq:pi_n_Upsilon} and \eqref{eq:1D_marginal_Upsilon} are used in the first and second equality, respectively.
        The last expression is indeed finite because $\mu_0 \in P_p(\mathcal{X})$.
        \\
        For the second term, we write
        \begin{align}
            \int_{\Gamma}  & | \gamma |^p_{b^{\alpha,p}} \d \pi_n (\gamma)  = \int_{\bm{\mathcal{X}}_n}  | \ell_{\bm{x}} |^p_{b^{\alpha,p}} \d \Upsilon_n (\bm{x}) \label{eq:tightness_computation_step_0} \\
            & = \int_{\bm{\mathcal{X}}_n} \Big( \sum_{m=0}^{n} 2^{m(\alpha p -1)} \hspace{-20pt} \sum_{\substack{i = k \frac{2^n}{2^m} \\ k \in \{0,1,\cdots, 2^m-1\}}} \hspace{-20pt} d (x_{i},x_{i+\frac{2^n}{2^m}})^p + \frac{2^{n (\alpha p - 1 )}}{2^{(p - \alpha p)}-1}  \sum_{i=0}^{2^n-1} d(x_i,x_{i+1})^p \Big) \d \Upsilon_n 
            \label{eq:tightness_computation_step_1}\\
            & = \sum_{m=0}^{n} 2^{m(\alpha p -1)} \hspace{-20pt} \sum_{\substack{i = k \frac{2^n}{2^m} \\ k \in \{0,1,\cdots, 2^m-1\}}} \hspace{-20pt} \int d (x_{i},x_{i+\frac{2^n}{2^m}})^p \d \Upsilon_n + \frac{2^{n (\alpha p - 1 )}}{2^{(p - \alpha p)}-1}  \sum_{i=0}^{2^n-1} \int d(x_i,x_{i+1})^p \d \Upsilon_n  \label{eq:tightness_computation_step_2} \\
            & =  \sum_{m=0}^{n} 2^{m(\alpha p -1)} \hspace{-20pt} \sum_{\substack{i = k \frac{2^n}{2^m} \\ k \in \{0,1,\cdots, 2^m-1\}}} \hspace{-20pt} W_p^p( \mu_{t^{(n)}_i},\mu_{t^{(n)}_{i+2^n/2^m}})  + \frac{2^{n (\alpha p - 1 )}}{2^{(p - \alpha p)}-1}  \sum_{i=0}^{2^n-1} W_p^p (\mu_{t^{(n)}_{i}},\mu_{t^{(n)}_{i+1}}) \label{eq:tightness_computation_step_3} \\
            &  = \sum_{m=0}^{n} 2^{m(\alpha p -1)} \,\, \sum_{k=0}^{2^m-1} W_p^p( \mu_{t^{(m)}_{k}},\mu_{t^{(m)}_{k+1}})  + \frac{2^{n (\alpha p - 1 )}}{2^{(p - \alpha p)}-1}  \sum_{i=0}^{2^n-1} W_p^p (\mu_{t^{(n)}_{i}},\mu_{t^{(n)}_{i+1}}) \label{eq:tightness_computation_step_4}\\
             & \leq | \mu |^p_{b^{\alpha,p}} + \frac{1}{2^{(p - \alpha p)}-1} | \mu |^p_{b^{\alpha,p}} \label{eq:tightness_computation_step_5} 
             \\
             & = \frac{1}{ 1- 2^{-(p - \alpha p)}} \, | \mu |^p_{b^{\alpha,p}}  \label{eq:tightness_computation_step_6} \\
             & \leq \frac{c_2(\alpha,p)^p}{ 1- 2^{-(p - \alpha p)}} \, | \mu |^p_{W^{\alpha,p}} < + \infty. \label{eq:tightness_computation_step_7}
        \end{align}
        Let us clarify the computation steps. 
        After applying  the push-forward \eqref{eq:pi_n_Upsilon} in the first step \eqref{eq:tightness_computation_step_0}, we note that for fixed $n$ and $\bm{x}=(x_0,\cdots ,x_{N_n})$, the curve $t \mapsto \ell_{\bm{x}}(t) $ is piecewise geodesic.
        We have previously computed the $b^{\alpha,p}$-semi-norm of such curves in \eqref{eq:balphap_Xn}, which we now use to obtain \eqref{eq:tightness_computation_step_1}. 
        The compatibility property \eqref{eq:2D_marginal_Upsilon} is used to go from \eqref{eq:tightness_computation_step_2} to  \eqref{eq:tightness_computation_step_3}, where even the 2-D marginals of $\Upsilon_n$ that are separated by more than one step are given by the Wasserstein distance.  
        The inequality in  \eqref{eq:tightness_computation_step_5} is justified by extending the upper limit of the first sum  in \eqref{eq:tightness_computation_step_4} to infinity, which then corresponds exactly to the definition $|\mu|^p_{b^{\alpha,p}}$. The second sum is nothing but one summand in the definition of  $|\mu|^p_{b^{\alpha,p}}$, so it can be bounded accordingly.
        To bound \eqref{eq:tightness_computation_step_6}, we use the equivalence \eqref{eq:equiv_norm_Walphap} of the semi-norms $|\cdot|_{W^{\alpha,p}}$ and $|\cdot|_{b^{\alpha,p}}$.
        The last expression in \eqref{eq:tightness_computation_step_7} is indeed finite by assumption.
        \\
        In summary, the computations in  (\ref{eq:tightness_computation_step_00})-(\ref{eq:tightness_computation_step_6}) confirm that the bound \eqref{eq:tightness_computation_claim} holds and thus, the family of measures $\{\pi_n\}_{n\in \mathbb{N}} \subset P(\Gamma)$ is tight. Then Prokhorov theorem implies that the set  $\{\pi_n\}_{n\in \mathbb{N}}$ is relatively (sequentially) compact with respect to the narrow topolog of $P(\Gamma)$, i.e., there exists a subsequence $\{\pi_{n_k}\}_{k\in \mathbb{N}} $ such that $\pi_{n_k} \to \pi$ narrowly on $P(\Gamma)$ as $k \to \infty$ to a limit point $\pi \in P(\Gamma)$.
        \\
        \textbf{Step 2 ($\pi$ is concentrated on $W^{\alpha,p}(I;\mathcal{X})$).}
        Here we prove a property related to the support of the limit measure $\pi$.
        Since $(\pi_{n_k})_k \subset P(\Gamma)$ narrowly converges to $\pi \in P(\Gamma)$ and $\gamma \mapsto | \gamma |^p_{W^{\alpha,p}}$ is a lower semi-continuous function from $\Gamma \to [0,+\infty]$, we have
        \begin{equation}
             \int_{\Gamma}|\gamma|_{W^{\alpha, p}}^p \d \pi (\gamma) \leq \liminf_{k \to \infty } \int_{\Gamma}  | \gamma |^p_{W^{\alpha,p}} \d \pi_{n_k} (\gamma) < + \infty.
        \end{equation}
       Notice that the integral on the right-hand side has been shown, in the previous step, to be bounded independent of $n_k$.
       We conclude 
        \begin{equation}
            |\gamma|_{W^{\alpha, p}} < + \infty \quad \textrm{for } \pi\textrm{-a.e. } \gamma \in \Gamma.
        \end{equation}  
    
        \noindent
        \textbf{Step 3 ($(e_t)_\#\pi=\mu_t$ for all $t\in I$).}   
        We need to show that for any $t \in I$,
        \begin{equation}
            \int_{\Gamma} \phi(\gamma_t) \d \pi (\gamma) = \int_{\mathcal{X}}  \phi (x) \d \mu_t (x) 
        \end{equation}
        holds for any $\phi \in C_b(\mathcal{X})$.
        It is enough to show this only for bounded Lipschitz functions. Take $\phi \in \mathrm{Lip}_b (\mathcal{X})$. We have 
        \begin{align}
            \int_{\Gamma} \phi (\gamma_t) \d \pi (\gamma) \,
            & \overset{(a)}{=}  \lim_{k \to \infty}  \int_{\Gamma} \phi (\gamma_{[2^{n_{k}} t]/2^{n_{k}} })  \d \pi_{n_k} (\gamma) \\
            & \overset{(b)}{=}  \lim_{k \to \infty} \int_{\bm{\mathcal{X}}_{n_k}} \phi(x_{[2^{n_k} t]}) \d \Upsilon_{n_k} (\bm{x}) \\
            & \overset{(c)}{=} \lim_{k \to \infty}  \int_{\mathcal{X}} \phi (x) \d \mu_{[2^{n_k} t]/2^{n_k} } (x) \\
            & \overset{(d)}{=}  \int_{\mathcal{X}} \phi (x) \d \mu_t (x).
        \end{align}
        As for the step $(a)$, observe that
        \begin{multline}
            \bigg|  \int \phi (\gamma_t)  \d \pi   - \int \phi (\gamma_{[2^{n_k} t]/2^{n_k} } )  \d \pi_{n_k}  \bigg|  \\
             \leq  \bigg|  \int \phi (\gamma_t)  \d \pi   - \int \phi (\gamma_t)  \d \pi_{n_k}   \bigg| + \bigg|  \int \phi (\gamma_t)  \d \pi_{n_k}     - \int \phi (\gamma_{[2^{n_k} t]/2^{n_k} } )  \d \pi_{n_k}  \bigg|.
        \end{multline}
        By taking limit $k \to \infty$, the first term on the right-hand side goes to zero by the narrow convergence of $(\pi_{n_k})_k$ to $\pi$ (note that the map $\gamma \mapsto \phi (\gamma_t)$ from $\Gamma \to \mathbb{R} $ is indeed continuous and bounded).
        To show that the second term also vanishes in the limit, we further estimate
        \begin{align}    
            \bigg|  \int \phi (\gamma_t)  \d \pi_{n_k}     - \int \phi (\gamma_{[2^{n_k} t]/2^{n_k} } )  \d \pi_{n_k}  \bigg| & \leq \mathrm{Lip} (\phi) \int d (\gamma_t, \gamma_{[2^{n_{k}} t]/2^{n_{k}} } )  \d \pi_{n_k} \\
            & \leq \mathrm{Lip} (\phi)  \bar{c}(\alpha,p) \left(t- \frac{[2^{n_{k}} t]}{2^{n_{k}} }\right)^{\alpha - \frac{1}{p}} \left( \int | \gamma |^p_{W^{\alpha,p}} \d \pi_{n_k} \right)^{\frac1p}, \label{eq:proof_lift_computation_step_a2}
         \end{align}
         where $d (\gamma_t, \gamma_{[2^{n_{k}} t]/2^{n_{k}} } )$ is estimated using Garsia--Rodemich--Rumsey inequality as described in \eqref{eq:distance_estimate_GRR} and we used Jensen's inequality for the integral. Note that thanks to the calculations in the proof of tightness, we now know that $\int | \gamma |^p_{W^{\alpha,p}} \d \pi_{n_k} $ is uniformly finite, which means that $\pi_{n_k}$-a.e. $\gamma$ has finite $W^{\alpha,p}$-semi-norm.
         In the limit, the last expression above approaches zero since 
         $$
         \frac{[N t]}{N} \to t \quad \mathrm{as} \quad N \to \infty. 
         $$
         Steps $(b)$-$(c)$ simply follow from the construction. More precisely, for fixed $k$, the 1-D marginals of $\pi_{n_k}$ at time $t$ coincide with $\mu_t$ whenever $t$ is of the form $t = \frac{i}{2^{n_k}}$ for some integer $i$. 
        \\
        Finally, step $(d)$ follows from the fact that $t \mapsto \mu_t$ is a continuous curve in $P_p(\mathcal{X})$, in particular, a narrowly continuous curve in $P(\mathcal{X})$. 
         \\
        \textbf{Step 4 ($(e_s,e_t)_\# \pi \in \mathrm{OptCpl}(\mu_s, \mu_t)$ for all $s,t \in I$).}
        For this claim to hold, the compatibility assumption is important.
        Given the previous step, we know that $(e_s,e_t)_\# \pi \in \mathrm{Cpl}(\mu_s, \mu_t)$ for all $t,s \in I$.
        This provides us with an estimate for the Wasserstein distance
        \begin{equation}\label{eq:proof_opti_inequ1}
            W_p^p ( \mu_t , \mu_s ) \leq  \int_\Gamma d (\gamma_t,\gamma_s)^p \d \pi (\gamma).
        \end{equation}
        To demonstrate the reverse inequality, we write 
        \begin{align}
            \int_\Gamma d (\gamma_t,\gamma_s)^p \d \pi (\gamma)
            & \overset{(a)}{\leq }  \liminf_{m \to \infty}  \int_{\Gamma}  d (\gamma_{[2^m t]/2^m}, \gamma_{[2^m s]/2^m} )^p  \d \pi (\gamma) \\
             & \overset{(b)}{\leq }  \liminf_{m \to \infty} \left( \liminf_{k \to \infty}  \int_{\Gamma} d (\gamma_{[2^m t]/2^m}, \gamma_{[2^m s]/2^m} )^p  \d \pi_{n_k} (\gamma) \right) \\
             & \overset{(c)}{=} \liminf_{m \to \infty}  W_p^p \big( \mu_{[2^m t]/2^m} , \mu_{[2^m s]/2^m} \big) \\
            & \overset{(d)}{=}  W_p^p ( \mu_t , \mu_s ),
            \label{eq:proof_opti_inequ2}
        \end{align}
        where in step $(a)$, we applied Fatou's lemma on the measure space $(\Gamma, \mathcal{B}(\Gamma), \pi)$. Notice that for every continuous curve $\gamma \in \Gamma$, we have 
        \begin{equation}
            d (\gamma_t,\gamma_s)^p = \lim_{m \to \infty} d (\gamma_{[2^m t]/2^m}, \gamma_{[2^m s]/2^m} )^p .
        \end{equation}
        Regarding step $(b)$, fix $m$ in the integrand and note that the map 
        $$
        \gamma \mapsto d (\gamma_{[2^m t]/2^m}, \gamma_{[2^m s]/2^m} )^p
        $$
        from $\Gamma \to [0, + \infty)$ is continuous (as explained in \eqref{eq:pointwise_conv_dst}), in particular, lower-semi continuous.
        Therefore, this step follows from the narrow convergence of $(\pi_{n_k})_k$ to $\pi$ (recall \eqref{eq:narrow_conv_lsc_func}).
        \\
        Step $(c)$ is a direct consequence of the construction of $\pi_{n_k}$ and the compatibility assumption.
        \\
        Finally, step $(d)$ is due to the fact that the curve $\mu:[0,1] \to P_p(\mathcal{X})$ is a continuous curve.
        \\
        To summarize, \eqref{eq:proof_opti_inequ1} and \eqref{eq:proof_opti_inequ2} give us equality. 
        \\
        \textbf{Step 5 ($|\mu|_{W^{\alpha,p}}^p = \int |\gamma|_{W^{\alpha,p}}^p \d \pi $).}
        By assumption, we have $|\mu|_{W^{\alpha,p}} < \infty$.
        Using the optimality of $\pi$, as shown in the previous step, we can compute
        \begin{align}
             |\mu|_{W^{\alpha,p}}^p
             & \coloneqq \iint_{[0,1]^2} \frac{W_p^p (\mu_s,\mu_t)}{|t-s|^{1+\alpha p}} \d s \d t \\
             & = \iint_{[0,1]^2} \frac{\int d(\gamma_s,\gamma_t)^p \d \pi (\gamma)}{|t-s|^{1+\alpha p}} \d s \d t \\
             & = \int \left( \iint_{[0,1]^2} \frac{ d(\gamma_s,\gamma_t)^p}{|t-s|^{1+\alpha p}} \d s \d t \right) \d \pi (\gamma) \\
             & = \int |\gamma|_{W^{\alpha,p}}^p \d \pi(\gamma), \label{eq:proof_opti_Wap_norm}
         \end{align}
        where we used Tonelli's theorem. Similarly, we have
         \begin{align}
             |\mu|_{b^{\alpha,p}}^p
             & \coloneqq \sum_{m=0}^{\infty} 2^{m(\alpha p -1)} \sum_{k=0}^{2^m-1} W_p^p (\mu_{t_{k}^{(m)}}, \mu_{t_{k+1}^{(m)}}) \\
             & = \lim_{M \to \infty} \int \sum_{m=0}^{M} 2^{m(\alpha p -1)} \sum_{k=0}^{2^m-1} d (\gamma_{t_{k}^{(m)}}, \gamma_{t_{k+1}^{(m)}})^p \d \pi (\gamma) \\
             & = \int \sum_{m=0}^{\infty} 2^{m(\alpha p -1)} \sum_{k=0}^{2^m-1} d (\gamma_{t_{k}^{(m)}}, \gamma_{t_{k+1}^{(m)}})^p \d \pi (\gamma) \\
             & = \int |\gamma|_{b^{\alpha,p}}^p \d \pi (\gamma) \label{eq:proof_opti_bap_norm}
         \end{align}
         where we used Beppo Levi's lemma. This completes the proof.
    \end{proof}
    
    An immediate consequence of the previous theorem and the embeddings
    \begin{align}
        C^{\upgamma\textrm{-}\mathrm{H\ddot{o}l}}  \subset W^{\alpha, p}  \subset C^{\alpha - \frac{1}{p}\textrm{-}\mathrm{H\ddot{o}l}}, \qquad 
        W^{\alpha, p}   \subset C^{\frac{1}{\alpha}\textrm{-}\mathrm{var}}
    \end{align}
    for $\frac{1}{p}< \alpha < \upgamma \leq  1$ is the obtaining a lift for $\upgamma$-H\"{o}lder compatible paths in $p$-Wasserstein space:     \begin{corollary}\label{crl:optimal_lift_mu_Holder_compatible}
        Let $(\mathcal{X}, d)$ be a complete, separable, and locally compact length metric space, and $I \coloneqq [0,T] \subset \mathbb{R}$.
        Let $(\mu_t) \in C^{\upgamma\textrm{-}\mathrm{H\ddot{o}l}} (I;P_p(\mathcal{X}))$ for some $1<p<\infty$ and $\frac{1}{p}<\upgamma \leq 1$.
        Assume that $(\mu_t)_{t \in I}$ is  compatible in $P_p(\mathcal{X})$.
        Then, construction \hyperref[itm:ConstructionB]{{\normalfont \textcircled{\small{B}}}} converges narrowly (up to a subsequence) to a probability measure $\pi \in P(C(I;\mathcal{X}))$ satisfying
        \begin{enumerate}[label=(\roman*), font=\normalfont]
            \item $\pi$ is concentrated on $W^{\alpha,p}(I;\mathcal{X}) \subset C^{(\alpha - \frac{1}{p})\textrm{-} \mathrm{H\ddot{o}l}}(I;\mathcal{X})$ for any $\alpha \in (\frac{1}{p},\upgamma)$;
            \item $(e_t)_\#\pi=\mu_t$ for all $t\in I$;
            \item $(e_s,e_t)_\# \pi \in \mathrm{OptCpl}(\mu_s, \mu_t)$ for all $s,t \in I$; and for any $\alpha \in (\frac{1}{p},\upgamma)$, we have  \eqref{eq:optimality_pi_1} and
            \begin{align}\label{eq:optimal_lift_mu_Holder_compatible}
               | \mu |_{\upgamma\textrm{-}\mathrm{H\ddot{o}l}}^p \geq c \int_{\Gamma_T} | \gamma |^p_{\alpha - \frac{1}{p}\textrm{-}\mathrm{H\ddot{o}l}} \d \pi (\gamma) \geq c | \mu |_{\alpha-\frac{1}{p}\textrm{-}\mathrm{H\ddot{o}l}}^p,
            \end{align}
            where $c = c(\upgamma,\alpha,p,T)$ is an explicit positive constant.
        \end{enumerate}
    \end{corollary}
    
    \begin{remark}
        In addition to the estimate above, we have for any $\alpha \in (\frac{1}{p},\upgamma)$: 
        \begin{equation}
             | \mu |_{\upgamma\textrm{-}\mathrm{H\ddot{o}l}}^p \geq c \, \int_{\Gamma_T} | \gamma |^p_{\frac{1}{\alpha} \textrm{-}\mathrm{var}} \d \pi (\gamma)  \geq c \, | \mu |_{p\textrm{-}\mathrm{var}}^p,
        \end{equation}
        where $c = c(\upgamma,\alpha,p,T)$ is another explicit positive constant.
    \end{remark}

    \begin{proof}
    Only the inequalities need to be shown.
    Take an arbitrary $\alpha \in (\frac{1}{p},\upgamma)$ and apply the previous results as follows:
        \begin{multline}
            | \mu |_{\upgamma\textrm{-}\mathrm{H\ddot{o}l}}^p \overset{\text{Rem. \ref{rmk:Holder-FractionalSobolev}}}{\geq} \tilde{c} |\mu|_{W^{\alpha,p}}^p \overset{\text{Thm. \ref{thm:optimal_lift_mu_Walphap_compatible}}}{=} \tilde{c} \int_{\Gamma_T} |\gamma|_{W^{\alpha,p}}^p \d \pi (\gamma) \\ \overset{\text{Thm. \ref{thm:FractionalSobolev-Holder}}}{\geq} \frac{\tilde{c}}{{\bar{c}\,}^p} \int_{\Gamma_T} | \gamma |^p_{\alpha - \frac{1}{p}\textrm{-}\mathrm{H\ddot{o}l}} \d \pi (\gamma) \overset{\text{Thm. \ref{thm:lift_to_mu_Hol}}}{\geq}  \frac{\tilde{c}}{{\bar{c}\,}^p} | \mu |_{\alpha-\frac{1}{p}\textrm{-}\mathrm{H\ddot{o}l}}^p.
            \end{multline}
            where 
            $$
            \tilde{c} \coloneqq \frac{(\upgamma p - \alpha p)(\upgamma p - \alpha p+1)}{2 T^{(\upgamma p - \alpha p+1)}}, \quad \text{and} \quad {\bar{c}\,}^p \coloneqq 32 \frac{\alpha p +1}{\alpha p -1}. 
            $$
            Similarly, we have
            \begin{align}
            \tilde{c} \,  \int_{\Gamma_T} |\gamma|_{W^{\alpha,p}}^p \d \pi (\gamma) \overset{\text{Thm.\ref{thm:FractionalSobolev-Holder}}}{\geq} \frac{\tilde{c}}{{\bar{c}}^p} \frac{1}{T^{\alpha p -1}}  \int_{\Gamma_T} | \gamma |^p_{\frac{1}{\alpha}\textrm{-}\mathrm{var}} \d \pi (\gamma) \overset{\text{Thm.\ref{thm:lift_to_mu_var}}}{\geq}   \frac{\tilde{c}}{{\bar{c}}^p} \frac{1}{T^{\alpha p -1}} | \mu |_{p\textrm{-}\mathrm{var}}^p.
            \end{align}
    \end{proof}

\subsection{Characterization of fractional Sobolev compatible curves in Wasserstein spaces}\label{subsec:from_mu_to_pi_and_back}
    By combining \cref{thm:lift_to_mu_Walphap,prop:equality_implies_compatibility,thm:optimal_lift_mu_Walphap_compatible}, we immediately obtain the following characterization:
    \begin{corollary}\label{crl:characterization_Walphap_compatible}
        Let $(\mathcal{X}, d)$ be a complete, separable, and locally compact length metric space, and $I \coloneqq [0,T] \subset \mathbb{R}$. Let $(\mu_t)_{t \in I} \subset P (\mathcal{X})$ with $\mu_0 \in P_p(\mathcal{X})$ and let  $1<p<\infty$ and $\frac{1}{p}< \alpha <  1$. Then the following are equivalent:
    \begin{itemize}
        \item[1.] $(\mu_t) \in  W^{\alpha,p}(I;P_p(\mathcal{X}))$ and $(\mu_t)$ is compatible in $P_p(\mathcal{X})$.
        \item[2.] $(\mu_t)$ has a lift $\pi \in P(C(I;\mathcal{X}))$ such that  $ |\mu|_{W^{\alpha,p}}^p = \int_{\Gamma_T} |\gamma|_{W^{\alpha,p}}^p \d \pi (\gamma) < + \infty.$
    \end{itemize}   
    \end{corollary}

\subsection{Characterization of geodesics via Besov regularity} On a metric space $(\mathcal{X},d)$ and for any $1<p<\infty$, it is well-known that the following are equivalent:
    \begin{itemize}
        \item[1.] $\gamma: [0,1] \to \mathcal{X}$ is a constant-speed geodesic.
        \item[2.] $\gamma: [0,1] \to \mathcal{X}$ is absolutely continuous and 
        $$ d(\gamma_0,\gamma_1)^p =  \int_{0}^{1} |\dot{\gamma}_t|^p \d t \,\, \big(= |\gamma|^p_{W^{1,p}}\big).$$
    \end{itemize}
    See e.g. \cite[Lemma 9.11]{ABS2021}.
    The goal of this section is to present a similar characterization using the Besov regularity $|\cdot|_{b^{\alpha,p}}$, which subsequently allows us to give a dynamic reformulation of the Wasserstein metric in the next section. We begin with a simple observation, which will be proven later with the main lemma.

    \begin{lemma}\label{lemma:bound_distance_balphap}
    Let $(\mathcal{X},d)$ be a metric space. Let $1\leq p < \infty$ and $0<\alpha<1$. For any (possibly discontinuous) $\gamma : [0,1] \to \mathcal{X}$ with $|\gamma|_{b^{\alpha, p}} < + \infty$, we have 
    \begin{equation}
        d(\gamma_0,\gamma_1)^p \leq \left(1-2^{-(p-\alpha p)}\right) |\gamma|^p_{b^{\alpha, p}}. 
    \end{equation}
    In particular, for any constant-speed geodesic $\gamma: [0,1] \to \mathcal{X}$, we have equality. 
    \end{lemma}

\begin{lemma}[A characterization of constant-speed geodesics via Besov regularity]\label{lemma:characterization_geodesics_balphap}
    Let $(\mathcal{X},d)$ be a metric space. Let $1< p < \infty$ and $0<\alpha<1$. Then the following are equivalent:
    \begin{itemize}
        \item[{\normalfont 1.}] $\gamma: [0,1] \to \mathcal{X}$ is a constant-speed geodesic.
        \item[{\normalfont 2.}] $\gamma: [0,1] \to \mathcal{X}$ is continuous and 
        \begin{equation}
        d(\gamma_0,\gamma_1)^p = \left(1-2^{-(p-\alpha p)}\right) |\gamma|^p_{b^{\alpha, p}}. 
    \end{equation}
    \end{itemize}
\end{lemma}

\begin{proof}[Proof of \cref{lemma:bound_distance_balphap,lemma:characterization_geodesics_balphap}]
    Suppose first $\gamma$ is a constant-speed geodesic. One simply computes
    \begin{align}
                | \gamma |^p_{b^{\alpha,p}} 
                & = d(\gamma_0,\gamma_1)^p \sum_{m=0}^{\infty} 2^{m  (\alpha p -  p )} = \frac{d (\gamma_0, \gamma_1)^p}{1-2^{-(p-\alpha p)}}.
                \end{align}
    Now suppose $\gamma$ is an arbitrary curve with $|\gamma|_{b^{\alpha, p}} < \infty$. Using our notation $t_k^{(m)} \coloneqq \frac{k}{2^m}$, we write
    \begin{align}
        | \gamma |^p_{b^{\alpha,p}}
        & =  \sum_{m=0}^{\infty} 2^{m(\alpha p -p)}  2^{m(p -1)}  \sum_{k=0}^{2^m-1} d \big(\gamma_{t_{k}^{(m)}}, \gamma_{t_{k+1}^{(m)}}\big)^p \\
        & \overset{(a)}{\geq } \sum_{m=0}^{\infty} 2^{m(\alpha p -p)}  \left(\sum_{k=0}^{2^m-1} d \big(\gamma_{t_{k}^{(m)}}, \gamma_{t_{k+1}^{(m)}}\big) \right) ^p \label{eq:balphap_geodesic_a}\\
        & \overset{(b)}{\geq } \sum_{m=0}^{\infty} 2^{m(\alpha p -p)}  d (\gamma_0, \gamma_1)^p \label{eq:balphap_geodesic_b} = \frac{ d (\gamma_0, \gamma_1)^p}{1-2^{-(p-\alpha p)}}.
    \end{align}
   For $p>1$, the inequality $(a)$ follows from the discrete \textrm{H\"{o}lder} inequality. In $(b)$, the triangle inequality is applied.
   Note that that $(a)$ is obviously an equality
    when $p=1$, requiring no estimation. 
    \\
    It remains to prove $2 \Rightarrow 1$ in the second lemma. Note that when $p>1$, we have equality in the discrete \textrm{H\"{o}lder} inequality $(a)$ if and only if $d (\gamma_{t_{k}^{(m)}}, \gamma_{t_{k+1}^{(m)}}) = c_{m}$ for some constant independent of $k$. This, together with the equality in the triangle inequality $(b)$, allows us to determine $c_m$ and moreover conclude that
    \begin{equation}
    d \big(\gamma_{t_{k}^{(m)}},\gamma_{t_{k'}^{(m)}}\big) =  |k' - k | \, |\Delta t _m | d(\gamma_0,\gamma_1) \qquad \forall k,k' \in \{0,1,\cdots, 2^m-1\},\,  m \in \mathbb{N}_0,
    \end{equation}
    where $\Delta t_m \coloneqq \frac{1}{2^m}$.
    In other words, we proved the constant-speed geodesic property on the set of dyadic points. To extend this for arbitrary $ 0 \leq s \leq  t \leq 1$, we use the continuity assumption: 
    \begin{align}
        d (\gamma_s,\gamma_t) 
        & = \lim_{m \to \infty } d \big( \gamma_{\frac{[2^m s]}{2^m}} , \gamma_{\frac{[2^m t]}{2^m}} \big) \\
        & = \lim_{m \to \infty } \left( \frac{[2^m t]}{2^m} - \frac{[2^m s]}{2^m} \right) d(\gamma_0,\gamma_1) \\
        & = \left(t -s \right) d(\gamma_0,\gamma_1). 
    \end{align}
\end{proof}

\subsection{Dynamic formulation of Wasserstein 
 distance via Besov energy}\label{subsec:fractional_Benamou-Brenier}

As a consequence of \cref{lemma:bound_distance_balphap,lemma:characterization_geodesics_balphap}, we obtain a generalization of the metric Benamou--Brenier formula to the fractional setting. Its proof is similar to that for $W^{1,p}$-energy (e.g. \cite[Theorem 9.13]{ABS2021}).

\begin{corollary}\label{crl:dynamic_formulation_Wp}
    Let $(\mathcal{X}, d)$ be a complete, separable, and geodesic metric space.  Let  $1<p<\infty$ and $ 0< \alpha <  1$. Then for every $\mu,\nu \in P_p(\mathcal{X})$, we have 
    \begin{multline}
        W_p^p(\mu,\nu) = \\ \left(1-2^{-(p-\alpha p)}\right)\min \left\{ \int_{\Gamma} |\gamma|_{b^{\alpha,p}}^p \d \pi (\gamma) : \quad \pi \in P(C([0,1];\mathcal{X})), \, (e_0)_\# \pi  = \mu, \, (e_1)_\# \pi  = \nu \right\}. 
    \end{multline}
    In addition, $\pi$ is a minimizer if and only if  $(e_0,e_1)_{\#} \pi \in \mathrm{OptCpl}(\mu,\nu)$ and  $\pi (Geo([0,1];\mathcal{X})) = 1$. 
\end{corollary}

\begin{proof}
    Let $\pi$ be an admissible path measure in the minimization above with $\int_{\Gamma} |\gamma|_{b^{\alpha,p}}^p \d \pi < + \infty$ and denote by $\mu_t \coloneqq (e_t)_{\#} \pi$ for all $t \in [0,1]$. By  \cref{lemma:bound_distance_balphap}, we have 
    \begin{equation}\label{eq:proof_dynamic_Wp_balphap}
        W_p^p (\mu,\nu) \leq \int_{\Gamma} d(\gamma_0,\gamma_1)^p \d \pi (\gamma) \leq  \left(1-2^{-(p-\alpha p)}\right) \int_{\Gamma} |\gamma|^p_{b^{\alpha, p}} \d \pi (\gamma). 
    \end{equation}
    On the other hand, it is easy to show, by direct construction, that a minimizer exists.  Let $\Upsilon \in \mathrm{OptCpl} (\mu,\nu)$  
     and set 
    \begin{equation}
             \pi^* \coloneqq (\ell)_{\#} \Upsilon \, \in P(C([0,1];\mathcal{X})), 
    \end{equation}
    where $\ell : \mathcal{X}\times\mathcal{X} \to Geo([0,1];\mathcal{X})$ is a $(\mathcal{B}(\mathcal{X}^2)_\Upsilon,\mathcal{C})$-measurable geodesic selection, whose existence is discussed in \cref{rmk:measurable_selection_geodesics}. Here, $\mathcal{B}(\mathcal{X}^2)_\Upsilon $ denotes the $\Upsilon$-completion of $\mathcal{B}(\mathcal{X}^2)$.
    \\
    Note that, by its very construction, $(e_0,e_1)_{\#} \pi^* \in \mathrm{OptCpl}(\mu,\nu)$, and $\pi^*$ is concentrated on $Geo([0,1];\mathcal{X})$. By \cref{lemma:characterization_geodesics_balphap} ($1 \Rightarrow 2$), we have 
    \begin{equation}
        W_p^p(\mu,\nu)  = \int_{Geo} d(\gamma_0,\gamma_1)^p \d \pi^* (\gamma) = \left(1-2^{-(p-\alpha p)}\right) \int_{Geo} |\gamma|^p_{b^{\alpha, p}} \d \pi^* (\gamma). 
    \end{equation}
    Finally, note that an admissible path measure $\pi$ is a minimizer if and only if the inequalities in  \eqref{eq:proof_dynamic_Wp_balphap} are equalities, which happens if and only if $(e_0,e_1)_{\#} \pi \in \mathrm{OptCpl}(\mu,\nu)$ and $\pi (Geo([0,1];\mathcal{X})) = 1$ as a consequence of  \cref{lemma:bound_distance_balphap} and \cref{lemma:characterization_geodesics_balphap} ($1 \Leftrightarrow 2$). 
\end{proof}

\section{Counterexamples}\label{sec:counterexamples_d}
\begin{example}[\textbf{A $\frac{1}{p}$-H\"{o}lder continuous/$p$-variation curve in a $p$-Wasserstein space without any lift on continuous paths}]\label{exp:p_Holder_p_Wasserstein}
    This elementary example is well-known in stochastic analysis, see e.g. \cite[Chapter I. Exercise (2.12)]{RevuzYor1999}, and it is studied in detail in \cite[Example 1.1]{AbediLiSchultz2024}.
    Consider the following collection of measures on $\mathcal{X} \coloneqq \mathbb{R}$:
    \begin{equation}
        \mu_t \coloneqq (1-t)\delta_0+  t  \delta_{1}, \qquad \forall t \in I \coloneqq [0,1].
    \end{equation}
    For any $p \in [1,\infty)$, we have 
    \begin{equation}
        W^p_p (\mu_s, \mu_t)  = |t-s|, \qquad \forall s,t \in I,
    \end{equation}
    which implies $(\mu_t) \in C^{\frac1p\textrm{-} \mathrm{H\ddot{o}l}} (I;P_p(\mathbb{R}))$
    and $(\mu_t) \in C^{p\textrm{-} \mathrm{var}} (I;P_p(\mathbb{R}))$ with $ |\mu|^p_{p\textrm{-}\mathrm{var}} = 1$ for all $p \in [1,\infty)$. This Wasserstein curve, however, has no lift on continuous paths.
    \\
    A stochastic process whose distribution is $\mu_t$ is given by the jump process $X_t \coloneqq \mathds{1}_{\{U \leq t\}} $, $t \in I$, where $U \sim \mathrm{Unif}[0,1]$ is a uniformly distributed random variable defined on some probability space. This yields
    \begin{equation}
        \mathbb{E}[|X_t - X_s|^p] = |t-s|, \qquad \forall s,t \in I,
    \end{equation}
    which clearly does not meet the Kolmogorov--\v Centsov continuity criterion. 
    \\
    This example shows the sharpness of the assumption $ 1/p < {\upgamma}$ in \cref{crl:optimal_lift_mu_Holder_compatible}.
\end{example}

\begin{example}[\textbf{A compatible absolutely continuous Wasserstein curve without any lift realizing its higher-order variations, H\"{o}lder regularity, or modulus of continuity}]\label{exp:modulus_Hold_var}
    In \cref{prop:equality_implies_compatibility}, we observed that the existence of a lift realizing $W^{\alpha,p}$-regularity of a Wasserstein curve $(\mu_t)$ imposes a condition on $(\mu_t)$ called compatibility.
    A similar natural question arises: \emph{what conditions are imposed if a lift realizes higher-order variations, H\"{o}lder regularity, or the modulus of continuity?} The goal of this example is to illustrate that this can easily fail to hold, highlighting that imposed conditions are very strong.
    \\
    Unlike the case $p=1$, where the 1-variation of $1$-Wasserstein curves can be realized by (optimal) lifts \cite{AbediLiSchultz2024}, this is not the case for $p$-variation of $p$-Wasserstein curves when $p>1$, even though in both cases, the norm is obtained by taking supremum over all partitions.
    The failure in the case $p>1$  arises because higher-order variations of different curves in the underlying space can potentially be achieved by \emph{different} sequences of partitions, in contrast with $p=1$, where the $1$-variation of curves can be achieved by the \emph{same} sequence of partitions with the shrinking mesh size.
    Recall \eqref{eq:var_limsupvar_1}, where $p=1$ is the only case that variation and its infinitesimal characterization coincide and capture the regularity ``locally''. Here, we observe that higher-order variations behave more like H\"{o}lder regularity and the modulus of continuity, rather than 1-variation.
    \\
    Consider the following collection of measures on $\mathcal{X} \coloneqq \mathbb{R}$:
    \begin{gather}
        \mu_t \coloneqq \frac{1}{2} \delta_{\gamma_t^1} + \frac{1}{2} \delta_{\gamma_t^2}, \qquad \forall t \in I \coloneqq [0,1],
    \end{gather}
    where $ \gamma_t^1 \coloneqq t$ and $\gamma_t^2 \coloneqq 3-|2t-1|$. Since these paths do not cross, we have $|\mathrm{Lift}(\mu_t)| = 1$. The only lift on continuous paths is:
    \begin{equation}
        \pi \coloneqq \frac{1}{2} \delta_{\gamma^1} + \frac{1}{2} \delta_{\gamma^2}.
    \end{equation}
    Obviously, two-dimensional marginals of $\pi$ give the optimal coupling at all pairs of times with respect to all $W_p$, $p\geq 1$, thus $(\mu_t)$ is compatible. By \cite[Theorem 3.3]{AbediLiSchultz2024} or direct computation,
    the $1$-variation of $(\mu_t)$ with respect to $1$-Wasserstein metric satisfies:
        \begin{equation}
        |\mu|_{1\textrm{-}\mathrm{var}} = \int|\gamma|_{1\textrm{-}\mathrm{var}} \d \pi .
        \end{equation}
    Now let us consider $p>1$. We have
    \begin{itemize}
        \item$|\gamma^1|^p_{p\textrm{-}\mathrm{var}} = 1 $ achieved by the partition $(0,1)$. 
        \item $|\gamma^2|^p_{p\textrm{-}\mathrm{var}} = 2 $ achieved by the partition $(0,\frac{1}{2},1)$. 
    \end{itemize}
    Moreover, note that any other partition yields a strictly smaller $p$-variation for these paths. 
    Now let $D=(t_i)$ be an arbitrary dissection of the interval $I$. $D$ is either $(0,1)$, or $(0,\frac{1}{2},1)$, or any other partition. In all three cases, we have
    \begin{align}
        \sum_{i} W_p^p (\mu_{t_{i}}, \mu _{t_{i+1}}) & =  \int \sum_{i} d(\gamma_{t_{i}}, \gamma _{t_{i+1}})^p \d\pi < \int |\gamma|_{p\textrm{-}\mathrm{var}}^p \d\pi.
    \end{align}
    Taking supremum over all possible partitions gives:
    \begin{equation}
            |\mu|_{p\textrm{-}\mathrm{var}}^p < \int|\gamma|_{p\textrm{-}\mathrm{var}}^p \d \pi(\gamma).
        \end{equation}
        By a similar argument or direct computation, one can conclude the same for H\"{o}lder regularities with $\theta \in (0,1]$,
        \begin{equation}
            |\mu|_{\theta\textrm{-}\mathrm{H\ddot{o}l}}^p < \int |\gamma|_{\theta\textrm{-}\mathrm{H\ddot{o}l}}^p \d \pi(\gamma),
        \end{equation}
        as well as for the modulus of continuity with $\delta =1$,
        \begin{equation}
            |\mu|_0^p <  \int |\gamma|_0^p \d \pi (\gamma).
        \end{equation}
        Thus, this compatible Wasserstein curve has no lift that can realize its higher-order variations, H\"{o}lder regularity, or the modulus of continuity. 
\end{example}

\begin{example}[\textbf{A non-absolutely continuous Wasserstein curve from superposition of absolutely continuous curves}]\label{exp:infinitesimal_variation} 
    The goal of this counterexample is twofold. First, we show that it is possible to construct a Wasserstein curve with low regularity by superposing regular curves. Second, we show that the energy of lifts with respect to the infinitesimal higher-order variation \eqref{eq:def_limsup_var}, fails to provide even a bound for the infinitesimal higher-order variation of the corresponding Wasserstein curve, as pointed out in \cref{subsec:from_pi_to_mu}.
    \\
    Specifically, here we construct a $\upgamma$-H\"{o}lder compatible curve $(\mu_t)$ in $p$-Wasserstein space ($1/p < \upgamma < 1$) by superimposing absolutely continuous curves in the underlying space. Here, $(\mu_t)$ can also be seen as the solution to the \emph{continuity equation}. 
    For $(\mu_t)$, there will be only \emph{one} lift $\pi$ on the space of continuous paths, so any construction should lead to this lift.
    Since $(\mu_t)$ forms a compatible collection, the lift here also enjoys the property that all of its two-dimensional marginals are optimal. 
    We show that $\frac{1}{\upgamma}$-variation of $(\mu_t)$ along the sequence of dyadic partitions has a non-zero limit, whereas this variation is zero for all curves in the underlying space (because they are of bounded 1-variation).  From this, we will conclude that
    \begin{equation}\label{eq:infinitesimal_variation_claim}
    |\mu|^{q}_{q\textrm{-}\mathrm{var}\textrm{-}\mathrm{limsup}} > \int_{\Gamma} |\gamma|^q_{q\textrm{-}\mathrm{var}\textrm{-}\mathrm{limsup}} \d \pi (\gamma) = 0, \qquad \text{where } q \coloneqq \frac{1}{\upgamma}. 
    \end{equation}
   
    \begin{figure}
        \centering
        \hspace{8pt}
        \begin{tikzpicture}
            \draw [->] (-0.4,0) -- (9,0);
            \draw [->] (0,0) -- (0,2.2);
            \draw (9,0) node[anchor=west] {\scriptsize $x$};
            \draw (0,2.2) node[anchor=east] {\scriptsize $y$};
            \draw (-0.05,1.2) -- (0.05,1.2);
            \draw (0,1.2) node[anchor=east] {\scriptsize $1$};
            \draw [|-|](2,1.2+0.25) -- (4,1.2+0.25);
            \draw (3,1.2+0.25) node[anchor=south] {\scriptsize $a > 1$};
            \draw [dotted] (0,1.2) -- (9,1.2);
            \draw [dotted] (2,0) -- (2,1.2);
            \draw [dotted] (4,0) -- (4,1.2);
            \draw [dotted] (6,0) -- (6,1.2);
            \draw [dotted] (8,0) -- (8,1.2);
            \draw [fill][blue] (0,0) circle [radius=4.00pt];
            \draw [fill][blue] (2,0) circle [radius=2.73pt];
            \draw [fill][blue] (4,0) circle [radius=1.87pt];
            \draw [fill][blue] (6,0) circle [radius=1.27pt];
            \draw [fill][blue] (8,0) circle [radius=0.87pt];
            \draw (-0.5,0) node[anchor=east] {\scriptsize \color{blue}$ \qquad\, \mu_0 \in P(\mathbb{R}^2)$ \, };
            \draw (0,-0.1) node[anchor=north ] {\scriptsize $j=0$};
            \draw (2,-0.1) node[anchor=north ] {\scriptsize $j=1$};
            \draw (4,-0.1) node[anchor=north ] {\scriptsize $j=2$};
            \draw (6,-0.1) node[anchor=north ] {\scriptsize $j=3$};
            \draw (8,-0.1) node[anchor=north ] {\scriptsize $j=4$};
            \draw (10,-0.1) node[anchor=north ] {\scriptsize $j=\cdots$};
            \draw (0.2,-1) node[anchor=east] {\scriptsize speed of particles = };
            \draw (0.2,-1.5) node[anchor=east] {\scriptsize mass of particles = };
            \draw (0.2,-1) node[anchor=west] {\scriptsize $2^{j+1}$ };
            \draw (0.2,-1.5) node[anchor=west] {\scriptsize $w_j \, \coloneqq \bar{w} \, 2^{-jp\upgamma}$ };
        \end{tikzpicture}
        
        \vspace{20pt}
        \begin{tikzpicture}
            \draw [->] (0,0) -- (4.5,0);
            \draw [->] (0,0) -- (0,2.5);
            \draw (0,2.5) node[anchor=east] {\scriptsize $y$};
            \draw (4,0) node[anchor=north] {\scriptsize $1$};
            \draw (-1,1) node[anchor=east] {\scriptsize $j=0$};
            \draw (-0.05,2) -- (0.05,2);
            \draw (0,2) node[anchor=east] {\scriptsize $1$};
            \draw [-] (0,0) -- (2,2) -- (4,0);
            \draw [->] (0,0-3) -- (4.5,0-3);
            \draw [->] (0,0-3) -- (0,2.5-3);
            \draw (0,2.5-3) node[anchor=east] {\scriptsize $y$};
            \draw (4,0-3) node[anchor=north] {\scriptsize $1$};
            \draw (-1,1-3) node[anchor=east] {\scriptsize $j=1$};
            \draw (-0.05,2-3) -- (0.05,2-3);
            \draw (0,2-3) node[anchor=east] {\scriptsize $1$};
            \draw [-] (0,0-3) -- (1,2-3) -- (2,0-3) -- (3,2-3) -- (4,0-3);
            \draw [->] (0,0-6) -- (4.5,0-6);
            \draw [->] (0,0-6) -- (0,2.5-6);
            \draw (0,2.5-6) node[anchor=east] {\scriptsize $y$};
            \draw (-1,1-6) node[anchor=east] {\scriptsize $j=2$};
            \draw (-0.05,2-6) -- (0.05,2-6);
            \draw (0,2-6) node[anchor=east] {\scriptsize $1$};
            \draw [-] (0,0-6) -- (0.5,2-6) -- (1,0-6) -- (1.5,2-6) -- (2,0-6) -- (2.5,2-6) -- (3,0-6) -- (3.5,2-6) -- (4,0-6);
            \draw (4.5,0-6) node[anchor=west] {\scriptsize time};
            \draw (0,0-6) node[anchor=north] {\scriptsize $0$};
            \draw (4,0-6) node[anchor=north] {\scriptsize $1$};
            \draw [dotted] (1,0-6) -- (1,0-1);
            \draw [dotted] (2,0-6) -- (2,2);
            \draw [dotted] (3,0-6) -- (3,0-1);

            \draw [dotted] (0.5,2-6) -- (0.5,0-6);
            \draw [dotted] (1.5,2-6) -- (1.5,0-6);
            \draw [dotted] (2.5,2-6) -- (2.5,0-6);
            \draw [dotted] (3.5,2-6) -- (3.5,0-6);
            \draw [<->](2,0-6.25) -- (3,0-6.25);
            \draw (2.5,0-6.25) node[anchor=north] {\scriptsize $\frac{1}{2^j}$};
        \end{tikzpicture}
        \captionsetup{font=footnotesize}
        \caption{\cref{exp:infinitesimal_variation} yields a compatible $(\mu_t) \in C^{\upgamma \textrm{-} \mathrm{H\ddot{o}l}} ([0,1];P_p(\mathbb{R}^2))$ from superposition of absolutely continuous curves. \textbf{Top:} The initial measure $\mu_0$  consists of a countable family of particles indexed by $j$, whose mass $w_j$ decreases with respect to $j$. As time runs, the particles oscillate vertically between 0 and 1. The lighter a particle is, the faster it moves. Their $x$-coordinates, which are separated by a distance $a>1$, remain constant. \textbf{Bottom:} $y$-coordinates of the first three particles as a function of time, as defined in \eqref{eq:counterexample_yj}. Each curve has constant speed for a.e. $t \in [0,1]$.}
        \label{fig:counterexample}
    \end{figure}
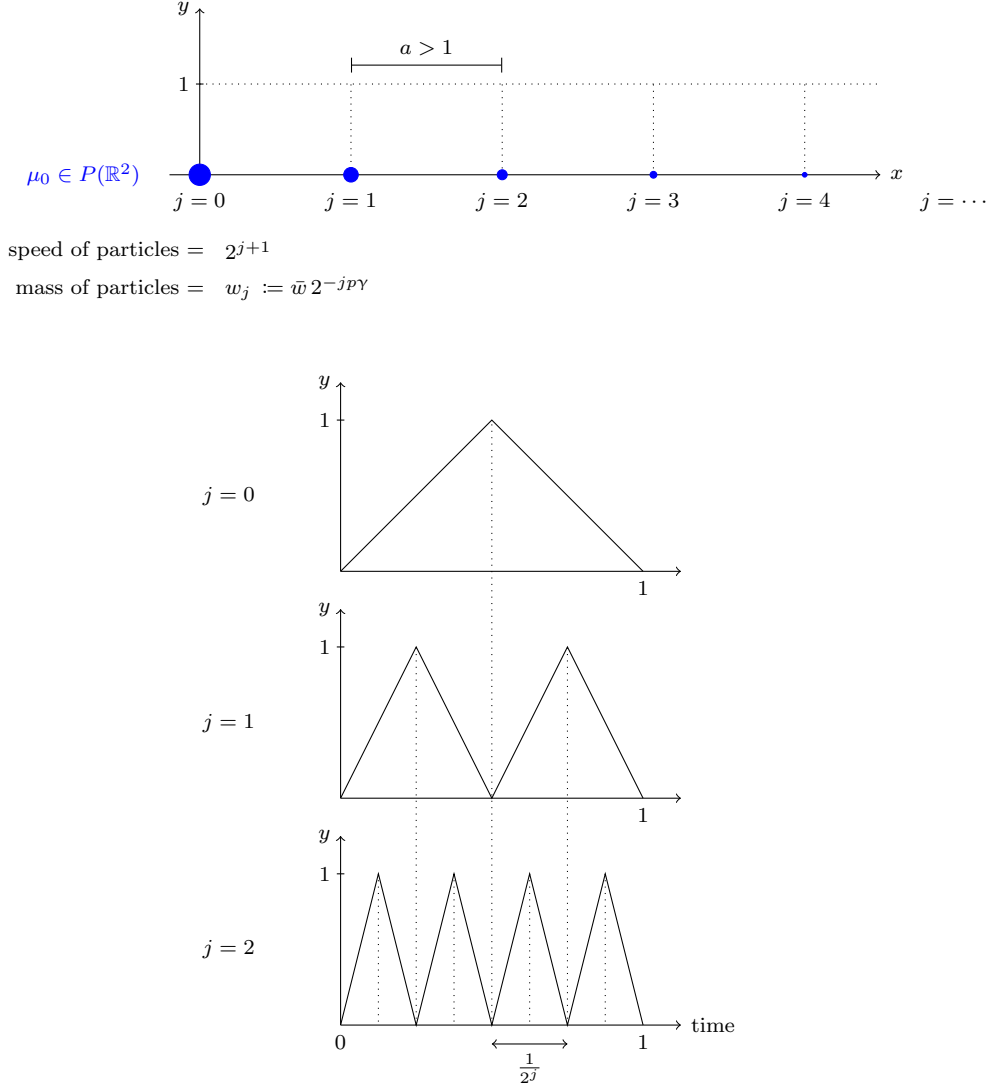  
    \noindent
    \\
    \textbf{Construction of $(\mu_t)$ and $\pi$.} Consider $\mathcal{X} \coloneqq \mathbb{R}^2$ equipped with the Euclidean distance.
    We introduce a countable family of particles indexed by $j \in \{0,1,2,\cdots\}$. Their trajectories $\gamma_t^j = (x^j_t , y^j_t )$ are continuous curves from $I \coloneqq [0,1] \to \mathbb{R}^2$, 
    with coordinates given by
    \begin{equation}
        \begin{cases}
            x_t^j \coloneqq j a  & t \in [0,1],\\
            y_t^j \coloneqq \sum_{l=0}^{2^j-1} \max \left\{1-2^{\left(j+1\right)}\left|t-\frac{1}{2^{\left(j+1\right)}}-\frac{l}{2^{\left(j\right)}}\right| , 0\right\} \qquad &  t \in [0,1],\label{eq:counterexample_yj} 
        \end{cases}
    \end{equation}
    for some $a>1$. 
    In other words, the $x$-coordinates of the particles, which are separated by a distance $a$, remain constant, whereas the $y$-coordinates oscillate between 0 and 1 with a constant speed (each $t \mapsto y^j_t$ consists of $2^j$ tent functions).
    The configuration of particles at time $t=0$ and the $y$-coordinates of the first three curves as a function of time are illustrated in \cref{fig:counterexample}. 
    The speed of each curve is given by
    \begin{equation}
        |\dot{\gamma}^j_t| = 2^{j+1} \quad \textrm{for } \textrm{a.e. } t \in I.
    \end{equation}
    We now assign a mass to each particle, denoted by $w_j$. Take $1<p<\infty$ and $\frac{1}{p}<\upgamma < 1$ and set
    \begin{equation}\label{eq:counterexample_mass}
        w_j \coloneqq \bar{w} \, 2^{-jp\upgamma}, \qquad \bar{w} (p, \upgamma)  \coloneqq 1 - 2^{-p \upgamma},
    \end{equation}
    where $\bar{w}$ is the normalization constant that makes the total mass equal to 1. According to this setting, the lighter a particle is, the faster it moves. We define
    a family of probability measure $(\mu_t) \subset P (\mathbb{R}^2)$ by 
    \begin{equation}\label{eq:counterexample_mu_t}
        \mu_t \coloneqq \sum_{j=0}^\infty w_j \delta_{\gamma^j_t}, \qquad \forall t \in I, 
    \end{equation}
    and a path measure
    $\pi \in P (C(I;\mathbb{R}^2))$ by
    \begin{equation}\label{eq:counterexample_pi}
        \pi \coloneqq  \sum_{j=0}^\infty w_j \delta_{\gamma^j}.
    \end{equation}
    It is clear that $\pi$ is a lift of $(\mu_t)$ and it is the only lift because the curves $(\gamma_t^j)$ do not cross each other nor do their mass $w_j$ vary.  
    It is worth bearing in mind that the entire construction depends on three parameters $ p, \upgamma, a$, which we fix in the range 
   \begin{align}
       1<p<\infty, \qquad \frac{1}{p}<\upgamma < 1, \qquad 1<a.
   \end{align}
    \\
    \textbf{Viewing $(\mu_t)$ as a solution to CE.}
   It is worth mentioning that $(\mu_t)$ can be seen as the solution to the continuity equation \eqref{eq:CE} with the initial condition $\mu_0 \in P(\mathbb{R}^2)$ and the velocity vector field $|v_t(x,y)| = 2^{x/a+1}$ satisfying the 1-integrability condition:        \begin{equation}\label{eq:counterexample_1_integrability_1}
            \int_{0}^{T} \int_{\mathbb{R}^{\mathrm{d}}} |v_t | \d \mu_t \d t  = \bar{w} 2 \sum_{j=0}^\infty 2^{j (1 - \upgamma p)} < + \infty.
        \end{equation}
    but not the $p$-integrability condition:
    \begin{equation}\label{eq:counterexample_1_integrability_p}
            \int_{0}^{T} \int_{\mathbb{R}^{\mathrm{d}}} |v_t |^p \d \mu_t \d t  = \bar{w} 2^p \sum_{j=0}^\infty 2^{j (p - \upgamma p)} = + \infty.
    \end{equation}
    \\
    \textbf{Claim 1 ($(\mu_t) \subset P_p(\mathbb{R}^2)$).} The $p$-th moments of the measures can be estimated by
    \begin{align} \label{eq:counterexample_1_moment}
        \int_{\mathbb{R}^2} \left| x^2+y^2\right|^{p/2} \d \mu_t (x,y)
        & \leq w_0 a^p +  \sum_{j=1}^\infty w_j \big(\sqrt{2} a j \big)^p = \bar{w} a ^p \Big( 1 + 2^{p/2}\sum_{j=1}^\infty j^p 2^{-jp\upgamma} \Big) < + \infty,
    \end{align}
    where the series converges because
    \begin{equation}
        \lim_{j \to \infty} \frac{(j+1)^p 2^{-(j+1)p\upgamma}}{j^p2^{-jp\upgamma}} = 2^{-p \upgamma} < 1. 
    \end{equation}
    \\
    \textbf{Claim 2 ($\int |\gamma|^p_{W^{1,p}} \d \pi = \infty$).}
    It follows from \eqref{eq:counterexample_1_integrability_p} that the $W^{1,p}$-energy of $\pi$ is infinite: 
    \begin{align}
        \int_{\Gamma} |\gamma|^p_{W^{1,p}} \d \pi (\gamma) & = \int_{\Gamma} \int_{0}^{1} |\dot{\gamma}_t|^p \d t \d \pi (\gamma) = + \infty.\label{eq:counterexample_W1p}
    \end{align}
    \\
    \textbf{Claim 3 ($\int |\gamma|^p_{W^{\alpha,p}} \d \pi < \infty $ for all $\alpha \in (\frac{1}{p}, \upgamma)$).}
    Here, we aim to show that, in contrast to the computed $W^{1,p}$-energy above, the $W^{\alpha,p}$-energy of the lift $\pi$ is finite.
    Let's begin by calculating $| \gamma^j |_{b^{\alpha,p}}$.
    Note that each $t \mapsto \gamma_t^j$ is a piece-wise geodesic curve, for which we have already computed the $b^{\alpha,p}$-norm in \cref{lemma:balphap_Xn}. Using this result with $n = j+1 $ and noticing that the only non-zero summand in the first sum in \eqref{eq:balphap_Xn} is for $m=n$, we obtain 
    \begin{align}
        | \gamma^j |_{b^{\alpha,p}}^p
          \, = \, 2^{(j+1)(\alpha p)} + \frac{2^{(j+1) (\alpha p - 1 )}}{2^{(p - \alpha p)}-1}  \sum_{i=0}^{2^{j+1}-1} 1^p  
         \, = \, \frac{1}{2^{- \alpha p} - 2^{- p} } \,  2^{j \alpha p}. 
    \end{align}
    Therefore, the $b^{\alpha,p}$-energy of the lift $\pi$ is given by  
    \begin{align}
         \int_{\Gamma} |\gamma|^p_{b^{\alpha,p}} \d \pi (\gamma)
          =   \sum_{j=0}^\infty w_j | \gamma^j |_{b^{\alpha,p}}^p   
          = \frac{\bar{w} }{2^{ - \alpha p} - 2^{-p} }  \sum_{j=0}^\infty 2^{j(\alpha p- \upgamma p )} < + \infty, \label{eq:counterexample_balphap}
    \end{align}
    where the geometric series converges only when $\alpha < \upgamma$.
    This also implies that $W^{\alpha,p}$-energy of the lift is finite by the equivalence of these norms in \cref{thm:Walphap_balphap}.
    This, together with $\mu_0 \in P_p(\mathbb{R}^2)$, allows us to apply \cref{thm:lift_to_mu_Walphap} to get 
    \begin{equation}
        |\mu|_{W^{\alpha,p}}^p \leq \int_\Gamma |\gamma|_{W^{\alpha,p}}^p \d \pi(\gamma) < + \infty. 
    \end{equation}
    In particular, $t \mapsto \mu_t$ is a continuous curve in $p$-Wasserstein space. 
    \\
    \textbf{Claim 4 (all two-dimensional marginals of $\pi$ are optimal when $a>1$).}
    As measures stay in the $p$-Wasserstein space, we can ask about the optimal plans between $\mu_s$ and $\mu_t$. 
    Let us fix $s,t$. 
    Note that the mass is always on the constant-$x$-lines, which are separated by a distance $a>1$, while $|y^j_t - y^j_s|\leq1$.
    It is, therefore, not optimal to transfer mass between $x$-lines, but rather keep it and move only vertically. 
    Accordingly, whenever $a>1$, $\pi$ is  optimal (for all $p\geq 1$) in the sense that
    \begin{equation}
        (e_s,e_t)_\# \pi \in \mathrm{OptCpl}(\mu_s, \mu_t) \quad \textrm{for all } s,t \in [0,1]. 
    \end{equation}
    The case $0<a\leq 1$ might be complicated and we do not claim anything about the optimality of $\pi$ in this case.
    The condition above also implies that $(\mu_t)$ is a compatible family of measures.
    As shown at the end of the proof of \cref{thm:optimal_lift_mu_Walphap_compatible}, we then have
    \begin{equation}
         |\mu|_{W^{\alpha,p}}^p = \int_\Gamma |\gamma|_{W^{\alpha,p}}^p \d \pi(\gamma),
    \end{equation}
    and the same equality holds for $|\cdot|_{b^{\alpha,p}}^p$.
    \\
    Given the optimality of $\pi$ and \eqref{eq:counterexample_W1p}, one can argue, by Lisini's  characterization of absolutely continuous curves \cite{Lisini2007}, 
    and the recent characterization of BV-curves in $1$-Wasserstein space \cite[Corollary 4.1]{AbediLiSchultz2024}, that $t \mapsto \mu_t$ is not an absolutely continuous curve in $p$-Wasserstein space, $p>1$, and it is not a BV-curve in 1-Wasserstein space, respectively.
    \\
    \textbf{Claim 5 ($(\mu_t) \in C^{\upgamma \textrm{-} \mathrm{H\ddot{o}l}} (I;P_p(\mathbb{R}^2))$).}
    Our aim is now to estimate $W_p(\mu_s,\mu_t)$ for any $s,t \in [0,1]$. Let us first consider the time points of the form 
    $$
    s = \frac{k}{2^m}, \quad t = \frac{k+1}{2^m},
    $$
    for some $m \in \mathbb{N}$ and $k \in \{ 0, 1, \cdots 2^m -1\}$. Notice that in the Wasserstein distance between the measures at these times, all curves $\gamma^j$ with $j\geq m$ have zero contribution. Thus, we have 
    \begin{align}
        W_p^p(\mu_{\frac{k}{2^m}}, \mu_{\frac{k+1}{2^m}}) & =  \bar{w}  \sum_{j=0}^{m-1} 2^{-jp\upgamma} \, d (\gamma^j_{\frac{k}{2^m}}, \gamma^j_{\frac{k+1}{2^m}})^p\\
        & = \bar{w}  \sum_{j=0}^{m-1} 2^{-jp\upgamma} \, \left| \frac{\Delta t_m }{\Delta t_{j+1} }\right|^p \\
        & = \frac{\bar{w} }{2^{ - \upgamma p} - 2^{-p} } \Big( | \Delta t_m |^{p \upgamma} - | \Delta t_m |^{p} \Big) \label{eq:counterexample_Wpp} \leq c^p  | \Delta t_m |^{p \upgamma}, \qquad 
    \end{align}
    where we denote by $c \coloneqq (\frac{\bar{w} }{2^{ - \upgamma p} - 2^{-p} })^{\frac1p}$ and use our usual notation $\Delta t_m \coloneqq \frac{1}{2^m}$ for any $m \in \mathbb{N}$.
    In other words, the H\"{o}lder condition is met on all dyadic time points. This, together with continuity of $(\mu_t)$, implies $\upgamma$-H\"{o}lder continuity by  
    \cref{thm:disc_char_Holder}. 
    \\
    \textbf{Claim 6 ($|\mu|^q_{q\textrm{-}\mathrm{var}\textrm{-}\mathrm{limsup}} > \int |\gamma|^q_{q\textrm{-}\mathrm{var}\textrm{-}\mathrm{limsup}} \d \pi$ for $q \coloneqq \frac1\upgamma >1$).}
    First of all, note that all $(\gamma^j)_{j \geq 0}$ are curves of bounded $1$-variation. As a result, \cref{lemma:limsupvar0} tells us that
    for all curves,
    $$
    | \gamma^j|_{q\textrm{-}\mathrm{var}\textrm{-}\mathrm{limsup}} = 0 \quad \textrm{for all } q >1. 
    $$
    Consequently, the integral on the right-hand side of the claim \eqref{eq:infinitesimal_variation_claim} is zero. 
    Second, based on $(\mu_t) \in C^{\upgamma \textrm{-} \mathrm{H\ddot{o}l}} (I;P_p(\mathbb{R}^2))$ and $\frac1\upgamma < p$, we immediately conclude:
    \begin{align}
        &| \mu |_{\frac1\upgamma\textrm{-}\mathrm{var}} < \infty, \qquad
         | \mu |_{\frac1\upgamma\textrm{-}\mathrm{var}\textrm{-}\mathrm{limsup}} < \infty, \\
        &| \mu |_{p\textrm{-}\mathrm{var}} < \infty, \qquad \,
        | \mu |_{p\textrm{-}\mathrm{var}\textrm{-}\mathrm{limsup}} = 0,
    \end{align}
    where we use the preliminary  \cref{lemma:infinitesimal_char_var,lemma:limsupvar0} and the embedding \eqref{eq:var_p_q_embedding}.
    All variations here are based on $p$-Wasserstein distance. 
    Now, we claim that $| \mu |_{1/\upgamma\textrm{-}\mathrm{var}\textrm{-}\mathrm{limsup}} > 0 $.
    It is enough to demonstrate that the variation of $(\mu_t)$ over a sequence of partitions whose mesh size goes to zero is strictly positive. We take the dyadic partition $D_m$ with mesh size $\Delta t_m \coloneqq \frac{1}{2^m}$, over which we computed the Wasserstein distance in \eqref{eq:counterexample_Wpp}. For fixed $m \in \mathbb{N}$, the $\frac1\upgamma$-variation is
    \begin{align}
        \sum_{t_i \in D_m} W_p^{\frac1\upgamma} (\mu_{t_i}, \mu_{t_{i+1}})
        & = \sum_{t_i \in D_m}  c^{\frac1\upgamma}  \Big( | \Delta t_m |^{p \upgamma} - | \Delta t_m |^{p} \Big)^{\frac{1}{p \upgamma}} = c^{\frac1\upgamma} \left( 1 - 2^{m p ( \upgamma - 1 )}\right)^{\frac{1}{p \upgamma}}.
    \end{align}
    Taking the limit yields 
    \begin{equation}
        \lim_{m \to \infty } \sum_{t_i \in D_m} W_p^{\frac1\upgamma} (\mu_{t_i}, \mu_{t_{i+1}}) = c^{\frac1\upgamma} > 0. 
    \end{equation}
\end{example}

\begin{example}[\textbf{Failure of compatibility under rotation and change of regularity under splitting}] \label{exp:compatibility_vs_noncompatibility}
        Consider $\mathcal{X} \coloneqq \mathbb{S}^1$ equipped with its arc length. Here, we consider the perimeter of the circle $\mathbb{S}^1$ to be 2. Take any $p \in [1,\infty)$ and let us start with the simplest constant-speed $p$-Wasserstein geodesic:
       \begin{align}\label{eq:compatibility_mu}
           \mu_t \coloneqq \delta_{e^{i \pi t}}, \qquad \forall t \in [0,1],
       \end{align}
       consisting of one particle moving along the circle with constant speed 1, as shown in \cref{fig:noncompatibility} (left), which also shows its distance $t \mapsto W_p(\mu_0,\mu_t) = t$ from the initial state. This collection of measures is obviously compatible.
       Now let us split the particle into two particles with equal mass, locate them on a diagonal, and again, let them rotate with constant speed 1. Continuing this splitting procedure, arranging the particles equidistantly, and letting them rotate always with constant speed 1 will yield the following:     \begin{align}\label{eq:noncompatibility_mu_j}
           \mu^j_t \coloneqq \frac{1}{2^{j+1}} \sum_{k=0}^{2^{j+1}-1} \delta_{e^{i\pi (t+k/2^j)}}, \qquad  \forall t \in [0,1],
       \end{align}
       where $j \in \{0,1,2,\cdots \}$. \cref{fig:noncompatibility} (right) shows the initial configurations for the first three indices.
       Unlike the original curve $(\mu_t)$, the curves $(\mu_t^j)$ are no longer geodesics but rather piece-wise geodesics oscillating between two states in the Wasserstein space with distance $\frac{1}{2^{j+1}}$. The time taken to transition from one state to the other is also equal to $\frac{1}{2^{j+1}}$, as illustrated by the graphs $t \mapsto W_p(\mu^j_0,\mu^j_t)$ in \cref{fig:noncompatibility}. Note that the metric speed in $p$-Wasserstein space for all $p\geq 1$ does not change under the splitting and stays 1 almost all times. 
       \\
       Denote by $\pi$ and $\pi^j$ the only existing lifts of $(\mu_t)$ and $(\mu_t^j)$ on the space of continuous paths, respectively. 
       The main observation here is that, while this splitting procedure doesn't change the energy of the lifts (with respect to any norm), it may change the regularity of Wasserstein curves with respect to certain norms. Let us first start with an observation on compatibility.
       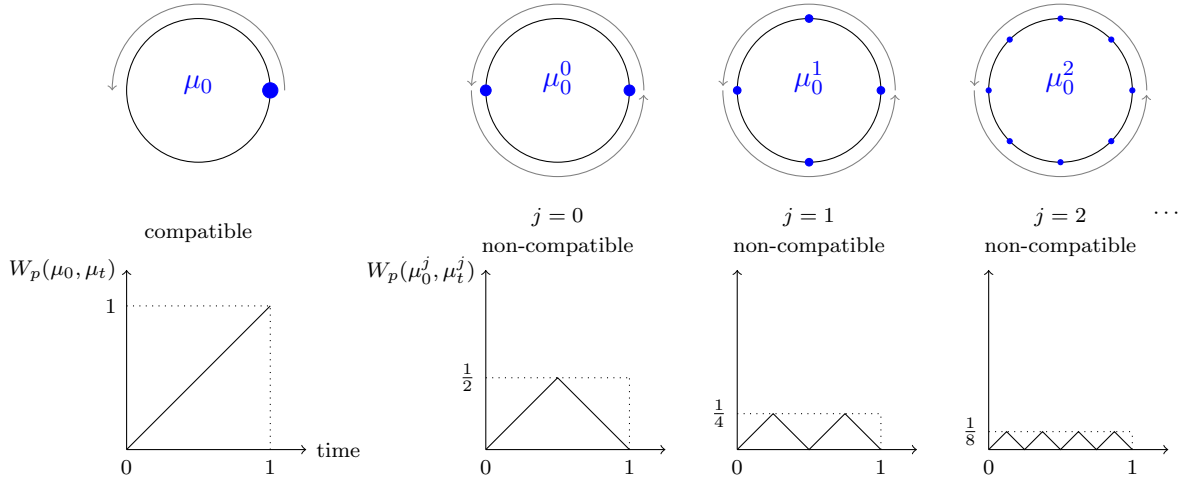
\begin{figure}
            \centering
            \vspace{10pt}
            \begin{tikzpicture}[scale=0.95]
                \draw [-] (0,0) circle [radius=1];
                \draw [-] (5,0) circle [radius=1];
                \draw [-] (8.5,0) circle [radius=1];
                \draw [-] (12,0) circle [radius=1];
                \draw (13.5,-1.5) node[anchor=north] {\scriptsize $\cdots$};
            \draw (0,-0.19) node[anchor=south] {\color{blue} $\mu_0$};
            \draw (5,-0.19) node[anchor=south] {\color{blue} $\mu^0_0$};
            \draw (8.5,-0.19) node[anchor=south] {\color{blue} $\mu^1_0$};
            \draw (12,-0.19) node[anchor=south] {\color{blue} $\mu^2_0$};
                \draw [fill][blue] (1,0) circle [radius=3pt];
                \draw [fill][blue] (1+5,0) circle [radius=2.12pt];
                \draw [fill][blue] (-1+5,0) circle [radius=2.12pt];
                \draw [fill][blue] (1+8.5,0) circle [radius=1.5pt];
                \draw [fill][blue] (-1+8.5,0) circle [radius=1.5pt];
                \draw [fill][blue] (0+8.5,1) circle [radius=1.5pt];
                \draw [fill][blue] (0+8.5,-1) circle [radius=1.5pt];
                \draw [fill][blue] (1+12,0) circle [radius=1pt];
                \draw [fill][blue] (-1+12,0) circle [radius=1pt];
                \draw [fill][blue] (0+12,1) circle [radius=1pt];
                \draw [fill][blue] (0+12,-1) circle [radius=1pt];
                \draw [fill][blue] ({cos(45)+12},{sin(45)}) circle [radius=1pt];
                \draw [fill][blue] ({cos(135)+12},{sin(135)}) circle [radius=1pt];
                \draw [fill][blue] ({cos(225)+12},{sin(225)}) circle [radius=1pt];
                \draw [fill][blue] ({cos(315)+12},{sin(315)}) circle [radius=1pt];
                \draw[->,gray] (1.2,0) arc [radius=1.2, start angle=0, end angle= 180];
                \draw[->,gray] (1.2+5,0) arc [radius=1.2, start angle=0, end angle= 177];
                \draw[->,gray] (-1.2+5,0) arc [radius=1.2, start angle=180, end angle= 357];
                \draw[->,gray] (1.2+8.5,0) arc [radius=1.2, start angle=0, end angle= 177];
                \draw[->,gray] (-1.2+8.5,0) arc [radius=1.2, start angle=180, end angle= 357];
                \draw[->,gray] (1.2+12,0) arc [radius=1.2, start angle=0, end angle= 177];
                \draw[->,gray] (-1.2+12,0) arc [radius=1.2, start angle=180, end angle= 357];
                \draw (5,-1.5) node[anchor=north] {\scriptsize $j=0$};
                \draw (8.5,-1.5) node[anchor=north] {\scriptsize $j=1$};
                \draw (12,-1.5) node[anchor=north] {\scriptsize $j=2$};
                \draw (0,-1.7) node[anchor=north] {\scriptsize compatible};
                \draw (5,-1.9) node[anchor=north] {\scriptsize non-compatible};
                \draw (8.5,-1.9) node[anchor=north] {\scriptsize non-compatible};
                \draw (12,-1.9) node[anchor=north] {\scriptsize non-compatible};
                \draw [->] (-1,-5) -- (-1+2.5,-5);
                \draw [->] (-1,-5) -- (-1,-5+2.5);
                \draw [->] (-1+5,-5) -- (-1+5+2.5,-5);
                \draw [->] (-1+5,-5) -- (-1+5,-5+2.5);
                \draw [->] (-1+8.5,-5) -- (-1+8.5+2.5,-5);
                \draw [->] (-1+8.5,-5) -- (-1+8.5,-5+2.5);
                \draw [->] (-1+12,-5) -- (-1+12+2.5,-5);
                \draw [->] (-1+12,-5) -- (-1+12,-5+2.5);
                \draw (-1,-5+2.5) node[anchor=east] {\scriptsize $W_p(\mu_0,\mu_t)$};
                 \draw (-1+5,-5+2.5) node[anchor=east] {\scriptsize $W_p(\mu^j_0,\mu^j_t)$};
                \draw (-1+2.5,-5) node[anchor=west] {\scriptsize time};
                \draw [-,xshift=-1 cm,yshift=-5cm] (0,0) -- (2,2);
                \draw [-,xshift=4 cm,yshift=-5cm] (0,0) -- (1,1) -- (2,0);
                \draw [-,xshift=7.5 cm,yshift=-5cm] (0,0) -- (0.5,0.5) -- (1,0) -- (1.5,0.5) -- (2,0);
                \draw [-,xshift=11 cm,yshift=-5cm] (0,0) -- (0.25,0.25) -- (0.5,0) -- (0.75,0.25) -- (1,0) -- (1.25,0.25) -- (1.5,0) -- (1.75,0.25) -- (2,0);
                \draw [dotted,xshift=-1 cm,yshift=-5cm] (2,0) -- (2,2);
                \draw [dotted,xshift=-1 cm,yshift=-5cm] (0,2) -- (2,2);
                \draw [dotted,xshift=4 cm,yshift=-5cm] (0,1) -- (2,1);
                \draw [dotted,xshift=4 cm,yshift=-5cm] (2,0) -- (2,1);
                \draw [dotted,xshift=7.5 cm,yshift=-5cm] (0,0.5) -- (2,0.5);
                \draw [dotted,xshift=7.5 cm,yshift=-5cm] (2,0) -- (2,0.5);
                \draw [dotted,xshift=11 cm,yshift=-5cm] (0,0.25) -- (2,0.25);
                \draw [dotted,xshift=11 cm,yshift=-5cm] (2,0) -- (2,0.25);
            \draw (-1,-5) node[anchor=north] {\scriptsize $0$};
            \draw (-1+2,-5) node[anchor=north] {\scriptsize $1$};
            \draw (-1,-5+2) node[anchor=east] {\scriptsize $1$};
            \draw (-1+5,-5+1) node[anchor=east] {\scriptsize $\frac12$};
            \draw (-1+8.5,-5+0.5) node[anchor=east] {\scriptsize $\frac14$};
            \draw (-1+12,-5+0.25) node[anchor=east] {\scriptsize $\frac18$};
                \draw (-1+5,-5) node[anchor=north] {\scriptsize $0$};
                \draw (-1+2+5,-5) node[anchor=north] {\scriptsize $1$};
                \draw (-1+8.5,-5) node[anchor=north] {\scriptsize $0$};
                \draw (-1+2+8.5,-5) node[anchor=north] {\scriptsize $1$};
                \draw (-1+12,-5) node[anchor=north] {\scriptsize $0$};
                \draw (-1+2+12,-5) node[anchor=north] {\scriptsize $1$};
            \end{tikzpicture}
        \captionsetup{font=footnotesize}        \caption{\cref{exp:compatibility_vs_noncompatibility} yields non-compatible piece-wise geodesic curves $(\mu_t^j) \in C ([0,1];P_p(\mathbb{S}^1))$ indexed by $j$ by modifying a constant-speed geodesic $(\mu_t)$ in the $p$-Wasserstein space on $\mathbb{S}^1$, where $p\in [1,\infty)$.
        \textbf{Top:} Initial measures defined in \eqref{eq:compatibility_mu} and \eqref{eq:noncompatibility_mu_j}. Each $\mu_t^j$ consists of $2^{j+1}$ particles with equal mass arranged equidistantly on the circle with a perimeter equal to 2. As time runs, the particles rotate around the circle with constant speed 1. \textbf{Bottom:} Wasserstein distance from the initial measure as a function of time. The metric speed of the Wasserstein curves is 1 for a.e. $t \in [0,1]$ for all $j$, whereas the fractional Sobolev norm tends to $0$ as $j \to \infty$. }
        \label{fig:noncompatibility}
    \end{figure}
       \\
       \textbf{Observation 0 (none of $(\mu_t^j)$ are compatible).}
       Unlike the original measures $(\mu_t)$, the measures $(\mu_t^j)$ are clearly not compatible anymore. Take $(\mu_t^0)$ for instance. As shown in \cref{fig:compatibility3_counterexample}, it is clear that the collections of measures at time points $\{0,\frac14,\frac24\}$ and $\{\frac14, \frac24, \frac34\}$ are compatible. But, for the collection at time $\{0,\frac14,\frac24,\frac34\}$, no multi-coupling exists in which all two-dimensional marginals are optimal, confirming the fact that checking the optimality only for three measures is not enough. (This conclusion remains unchanged even if the measures are absolutely continuous with respect to the Lebesgue measure, say on $\mathbb{R}^\mathrm{d}$. Thinking of this example in $\mathbb{R}^2$ and replacing the points of delta measures with sufficiently small balls does not change the conclusion.)
        \begin{figure}
            \centering
            \begin{tikzpicture}[scale=0.8]
            \draw [-] (0,0) circle [radius=1.5];
             \draw [dotted] (0,1.5) -- (0,-1.5);
             \draw [dotted] ({1.5*sin(45)},{1.5*cos(45)}) --  ({-1.5*sin(45)},{-1.5*cos(45)});
             \draw [dotted] ({1.5*sin(90)},{1.5*cos(90)}) -- ({-1.5*sin(90)},{-1.5*cos(90)});
             \draw [dotted] ({1.5*sin(135)},{1.5*cos(135)}) -- ({-1.5*sin(135)},{-1.5*cos(135)});
            \draw [fill][blue] ({1.5*sin(90)},{1.5*cos(90)}) circle [radius=2pt];
            \draw [fill][blue] ({-1.5*sin(90)},{-1.5*cos(90)}) circle [radius=2pt];
            \draw ({1.5*sin(90)},{1.5*cos(90)}) node[anchor= west] {{\color{blue}$\mu^0_0$}};
            \draw [fill][cyan] (0,1.5) circle [radius=2pt];
            \draw [fill][cyan] (0,-1.5) circle [radius=2pt];
            \draw (0,1.5) node[anchor=south] {{\color{cyan}$\mu^0_{\frac24}$}};
            \draw [fill][green] ({1.5*sin(45)},{1.5*cos(45)}) circle [radius=2pt];
            \draw [fill][green] ({-1.5*sin(45)},{-1.5*cos(45)}) circle [radius=2pt];
             \draw ({1.5*sin(45)},{1.5*cos(45)}) node[anchor= west] {{\color{green}$\mu^0_{\frac14}$}};
            \draw [fill][red] ({1.5*sin(135)},{1.5*cos(135)}) circle [radius=2pt];
            \draw [fill][red] ({-1.5*sin(135)},{-1.5*cos(135)}) circle [radius=2pt];
            \draw ({1.5*cos(135)},{1.5*sin(135)}) node[anchor= east] {{\color{red}$\mu^0_{\frac34}\,$}};
            \draw[->,gray] (2.35,0) arc [radius=2.35, start angle=0, end angle= 135];
            \draw[->,gray] (-2.35,0) arc [radius=2.35, start angle=180, end angle= 315];
            \end{tikzpicture}
            \captionsetup{font=footnotesize}
            \caption{\cref{exp:compatibility_vs_noncompatibility}. The measure $(\mu^{j=0}_t)$ defined in \eqref{eq:noncompatibility_mu_j} at 4 time points. 
            The collections at times $\{0,\frac14,\frac24\}$ and $\{\frac14, \frac24, \frac34\}$ are compatible, but not at $\{0,\frac14,\frac24,\frac34\}$. To confirm the compatibility, checking the optimality only for three measures is not enough.}
            \label{fig:compatibility3_counterexample}
        \end{figure}
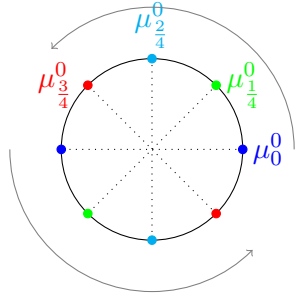
       \\
       \textbf{Observation 1 (modulus of continuity of $(\mu_t^j)$ decreases in $j$).} Let us study the modulus of continuity with the time window $\delta = 1$. Recall our notation in this case: $|X|_0 \coloneqq \sup_{s,t \in [0,1]} d(X_s,X_t)$ for any curve $(X_t)_{t \in [0,1]}$ taking values in a metric space. We have
        \begin{equation}
            \frac{1}{2^{(j+1)p}} = |\mu^j|^p_0 < \int |\gamma|^p_0 \d \pi^j = \int |\gamma|^p_0 \d \pi = |\mu|^p_0 = 1.
        \end{equation}
        While the inequality above always holds thanks to \cref{thm:lift_to_mu_C}, here we even have a strict inequality. 
       \\
       \textbf{Observation 2 (H\"{o}lder regularity of $(\mu_t^j)$ decreases in $j$).} Take $\theta \in (0,1]$ and recall \cref{thm:lift_to_mu_Hol}. Here, we get
       \begin{equation}
       \frac{1}{2^{(j+1)p(1-\theta)}}= |\mu^j|_{\theta\textrm{-}\mathrm{H\ddot{o}l}}^p < \int |\gamma|_{\theta\textrm{-}\mathrm{H\ddot{o}l}}^p \d \pi^j = \int |\gamma|_{\theta\textrm{-}\mathrm{H\ddot{o}l}}^p \d \pi = |\mu|_{\theta\textrm{-}\mathrm{H\ddot{o}l}}^p = 1.
       \end{equation}
       We observe that the H\"{o}lder regularity of $(\mu_t^j)$ also goes to zero as $j \to 0$. 
       \\
       \textbf{Observation 3 ($p$-variation of $(\mu_t^j)$ decreases in $j$ for all $p>1$ while $1$-variation remains constant).} The $1$-variation of the curve $(\mu_t^j)$ with respect to $1$-Wasserstein distance is equal to 1 for all $j$. Also recalling \cite[Theorem 3.3]{AbediLiSchultz2024}, we write
        \begin{equation}
            1 =  |\mu^j|_{1\textrm{-}\mathrm{var}} = \int|\gamma|_{1\textrm{-}\mathrm{var}} \d \pi^j = \int|\gamma|_{1\textrm{-}\mathrm{var}} \d \pi = |\mu|_{1\textrm{-}\mathrm{var}} = 1.
        \end{equation}
       Now take $p>1$ and note that the $p$-variation of the curve $(\mu_t^j)$ with respect to $p$-Wasserstein metric is achieved by the partition $(\frac{i}{2^{j+1}})_{i \in \{0,\cdots 2^{j+1}\}}$. This can be verified using the interpolation result \cite[Proposition 5.5]{FrizVictoir2010book}. We thus obtain
       \begin{equation}
            \frac{1}{2^{(j+1)(p-1)}} = |\mu^j|_{p\textrm{-}\mathrm{var}}^p < \int|\gamma|_{p\textrm{-}\mathrm{var}}^p \d \pi^j = \int|\gamma|_{p\textrm{-}\mathrm{var}}^p \d \pi = |\mu|^p_{p\textrm{-}\mathrm{var}} = 1.
        \end{equation}
        As for the inequality, we recall \cref{thm:lift_to_mu_var}.
        \\
       \textbf{Observation 4 ($W^{1,p}$-regularity of $(\mu_t^j)$ remains constant).}
       Consider $1 < p< \infty$. Each $(\mu_t^j)$ is an absolutely continuous curve in $p$-Wasserstein space with its metric speed $|\dot{\mu}_t^j | = 1$ for almost all time for all $j$. Recalling \cite[Theorem 5]{Lisini2007}, we have here:
        \begin{equation}
           1=|\mu^j|^p_{W^{1,p}} = \int |\gamma|^p_{W^{1,p}} \d \pi^j = \int |\gamma|^p_{W^{1,p}} \d \pi = |\mu|^p_{W^{1,p}}  =  1.
       \end{equation}
       \textbf{Observation 5 ($W^{\alpha,p}$-regularity of $(\mu_t^j)$ decreases in $j$).}
       Consider $1 < p< \infty$ and  $1/p<\alpha< 1$. Then
       \begin{equation}
           \frac{c}{2^{(j+1)p(1-\alpha )}}   = |\mu^j|^p_{b^{\alpha,p}} < \int |\gamma|^p_{b^{\alpha,p}} \d \pi^j  = \int |\gamma|^p_{b^{\alpha,p}} \d \pi = |\mu|^p_{b^{\alpha,p}} = c,
       \end{equation}
       where $c = c(\alpha,p)  \coloneqq 2^{(p - \alpha p )}/(2^{(p - \alpha p ) }- 1)$. Regarding the inequality here,  we recall \cref{thm:lift_to_mu_Walphap}, which is also applicable for the $b^{\alpha,p}$-norm. Since all $(\mu_t^j)$s are piece-wise geodesics, we have applied \cref{lemma:balphap_Xn} to calculate $|\mu^j|_{b^{\alpha,p}}$.
       We observe that while the energy of $\pi^j$ remains constant, $|\mu^j|_{b^{\alpha,p}} \to 0$ as $j \to \infty$. 
       This observation will be used later in the main counterexample of the paper. There, we will start with a lift having infinite energy, and using the trick above, we will decrease the norm of the Wasserstein curve to a finite quantity.
\end{example}

\begin{example}[\textbf{A non-compatible Wasserstein curve of finite $W^{\alpha,p}$-norm but without any lift of finite $W^{\alpha,p}$-energy}]\label{exp:non_compatibile_infinite_energy}
    This is the main counterexample of the paper, whose goal is to show that if we are given $(\mu_t) \in W^{\alpha,p} (I;P_p(\mathcal{X}))$ with $1/p<\alpha< 1$ but without the compatibility property, there can be no lift with finite $W^{\alpha,p}$-energy. 
    Specifically, we construct a non-compatible curve $(\mu_t)_{t \in [0,1]}$ in a $p$-Wasserstein space with finite $W^{\alpha,p}$-norm whose only existing lift $\pi$ has infinite $W^{\alpha,p}$-energy:
    \begin{equation}\label{eq:non_compatibile_infinite_energy_claim}
        |\mu|_{W^{\alpha,p}}^p <  \int_{\Gamma} |\gamma|_{W^{\alpha,p}}^p \d \pi (\gamma) = + \infty.
    \end{equation}
    The construction follows by combining \cref{exp:compatibility_vs_noncompatibility,exp:infinitesimal_variation}, leading to an example on a cylinder. 
    Due to the similarities, we shall not repeat all the details.
    Here, by choosing $\upgamma = \alpha$ for the mass $w_j$ in \eqref{eq:counterexample_mass}, we let the $b^{\alpha,p}$-energy \eqref{eq:counterexample_balphap} of the lift $\pi$ diverge. Meanwhile, by modifying the measures with the particle splitting procedure explained in \cref{exp:compatibility_vs_noncompatibility}, we let the $b^{\alpha,p}$-norm of $(\mu_t)$ decrease and eventually converge to a finite quantity. This procedure clearly won't change the energy. In particular, the computations in the aforementioned examples can be used.
    \\
    \textbf{Construction of $(\mu_t)$ and $\pi$.} Consider a two-dimensional cylinder $\mathcal{X} \coloneqq  \mathbb{S}^{1} \times \mathbb{R} $ equipped with its intrinsic metric, where the perimeter of the circle $\mathbb{S}^{1}$ shall be considered to be equal to 2. 
    We again introduce a countable family of particles. Here they are located on the cylinder in different circles indexed by $j \in \{0,1,2,\cdots\}$ and separated by a distance $a > 2$, as shown in \cref{fig:counterexample_2}. As time runs in the interval $I \coloneqq [0,1]$, the particles rotate around the cylinder at a constant speed while staying on the same circle.
    The speed of particles on circle $j$ is chosen to be $2^{j+1}$. Therefore, the distance of particles from their initial positions as a function of time will be like $(y^j_t)$ in \eqref{eq:counterexample_yj}, whose graphs are shown in \cref{fig:counterexample}.
    As motivated above, the total mass on each circle is chosen here to be
    \begin{equation}
        w_j \coloneqq \bar{w} \, 2^{-jp\alpha}, \qquad \bar{w} (p, \alpha)  \coloneqq 1 - 2^{-p \alpha},
    \end{equation}
    where $1<p<\infty$ and $\frac{1}{p}<\alpha < 1$. 
    In contrast to the previous example, the mass here will be further split equally among $2^{j+1}$ particles on each circle to make it non-compatible in light of \cref{exp:compatibility_vs_noncompatibility}. Finally, we define $(\mu_t)$ and $\pi$ as in \eqref{eq:counterexample_mu_t}-\eqref{eq:counterexample_pi} by a weighted sum of Dirac deltas on the particles. Note that $\pi$ is the only lift of $(\mu_t)$ on the space of continuous paths. Let us emphasize that the construction depends on three parameters $ p, \alpha, a$, which we fix in the range 
    \begin{align}
       1<p<\infty, \qquad \frac{1}{p}<\alpha < 1, \qquad 2<a.
    \end{align}
    \textbf{Viewing $(\mu_t)$ as a solution to CE.} Here, $(\mu_t)$ is the trivial solution to the continuity equation on the cylinder with the initial condition $\mu_0$ and a time-independent velocity vector field:
    \begin{equation}\label{eq:velocity_cylinder}
        v_t(\theta,z) \coloneqq \left( 2^{z/a+1}\right) \hat{\theta},
    \end{equation}
    in cylindrical coordinates. This satisfies the 1-integrability condition but not the $p$-th one as in \eqref{eq:counterexample_1_integrability_1}-\eqref{eq:counterexample_1_integrability_p}.
    \begin{figure}
        \centering
        \begin{tikzpicture}
            \draw [-] (0,0) -- (10,0);
            \draw [-] (0,2) -- (10,2);
            \draw [dashed] (10,0) -- (11,0);
            \draw [dashed] (10,2) -- (11,2);
            \draw (0,1) ellipse [x radius=0.5, y radius=1];
            \draw[-] (4,1) ++(0,-1) arc (-90:90:0.5 and 1);
            \draw[dotted] (4,1) ++(0,1) arc (90:270:0.5 and 1);
            \draw[-] (8,1) ++(0,-1) arc (-90:90:0.5 and 1);
            \draw[dotted] (8,1) ++(0,1) arc (90:270:0.5 and 1);
            \draw[dashed] (0,1) -- (0.5,1);
            \draw[dashed] (0,1) -- (0.4,1.59);
            \draw (0,1.12) node[anchor= west] {\scriptsize $\theta$};
            \draw [dashed](0.4,1.59) -- (0.4+4,1.59);
             \draw (0.4+2,1.59) node[anchor= south] {\scriptsize $z$};
             \draw (0.4+4+0.45,1.59) node[anchor= south] {\scriptsize $v_t(\theta,z)$};
            \draw[->, line width=0.6pt] (0.4+4+0.002,1.59) -- (0.4+4-0.1,1.59+0.38);
            \draw[->,gray] (0.65,1.1) arc (0:180:0.65 and 1.1);
            \draw[<-,gray] (0.65,0.9) arc (0:-180:0.65 and 1.1);
            \draw (6,2+0.2) node[anchor=south] {\scriptsize $a > 2$};
            \draw [|-|](4,2+0.2) -- (8,2+0.2);
            \draw[|-|] (8.1,0.8) ++(0,-1) arc (-90:90:0.65 and 1.2);
            \draw (8.7,1) node[anchor=west] {\scriptsize  length = 1};
            \draw [fill][blue] (0.5,1) circle [radius=2.6pt];
            \draw [fill][blue] (-0.5,1) circle [radius=2.6pt];
            \draw [fill][blue] (4,0) circle [radius=1.7pt];
            \draw [fill][blue] (4,2) circle [radius=1.7pt];
            \draw [fill][blue] (4+0.5,1) circle [radius=1.7pt];
            \draw [fill][blue] (4-0.5,1) circle [radius=1.7pt];
            \foreach \angle in {0, 45, 90, 135, 180, 225, 270, 315} \draw [fill][blue] ({8 + 0.5*cos(\angle)}, {1 + 1*sin(\angle)}) circle [radius=0.87pt];
            \draw (-0.5,1) node[anchor=east] {\scriptsize \color{blue}$ \qquad\, \mu_0 \in P(\mathbb{S}^{1} \times \mathbb{R})$ \, };
            \draw (0,-0.25) node[anchor=north ] {\scriptsize $j=0$};
            \draw (4,-0.25) node[anchor=north ] {\scriptsize $j=1$};
            \draw (8,-0.25) node[anchor=north ] {\scriptsize $j=2$};
            \draw (0.2,-1) node[anchor=east] {\scriptsize speed of particles = };
            \draw (0.2,-1.5) node[anchor=east] {\scriptsize $\#$ of particles on each circle = };
            \draw (0.2,-2) node[anchor=east] {\scriptsize total mass on each circle = };
            \draw (0.2,-1) node[anchor=west] {\scriptsize $2^{j+1}$ };
            \draw (0.2,-1.5) node[anchor=west] {\scriptsize $2^{j+1}$ };
            \draw (0.2,-2) node[anchor=west] {\scriptsize $w_j \, \coloneqq \bar{w} \, 2^{-jp\alpha}$ };
        \end{tikzpicture}
        \captionsetup{font=footnotesize}
        \caption{\cref{exp:non_compatibile_infinite_energy} yields a non-compatible $(\mu_t) \in W^{\alpha,p} ([0,1];P_p(\mathbb{S}^{1} \times \mathbb{R}))$ on a cylinder after combining \cref{exp:compatibility_vs_noncompatibility,exp:infinitesimal_variation}. Here $(\mu_t)$ is the trivial solution to the continuity equation on the cylinder with initial condition $\mu_0$ (in blue) and the velocity field $v_t(\theta,z)$ defined in \eqref{eq:velocity_cylinder}. Initially, the particles are located in different circles, indexed by $j$ and separated by $a>2$. As time runs, the particles rotate around the cylinder at a constant speed. Here the mass is further split equally among $2^{j+1}$ particles on each circle to make $(\mu_t)$ non-compatible and decrease its $W^{\alpha,p}$-regularity.}
        \label{fig:counterexample_2}
    \end{figure}
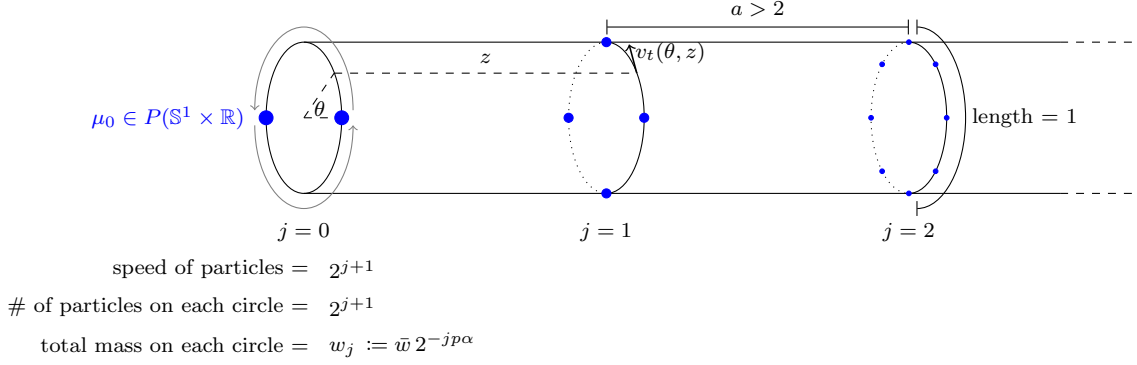
    \\
    \textbf{Claim 1 ($(\mu_t) \subset P_p(\mathcal{X})$).} The $p$-th moments of measures can be estimated similar to \eqref{eq:counterexample_1_moment} to ensure that $(\mu_t)$ stays in the $p$-Wasserstein space. Note that the collection $(\mu_t)$ is \emph{not} compatible as discussed in \cref{exp:compatibility_vs_noncompatibility}. 
    \\
     \textbf{Claim 2 ($\int |\gamma|^p_{W^{1,p}} \d \pi = \infty $).} 
     As in \eqref{eq:counterexample_W1p}, the $W^{1,p}$-energy of $\pi$ remains infinite:  
    \begin{align}
        \int_{\Gamma} |\gamma|^p_{W^{1,p}} \d \pi (\gamma) & = \int_{\Gamma} \int_{0}^{1} |\dot{\gamma}_t|^p \d t \d \pi (\gamma)  = \bar{w} 2^p \sum\nolimits_{j=0}^\infty 2^{j (p - \alpha p )} = + \infty.  \label{eq:counterexample_2_W1p}
    \end{align}
    \\
    \textbf{Claim 3 ($\int |\gamma|^p_{W^{\alpha,p}} \d \pi = \infty $).}
    The chosen quantity for $w_j$ here makes the geometric series in \eqref{eq:counterexample_balphap} diverge:
    \begin{align}
         \int_{\Gamma} |\gamma|^p_{b^{\alpha,p}} \d \pi (\gamma)
          = \frac{\bar{w} }{2^{ - \alpha p} - 2^{-p} }  \sum\nolimits_{j=0}^\infty 2^{j(\alpha p- \alpha p )} = + \infty.
    \end{align}
    \\
    \textbf{Claim 4 ($ |\mu|_{W^{\alpha,p}} < \infty $).} Let us denote by $\mu_t^j$ the probability measure on circle $j$ and write:
    $$
    \mu_t = \sum\nolimits_{j=0}^{\infty} w_j \mu_t^j, \qquad t \in I. 
    $$
    Since the supports of these measures are separated by sufficiently large distance $a>2$, and the mass on each circle remains unchanged in time, we can write:
    \begin{equation}\label{eq:counterexample_2_Wpp_sum}
        W_p^p (\mu_s,\mu_t) = \sum\nolimits_{j=0}^\infty w_j W_p^p(\mu^j_s,\mu^j_t).
    \end{equation}
    Each $(\mu_t^j)$ is a piece-wise geodesic in $P_p(\mathcal{X})$, oscillating between two states separated by the distance $D_j \coloneqq \frac{1}{2^{j+1}}$. The time taken to transition from one state to the other is equal to $\frac{1}{2^{2(j+1)}}$. So we can apply the formula \eqref{eq:balphap_Xn} with $n = 2(j+1)$, in which the only non-zero summand in the first sum is for $m=n$. We obtain 
    \begin{align}
        |\mu^j|^p_{b^{\alpha,p}}
        = D_j^p \, 2^{n \alpha p} + D_j^p \, \frac{2^{n \alpha p}}{2^{p-\alpha p} - 1 }  = \frac{2^{2 \alpha p - p}}{1-2^{\alpha p - p}} \cdot 2^{j(2\alpha p - p )}. 
    \end{align}
    Finally taking \eqref{eq:counterexample_2_Wpp_sum} into account and exchanging the order of summation, we obtain
    \begin{align}
        |\mu|^p_{b^{\alpha,p}} = \sum_{j=0}^{\infty} w_j |\mu^j|^p_{b^{\alpha,p}} = c(\alpha,p) \sum_{j=0}^{\infty} 2^{j (\alpha p - p) } < + \infty,
    \end{align}
    where $c(\alpha,p)$ is a constant. This confirms the claim \eqref{eq:non_compatibile_infinite_energy_claim}.
\end{example}

\noindent\textbf{Data availability.} Data sharing is not applicable to this article as no datasets were generated or analyzed during the current study.



\bibliographystyle{abbrvurl}
\bibliography{ms}

\end{document}